\documentclass[a4paper,11pt,reqno]{amsart}
\usepackage{amsmath, amssymb, amsthm, amsfonts, amsgen, stmaryrd, mathrsfs}
\usepackage[centertags]{amsmath}
\usepackage{graphicx}
\usepackage{caption}
\usepackage{subcaption}
\usepackage{color}
\numberwithin{equation}{section}
\newcommand{\ds}{\displaystyle}
\newcommand{\R}{{\mathbb{R}}}
\newcommand{\Z}{{\mathbb{Z}}}
\newcommand{\N}{{\mathbb{N}}}
\newcommand{\dx}{\,dx}

\newcommand{\ie}{{; \it i.e., }}

\newcommand{\A}{\mathcal{A}}

\newcommand{\T}{\mathcal{T}}

\newcommand{\C}{\mathscr{C}}

\let\e=\varepsilon
\let\O=\Omega

\let\G=\Gamma

\providecommand{\newoperator}[3]{%
  \newcommand*{#1}{\mathop{#2}#3}}
\newoperator{\Per}{\mathrm{Per}}{\nolimits}
\newoperator{\id}{\mathrm{id}}{\nolimits}
\newoperator{\supp}{\mathrm{supp}}{\nolimits}
\newoperator{\tr}{\mathrm{tr}}{\nolimits}
\newoperator{\diag}{\mathrm{diag}}{\nolimits}
\newoperator{\diam}{\mathrm{diam}}{\nolimits}
\newoperator{\dist}{\mathrm{dist}}{\nolimits}
\newoperator{\Div}{\mathrm{div}}{\nolimits}
\newoperator{\cof}{\mathrm{cof}}{\nolimits}
\newoperator{\sspan}{\mathrm{span}}{\nolimits}
\newoperator{\gconv}{\overset{\Gamma}{\rightarrow}}{\nolimits}
\newoperator{\sconv}{\overset{*}{\rightharpoonup}}{\nolimits}
\newoperator{\intS}{\overset{\circ}{S}}{\nolimits}
\newoperator{\dd}{\mathrm{d}}{\nolimits}
\newoperator{\coff}{\mathrm{cof}}{\nolimits}
\newoperator{\conv}{\mathrm{conv}}{\nolimits}
\newoperator{\vvec}{\mathrm{vec}}{\nolimits}
\newoperator{\argmin}{\mathrm{argmin}}{\nolimits}
\newcommand{\LL}{\mathcal{L}}
\newcommand{\HH}{\mathcal{H}}
\newcommand{\W}{\mathcal{W}}

\newtheorem{definition}{Definition}[section]

\newtheorem{theorem}[definition]{Theorem}
\newtheorem{prop}[definition]{Proposition}
\newtheorem{corollary}[definition]{Corollary}
\newtheorem{rem}[definition]{Remark}

\begin{document}
\title[Finite-difference approximations of free-discontinuity problems]
{Quantitative analysis of finite-difference approximations of free-discontinuity problems}
\author[A. Bach, A. Braides, C.I. Zeppieri]{Annika Bach, Andrea Braides, and Caterina Ida Zeppieri}

\address[Annika Bach\footnote{Now Zentrum Mathematik-M7, Technische Universit\"at M\"unchen, Boltzmannstr. 3, 85747 Garching, Germany. E-mail address: annika.bach@ma.tum.de} and Caterina Ida Zeppieri]{Angewandte Mathematik, Universit\"at M\"unster, Einsteinstr. 62, 48149 M\"unster, Germany}
\email{a\_bach10@uni-muenster.de}
\email{caterina.zeppieri@uni-muenster.de}
\address[Andrea Braides]{Dipartimento di Matematica, Universit\`a di Roma ``Tor Vergata'',
via della Ricerca Scientifica 1, 00133 Rome, Italy} 
\email{braides@mat.uniroma2.it}

\begin{abstract}
Motivated by applications to image reconstruction, in this paper we analyse a \emph{finite-difference discretisation} of the Ambrosio-Tortorelli functional. Denoted by $\e$ the elliptic-approximation parameter and by $\delta$ the discretisation step-size, we fully describe the relative impact of $\e$ and $\delta$ in terms of $\Gamma$-limits for the corresponding discrete functionals, in the three possible scaling regimes. We show, in particular, that when $\e$ and $\delta$ are of the same order, the underlying lattice structure affects the $\Gamma$-limit which turns out to be an anisotropic free-discontinuity functional. 
\end{abstract}

\maketitle

\begin{center}
\begin{minipage}{12cm}
\small{ 
 \noindent {\bf Keywords}:  finite-difference discretisation, Ambrosio-Tortorelli functional, $\G$-convergence, elliptic approximation, free-discontinuity functionals.  

\vspace{6pt} \noindent {\it 2000 Mathematics Subject
Classification:} 49M25, 49J45, 68U10, 65M06.}
\end{minipage}
\end{center}


\section{Introduction}
The detection of objects and object contours in images is a central issue in Image Analysis and Computer Vision. 
From a mathematical modelling standpoint, a grey-scale image can be described in terms of a scalar function $g:\Omega\to [0,1]$ (here, $\O \subset \R^n$ is a set parameterising the image domain, \textit{e.g.}, a rectangle in the plane), which measures, at every point in $\O$, the brightness (or grey-level) of the picture. After a model introduced by Mumford and Shah \cite{MS}, the 
relevant information from an input image $g$ can be obtained from a ``restored'' image described by a function $u$ which solves the minimisation problem
\begin{equation}\label{def:MS}
\min\bigg\{M\!S(u)+\int_\Omega|u-g|^2dx \colon u\in SBV(\O)\bigg\},
\end{equation}
where
\begin{equation}\label{intro:MS}
M\!S(u)=\int_\Omega|\nabla u|^2dx+\HH^{n-1}(S_u)
\end{equation}
is the so-called Mumford-Shah functional and $SBV(\O)$ denotes the space of special functions of bounded variation in $\O$ \cite{DGA}, $S_u$ denotes the discontinuity set of $u$, and $\HH^{n-1}$ is the $(n-1)$-dimensional Hausdorff measure. 
By solving \eqref{def:MS} the discontinuous function $g$ is replaced by a function $u$ which is ``close'' to $g$ and at the same time is smooth outside its discontinuity set $S_u$. The latter, moreover, having a ``minimal'' $(n-1)$-dimensional Hausdorff measure will only detect the relevant contours in the input image $g$. We note that a more complete Mumford-Shah functional would be of the form $\alpha\int_\Omega|\nabla u|^2dx+\beta\HH^{n-1}(S_u)$ with $\alpha,\beta$ positive ``contrast parameters''. In the analysis carried out in the present paper it is not restrictive to set $\alpha=\beta=1$.
Although the relevant space dimension for Image Analysis is $n=2$, we define our problems in a $n$-dimensional setting for the sake of generality, and also because in the case $n=3$ the Mumford-Shah functional has an important mechanical interpretation as it coincides with Griffith's fracture energy in the anti-plane case (see \cite{BFM}).
Problem \eqref{def:MS} is a weak formulation proposed by De Giorgi and Ambrosio of the original minimisation problem proposed by Mumford and Shah, where the minimisation is performed on pairs $(u, K)$, with $K$ piecewise-regular closed set and $u$ smooth function outside $K$. In the weak formulation \eqref{def:MS}-\eqref{intro:MS} the set $K$ is replaced by the discontinuity set of $u$, and a solution of the original problem is obtained by setting $K=\overline S_u$ and proving regularity properties of $K$ (see the recent review paper \cite{Focardi}).

The existence of solutions to \eqref{def:MS} following the direct methods of the Calculus of Variations is by now classical
\cite{ambrosio3}. However, the numerical treatment of \eqref{def:MS} presents major difficulties which are mainly due to the presence of the surface term $\HH^{n-1}(S_u)$. A way to circumvent these difficulties is to replace the Mumford-Shah functional in \eqref{def:MS} with an elliptic approximation studied by Ambrosio-Tortorelli \cite{AT1,AT2}, which provides one of the reference 
approximation argument used in the literature (see \textit{e.g.}~\cite{AK,ChSh,CEN,MoS,OsP,VeC}).
Following the Ambrosio-Tortorelli approximation argument, in place of \eqref{def:MS} one considers a family of scale-dependent problems
\begin{equation}\label{i:MS-e}
\min\bigg\{AT_\e(u,v)+\int_\Omega|u-g|^2dx \colon u,v\in W^{1,2}(\O)\bigg\},
\end{equation}
where  
\begin{equation}\label{i:AT}
AT_\e(u,v)=\int_\Omega (v^2+\eta_\e)|\nabla u|^2dx+\frac{1}{2}\int_\Omega\Bigl(\frac{(v-1)^2}{\e}+\e|\nabla v|^2\Bigr)dx.
\end{equation}
Formally, when the approximation parameter $\e>0$ is small, the first term in the second integral of \eqref{i:AT} forces $v$ to be close to the value $1$ except on a ``small'' set, which can be regarded as an approximation of $S_u$. Additionally, the presence of the term $v^2$ in the first integral allows $u$ to have a large gradient where $v$ is close to zero. Finally, the optimisation of the singular-perturbation term with $|\nabla v|^2$ produces a transition layer around $S_u$ giving exactly the surface term present in $M\!S$.
The parameter $\eta_\e>0$ is used in the numerical simulations in order to have well-posed minimisation problems in \eqref{i:MS-e}; it is taken much smaller than $\e$, but does not intervene in the mathematical analysis.  
It is interesting to note that the coefficient $v^2+\eta_\e$ can be also interpreted as a damage parameter (see e.g. \cite{FrMa}), so that, within Fracture Theory, $AT_\e$ can be seen as an approximation of Griffith's Fracture by concentrated damage. More in general the functionals $AT_\e$ are a prototype of phase-field models for free-discontinuity problems.

Since the functionals in \eqref{i:MS-e} are equi-coercive and $AT_\e$ converge to $M\!S$ in the sense of $\Gamma$-convergence \cite{AT2}, solving \eqref{i:MS-e} gives pairs $(u_\e,v_\e)$, where $u_\e$ approximates a solution $u$ to \eqref{def:MS} and $v_\e$ provides a diffuse approximation of the corresponding discontinuity set $S_u$. 
Moreover, since the functionals $AT_\e$ are elliptic, the difficulties arising in the discretisation of the free-discontinuity set are prevented and \emph{finite-elements} or \emph{finite-difference} schemes for $AT_\e$ can be implemented. 
From the $\Gamma$-convergence of the numerical approximations of $AT_\e$ with mesh size $\delta$ (at fixed $\e$) and the $\Gamma$-convergence of $AT_\e$ to $M\!S$, a diagonal argument shows that if the mesh size $\delta=\delta(\e)$ is fine enough then the numerical approximations of $AT_\e$ with mesh size $\delta(\e)$ $\Gamma$-converge to $M\!S$. For \emph{finite-elements} schemes Bellettini and Coscia (\cite{belcos94}; see also Bourdin \cite{Bou} for the numerical implementation) showed more precisely that this holds if the mesh-size $\delta$ is chosen such that $\delta(\e)\ll \e$. In other words, this assumption on $\delta$ allows a ``separation-of-scale'' argument and to regard separately the two limits as $\e$ and $\delta$ tend to $0$, respectively. Conversely, note that if $\delta>\!\!>\e$ then the second integral in \eqref{i:AT} diverges unless $v$ is uniformly close to $1$, which implies that the domain of the $\Gamma$-limit of numerical approximations of $AT_\e$ with mesh size $\delta$ is with $u\in W^{1,2}(\Omega)$, and hence the $\Gamma$-limit is not $M\!S$ (see also the arguments of Section \ref{sec:superc}).

In order to illustrate in general the combined effect of $\delta$ and $\e$, in particular when $\delta$ is of the same order of $\e$, we briefly recall some analyses which started from
a different discrete approximation scheme for \eqref{def:MS}. Chambolle in \cite{chambolle95}, 
considered a {\em finite-difference} approximation of $M\!S$  based on an earlier model by Blake and Zissermann \cite{BZ}: in the case of space-dimension $n=1,2$, Chambolle studied the asymptotic behaviour of the discrete functionals given by
\begin{equation}\label{i:cham-func}
F_\e(u)=\frac{1}{2}\sum_{\begin{smallmatrix}i,j\in\Omega\cap\e\Z^n\\|i-j|=\e\end{smallmatrix}}\e^n\min\bigg\{\left|\frac{u(i)-u(j)}{\e}\right|^2,\frac{1}{\e}\bigg\},
\end{equation}
where the energies depend on finite differences through a {\em truncated quadratic potential} with threshold energy ${1/\e}$.
If $n=1$ he showed that the functionals $F_\e$ $\Gamma$-converge to $M\!S$ with respect to an appropriate discrete-to-continuum convergence of lattice functions. In dimension $n=2$, however, the $\Gamma$-limit of $F_\e$ turns out to be anisotropic and given by
\begin{equation}\label{i:cham-F}
F(u)=\int_\Omega|\nabla u|^2dx+\int_{S_u}|\nu_u|_1d\HH^{1},
\end{equation}
where $\nu_u$ denotes the normal to $S_u$ and $|\nu|_1=|\nu_1|+|\nu_2|$ is the $1$-norm of the vector $\nu$, which appears in the limit due to the specific geometry of the underlying lattice $\e\Z^2$. 
Using the lattice energy \eqref{i:cham-F} as a model, some continuum approximations of the original isotropic Mumford-Shah energy have been obtained. Notably, a continuum finite-difference approximation was conjectured by De Giorgi and proved by Gobbino \cite{Go}, while a non-local version involving averages of gradients in place of finite differences was proved by Braides and Dal Maso \cite{BDM}. 

Various modifications of $F_\e$ have been studied, many of which in the direction of obtaining more general surface terms in the limit energies. In \cite{chambolle99} Chambolle introduced a variant of \eqref{i:cham-func} where arbitrary finite differences and truncated energy densities with variable threshold energies are considered. He showed that this new class of functionals provides discrete approximations of image-segmentation functionals where the anisotropy is ``reduced'' with respect to \eqref{i:cham-F} (see also the paper by Braides and Gelli \cite{BG02}).  
Braides and Piatnitski \cite {BPiat} examined random mixtures of truncated quadratic and simply quadratic interactions producing surface energies whose anisotropy can be described through percolation results,
whereas in the recent paper \cite{Ruf17} Ruf shows that the anisotropy in the limit functional can be prevented by considering discrete approximating functionals defined on \emph{statistically isotropic} lattices. The form of the surface energy can be studied separately by examining energies on lattice spin functions (see \textit{e.g.}~\cite{CdL,ABC} and \cite{BSeul} and the references therein); in particular patterns of interactions (corresponding to different threshold values in the truncated quadratic potentials) satisfying design constraints and giving arbitrary surface energies has been recently described by Braides and Kreutz \cite{BKr}.
As finite-difference schemes involving energies as in \eqref{i:cham-func} are concerned, in \cite{chamdal99} Chambolle and Dal Maso show that macroscopic anisotropy can be avoided by considering \emph{alternate finite-elements} of suitable local approximations of the Mumford-Shah functional. 

The finite-difference schemes described above suggest that in the numerical implementation of the Ambrosio-Tortorelli approximation, for general values 
of the mesh-size $\delta$ and the parameter $\e$  the anisotropy of the surface term cannot be ruled out as in
the case considered in \cite{belcos94}. 
In terms of $\Gamma$-convergence, we may expect that for a general dependence of $\delta$ on $\e$ a discretisation of $AT_\e$ with mesh size $\delta$
shall not converge to the Mumford-Shah functional but rather to some anisotropic functional of the form
\begin{equation}\label{genfo}
E(u)=\int_\Omega|\nabla u|^2dx+\int_{S_u}\varphi(\nu_u)d\HH^{n-1},
\end{equation}
where the surface integrand $\varphi$ reflects the geometry of the underlying lattice and may depend on the interaction between $\delta$ and $\e$. 
These considerations motivate the analysis carried out in the present paper. 

In the spirit of a recent paper by Braides and Yip \cite{BY12} in which the discretisation of the Modica-Mortola functional \cite{MoMo} is analysed, here we propose and analyse a \emph{finite-difference} discretisation of the Ambrosio-Tortorelli functionals\ie we consider the functionals defined as
\begin{equation}\label{def:energy}
E_\e(u,v)=\frac{1}{2}\Bigg(\sum_{\begin{smallmatrix}i,j\in\Omega\cap\delta\Z^n\\|i-j|=\delta\end{smallmatrix}}\hspace*{-1em}\delta^n(v^i)^2\left|\frac{u^i-u^j}{\delta}\right|^2+\sum_{i\in\Omega\cap\delta\Z^n}\hspace*{-1em}\delta^n\frac{(v^i-1)^2}{\varepsilon}+\frac{1}{2}\sum_{\begin{smallmatrix}i,j\in\Omega\cap\delta\Z^n\\|i-j|=\delta\end{smallmatrix}}\hspace*{-1em}\varepsilon\delta^n\left|\frac{v^i-v^j}{\delta}\right|^2\Bigg)
\end{equation}
and we study their limit behaviour as $\e$ and $\delta$ simultaneously tend to zero. Since the discrete functionals in \eqref{def:energy} are more explicit than the Bellettini-Coscia finite-elements discretisation, we are in a position to perform a rather detailed $\Gamma$-convergence analysis for $E_\e$ in all the three possible scaling regimes\ie $\delta\ll\e$ (subcritical regime), $\delta\sim\e$ (critical regime), and $\delta\gg \e$ (supercritical regime).    
 More precisely, if $\ell:=\lim_{\e}\frac{\delta}{\e}$, in Theorem \ref{t:main-result} we prove that for every $\ell \in [0,+\infty]$ and for $n=2$ the functionals $E_\e$ $\Gamma$-converge to 
\[E_\ell(u)=\int_\Omega|\nabla u|^2dx+\int_{S_u}\varphi_\ell(\nu_u)d\HH^{1},\]
for some surface integrand $\varphi_\ell \colon S^{1} \to [0,+\infty]$.  Furthermore, we show that in the subcritical regime $\varphi_0\equiv 1$ so that $E_0=M\!S$, in the critical regime $\varphi_\ell$ explicitly depends on the normal $\nu$ (see \eqref{i:phi-ell}, below), and finally in the supercritical regime $\varphi_\infty\equiv +\infty$, so that $E_\infty$ is finite only on the Sobolev Space $W^{1,2}(\O)$ and it coincides with the Dirichlet functional.  

It is worth mentioning that the convergence results in the extreme cases $\ell=0$ and $\ell=+\infty$ actually hold true in any space dimensions (see Section \ref{s:subcritical}  for the case $\ell=0$ and Section \ref{sec:superc} for the case $\ell=+\infty$), whereas the convergence result in the critical case is explicit only for $n=1,2$. In fact, for $\ell\in (0,+\infty)$ the surface integrand $\varphi_\ell$ can be  explicitly determined only for $n=1,2$ (see Theorem \ref{characterization:phi} and Remark \ref{rem:coordinate}), while for $n>2$ we can only prove an abstract compactness and integral representation result (see Theorem \ref{gamma:compactness} and Theorem \ref{int:rep}) which, in particular, does not allow us to exclude that the surface energy density may also depend on the jump opening. The main difference between the case $n=2$ and $n=3$ (and higher) is related to the problem of describing the structure of the sets of lattice sites where the parameter $v$ is close to $0$,
which approximates the set jump $S_u$. In principle, if that discrete set presents ``holes'' the limit surface energy may depend on the values $u^\pm$ of $u$ on both sides of $S_u$ (see \cite{BSigalotti}). In two dimensions this is ruled out by showing that such lattice sets can be locally approximated by a continuous line. In dimension $n>2$ deducing that such set is approximately described by a hypersurface seems more complex and in this case the difficulties are similar to those encountered in some lattice spin problems (\textit{e.g.}, when dealing with dilute lattice spin systems \cite{BPiat2}).

\bigskip

Below we briefly outline the analysis carried out in the present paper, in the three different scaling regimes.

\smallskip 

\emph{Subcritical regime: $\ell=0$.} In this regime the $\Gamma$-limit of the finite-difference discretisation $E_\e$ is the Mumford-Shah functional $M\!S$, as in the case of the finite-elements discretisation analysed by Bellettini and Coscia in \cite{belcos94}. Even if the scaling regime is the same as in \cite{belcos94}, the proof of the $\Gamma$-convergence result for $E_\e$ is substantially different. In particular, the most delicate part in the proof of the $\Gamma$-convergence result is to show that the lower-bound estimate holds true. Indeed, in our case the form (and the non-convexity) of the first term in $E_\e$ makes it impossible to have an inequality of the type   
\[E_\e(u,v)\geq AT_\e(\tilde{u},\tilde{v})+o(1),\]
where $\tilde{u}$ and $\tilde{v}$ denote suitable continuous interpolations  of $u$ and $v$, respectively.
Then, to overcome this difficulty we first prove a non-optimal \emph{asymptotic lower bound} for $E_\e$ which allows us to show that the domain of the $\Gamma$-limit is $(G)SBV(\O)$ (see Proposition \ref{compactness}). Subsequently, we combine this information with a careful blow-up analysis, which eventually provides us with the desired optimal lower bound (see Proposition \ref{prop:liminf}). Finally, the upper-bound inequality follows by an explicit construction (see Proposition \ref{prop:limsup}).   

\smallskip 

\emph{Critical regime: $\ell\in(0,\infty)$.} When the two scales $\e$ and $\delta$ are comparable, we appeal to the so-called ``direct methods'' of $\Gamma$-convergence to determine the $\Gamma$-limit of $E_\e$. Namely, we show that $E_\e$ admits a $\Gamma$-convergent subsequence whose limit is an integral functional of the form
\begin{equation}\label{i:Eell}
E_\ell(u)=\int_\Omega|\nabla u|^2dx+\int_{S_u}\phi_\ell([u],\nu_u)d\HH^{n-1},
\end{equation}
for some Borel function $\phi_\ell:\R\times S^{n-1}\to [0,+\infty)$. Here in general the surface integrand $\phi_\ell$ depends on the subsequence, on the jump-opening $[u]=u^+-u^-$, and on the normal $\nu_u$ to the jump-set $S_u$. The delicate part in the convergence result as above is to show that the abstract $\Gamma$-limit satisfies the assumptions needed to represent it in an integral form as in \eqref{i:Eell}. Specifically, a so-called ``fundamental estimate'' for the functionals $E_\e$ is needed (see Proposition \ref{fund:est}).   

For $n=2$, which is the most relevant case for the applications we have in mind, we are able to explicitly characterise the function $\phi_\ell$. In particular we prove that $\phi_\ell$ does not depend on the subsequence and on the jump-amplitude $[u]$.  Specifically, we show that $\phi_\ell\equiv\varphi_\ell$ where
\begin{multline}\label{i:phi-ell}
\varphi_\ell(\nu)=\lim_{T\to+\infty}\frac{1}{2T}\inf\Bigg\{\ell\hspace*{-.5em}\sum_{i\in TQ^\nu\cap\Z^2}\hspace*{-.5em}(v^i-1)^2+\frac{1}{2\ell}\hspace*{-1em}\sum_{\begin{smallmatrix}i,j\in TQ^\nu\cap\Z^2\\|i-j|=1\end{smallmatrix}}\hspace*{-1em}|v^i-v^j|^2\colon \exists\ \text{channel $\mathscr C$ in}\ TQ^\nu\cap\Z^2\colon 
\\
 \text{$v=0$ on $\C$},\ v=1\ \text{otherwise near}\ \partial TQ^\nu\Big\}.
\end{multline}
In \eqref{i:phi-ell} a channel $\mathscr C$ (see Definition \ref{def:channel} for a formal definition) is a path on the square lattice $\mathbb Z^2$ connecting two opposite sides of the square and can be interpreted as a ``discrete approximation'' of the discontinuity line $\{x\in \R^2 \colon \langle x,\nu \rangle=0\}$. 
When $\nu=e_1, e_2$ we show that the channel $\mathscr C$ in \eqref{i:phi-ell} is actually flat and it coincides with the discrete interface $\{x\in \R^2 \colon \langle x,\nu \rangle=0\}\cap \Z^2$. As a consequence, the minimisation problem defining $\varphi_\ell$ turns out to be one-dimensional (see Remark \ref{rem:coordinate}).  

For $n=1$ the function $\phi_\ell$ is also explicit and equal to a constant; the proof of this fact is a consequence of more elementary one-dimensional arguments and is briefly discussed in Remark \ref{rem:coordinate}. 

\smallskip

\emph{Supercritical regime: $\ell=+\infty$.} In this scaling regime discontinuities have a cost proportional to $\delta/\e \gg 1$ and are therefore forbidden. In fact the $\Gamma$-limit $E_\infty$ turns out to be finite only in $W^{1,2}(\Omega)$ (see Proposition \ref{compactness}) and 
\[E_\infty(u)=\int_\Omega|\nabla u|^2dx.\]
In order to allow for the development of discontinuities in the limit, in the spirit of Braides and Truskinovsky \cite{BT08}, in this case we also analyse the asymptotic behaviour of a suitably rescaled variant of $E_\e$ whose $\Gamma$-limit is still a functional of the form \eqref{genfo} with $\varphi(\nu)=|\nu|_\infty$ (see Theorem \ref{thm:rescaled}), so that in this supercritical regime we recover a crystalline surface energy.

\medskip

The paper is organised as follows. In Section \ref{sec:setting-statement} we introduce a few notation and state the main $\Gamma$-convergence result Theorem \ref{t:main-result}. In Section \ref{sec:domain} we determine the domain of the $\Gamma$-limit in the three scaling regimes, we prove an equicoercivity result for a suitable perturbation of the fuctionals $E_\e$, and study the convergence of the associated minimisation problems. In Sections \ref{s:subcritical}, \ref{sec:critical}, and \ref{sec:superc} we prove the $\Gamma$-convergence result Theorem \ref{t:main-result}, respectively, in the subcritical, critical, and supercritical regime. In Section \ref{sec:superc} we also analyse the asymptotic behaviour of a sequence of functionals which is equivalent to $E_\e$ in the sense of $\Gamma$-convergence (see \cite{BT08}).  Eventually, in Section \ref{sec:interp} we show that for $n=2$ and $\ell \in (0,+\infty)$, the surface integrand $\varphi_\ell$ interpolates the two extreme regimes $\ell=0$ and $\ell=+\infty$.

\section{Setting of the problem and statement of the main result}\label{sec:setting-statement}



\noindent \textbf{Notation.}  Let $n\geq 1$, we denote by $\Omega \subset \R^n$ an open bounded set of with Lipschitz boundary. We furthermore denote by $\mathscr A(\Omega)$ the family of all open subsets of $\Omega$ and by  $\mathscr A_L(\Omega)\subset\mathscr A(\Omega)$ the family of all open subsets of $\Omega$ with Lipschitz boundary. If $A',A\in\mathscr A(\Omega)$ are such that $A'\subset\subset A$, we say that $\varphi$ is a cut-off function between $A'$ and $A$ if $\varphi\in C_c^\infty(A)$, $0\leq\varphi\leq 1$ and $\varphi\equiv 1$ on $A'$. 

If $t\in\R$ we denote by $\lfloor t\rfloor$ its integer part.
If $\nu=(\nu_1,\ldots,\nu_n)\in\R^n$ we denote by $|\nu|$ the euclidian norm of $\nu$. Moreover, we set $|\nu|_1:=\sum_{k=1}^n|\nu_k|$ and $|\nu|_\infty:=\max_{1\leq k\leq n}|\nu_k|$. We use the notation $\left<\nu,\xi\right>$ for the scalar product between $\nu,\xi\in\R^n$. We set $S^{n-1}:=\{\nu\in\R^n \colon |\nu|=1\}$ and for every $\nu\in S^{n-1}$ we denote by $\Pi_\nu:=\{x\in \R^n \colon \langle x, \nu\rangle=0\}$ the hyperplane through $0$ and orthogonal to $\nu$. We also denote by  $\Pi_\nu^+$ and $\Pi_\nu^-$ the two half spaces defined, respectively, as $\Pi_\nu^+:=\{x\in \R^n \colon \langle x, \nu\rangle>0\}$ and $\Pi_\nu^-:=\{x\in \R^n \colon \langle x, \nu\rangle\leq 0\}$.
For every $\nu\in S^{n-1}$ we denote by $Q^\nu\subset\R^n$ a given cube centred at $0$ with side length $1$ and with one face orthogonal to $\nu$, and for all $x_0\in\R^n$ and $\rho>0$ we set $Q^\nu_\rho(x_0)=x_0+\rho Q^\nu$. If $\{e_1,\ldots,e_n\}$ denotes the standard basis in $\R^n$ and $\nu=e_k$ for some $1\leq k\leq n$ we choose $Q=Q^\nu$ the standard coordinate cube and simply write $Q_\rho(x_0)$.

By $\LL^n$ and $\HH^k$ we denote the Lebesgue measure and the $k$-dimensional Hausdorff measure in $\R^n$, respectively. For $p\in [1,+\infty]$ we use standard notation $L^p(\Omega)$ for the Lebesgue spaces and $W^{1,p}(\Omega)$ for the Sobolev spaces.
We denote by $SBV(\O)$ the space of special functions of bounded variation in $\O$ (for the general theory see \textit{e.g.}, \cite{AFP, braides98}).
If  $u\in SBV(\Omega)$ we denote by $\nabla u$ its approximate gradient, by $S_u$ the approximate discontinuity set of $u$, by $\nu_u$ for the generalised outer normal to $S_u$, and $u^+$ and $u^-$ are the traces of $u$ on both sides of $S_u$. We also set $[u]:=u^+-u^-$. Moreover, we consider the larger space $GSBV(\Omega)$, which consists of all functions $u\in L^1(\Omega)$ such that for each $m\in\N$ the truncation of $u$ at level $m$ defined as $u^m:=-m\vee(u\wedge m)$ belongs to $SBV(\Omega)$. Furthermore, we set
\[SBV^2(\Omega):=\{u\in SBV(\Omega): \nabla u\in L^2(\Omega)\ \text{and}\ \HH^{n-1}(S_u)<+\infty\}\]
and
\[GSBV^2(\Omega):=\{u\in GSBV(\Omega): \nabla u\in L^2(\Omega)\ \text{and}\ \HH^{n-1}(S_u)<+\infty\}.\]
It can be shown that $SBV^2(\Omega)\cap L^\infty(\Omega)=GSBV^2(\Omega)\cap L^\infty(\Omega)$.

Let $u,w$ be two measurable functions on $\R^n$ and let $A\subset \R^n$ be open, bounded and with Lipschitz boundary; by ``$u=w$ near $\partial A$'' we mean that there exists a neighbourhood $U$ of $\partial A$ in $\R^n$ such that $u=w$ $\mathcal{L}^n$-a.e.\ in $U \cap A$.


\medskip

\noindent {\bf Setting.} Throughout the paper $\varepsilon>0$ is a strictly positive parameter and  $\delta=\delta(\varepsilon)>0$ is a strictly increasing function of $\e$ such that such that $\delta(\varepsilon)\to 0$ decreasingly as $\varepsilon\to 0$ decreasingly. 
Set  
\begin{equation}\label{def:ell}
\ell:=\lim_{\e\to 0}\frac{\delta(\e)}{\e}.
\end{equation}
We now introduce the discrete functionals which will be analysed in this paper. To this end let $\O\subset \R^n$ be open, bounded, and with Lipschitz boundary. Let $\Omega_\delta:=\Omega\cap\delta\Z^n$ denote the portion of the square lattice of mesh-size $\delta$ contained in $\O$ and for every $u:\Omega_\delta\to\R$ set $u^i:=u(i)$, for $i\in\O_\delta$. It is customary to identify the discrete functions defined on the lattice $\Omega_\delta$ with their piecewise-constant counterparts belonging to the class
\[\A_\e(\Omega):=\{u\in L^1(\Omega):\ u\ \text{constant on}\ i+[0,\delta)^n\ \hbox{for all }i\in\Omega\cap\delta\Z^n\},\]
by simply setting
\begin{equation}\label{piec:const}
u(x):=u^i\quad \text{for every}\; x\in i+[0,\delta)^n\; \text{and for every}\; i \in \Omega_\delta.
\end{equation}
If $(u_\varepsilon)$ is a sequence of functions defined on the lattice $\Omega_\delta$ and $u\in L^1(\Omega)$, by $u_\varepsilon\to u$ in $L^1(\Omega)$ we mean that the piecewise-constant interpolation of $(u_\e)$ defined as in \eqref{piec:const} converges to $u$ in $L^1(\Omega)$.

We define the discrete functionals $E_\varepsilon:L^1(\Omega)\times L^1(\Omega)\to [0,+\infty]$ as
\begin{equation}\label{def:funct}
E_\varepsilon(u,v):=
\begin{cases}
\ds \frac{1}{2}\Bigg(\sum_{\begin{smallmatrix}i,j\in\Omega_\delta\\|i-j|=\delta\end{smallmatrix}}\delta^n(v^i)^2\Big|\frac{u^i-u^j}{\delta}\Big|^2+ 
\sum_{i\in\Omega_\delta}\delta^n\frac{(v^i-1)^2}{\varepsilon}+\frac{1}{2}\sum_{\begin{smallmatrix}i,j\in\Omega_\delta\\|i-j|=\delta\end{smallmatrix}}\varepsilon\delta^n\Big|\frac{v^i-v^j}{\delta}\Big|^2\Bigg) \\
\hspace{7.5cm}
\text{if}\ u,v\in \A_\e(\Omega),\ 0\leq v\leq 1,
\cr \cr
+\infty \hspace{6.8cm} \text{otherwise in}\ L^1(\Omega)\times L^1(\Omega),
\end{cases}
\end{equation}
It is also convenient to consider the functionals $F_\varepsilon$, $G_\varepsilon$ given by
\begin{equation}\label{def:bulk}
F_\varepsilon(u,v):=\frac{1}{2}\sum_{\begin{smallmatrix}i,j\in\Omega_\delta\\|i-j|=\delta\end{smallmatrix}}\delta^n(v^i)^2\Big|\frac{u^i-u^j}{\delta}\Big|^2
\end{equation}
and
\begin{equation}\label{def:surface}
G_\varepsilon(v):=\frac{1}{2}\bigg(\sum_{i\in\Omega_\delta}\delta^n\frac{(v^i-1)^2}{\varepsilon}+\frac{1}{2}\sum_{\begin{smallmatrix}i,j\in\Omega_\delta\\|i-j|=\delta\end{smallmatrix}}\varepsilon\delta^n\Big|\frac{v^i-v^j}{\delta}\Big|^2\bigg),
\end{equation}
so that in more compact notation we may write
\begin{equation*}
E_\varepsilon(u,v):=
\begin{cases}
F_\varepsilon(u,v)+G_\varepsilon(v) &\text{if}\ u,v\in \A_\e(\Omega),\ 0\leq v\leq 1,\\
+\infty &\text{otherwise in}\ L^1(\Omega)\times L^1(\Omega).
\end{cases}
\end{equation*}
In what follows we will also make use of the following equivalent expressions for $F_\e$ and $G_\e$: 
\[F_\varepsilon(u,v)=\frac{1}{2}\sum_{i\in\Omega_\delta}\delta^n(v^i)^2\bigg(\sum_{\begin{smallmatrix}k=1\\i+\delta e_k\in\Omega_\delta\end{smallmatrix}}^n\bigg|\frac{u^i-u^{i+\delta e_k}}{\delta}\Big|^2+\sum_{\begin{smallmatrix}k=1\\i-\delta e_k\in\Omega_\delta\end{smallmatrix}}^n\bigg|\frac{u^i-u^{i-\delta e_k}}{\delta}\Big|^2\bigg)\]
and
\[G_\varepsilon(v)=\frac{1}{2}\bigg(\sum_{i\in\Omega_\delta}\delta^n\frac{(v^i-1)^2}{\varepsilon}+\sum_{i\in\Omega_\delta}\sum_{\begin{smallmatrix}k=1\\i+\delta e_k\in\Omega_\delta\end{smallmatrix}}^n\varepsilon\delta^n\Big|\frac{v^i-v^{i+\delta e_k}}{\delta}\Big|^2\bigg),\]
%
%
For $U\in\mathscr A(\Omega)$ we will also need to consider the localised versions of $F_\e$ and $G_\e$\ie for every $U\in\mathscr A(\Omega)$ we set
\begin{equation}\label{def:mesh:int}
U_\delta:=U\cap\delta\Z^n,
\end{equation}
and
\[F_\varepsilon(u,v,U):=\frac{1}{2}\sum_{i\in U_\delta}\delta^n(v^i)^2\sum_{\begin{smallmatrix}k=1\\i\pm\delta e_k\in U_\delta\end{smallmatrix}}^n\Big|\frac{u^i-u^{i\pm\delta e_k}}{\delta}\Big|^2,\]
\[G_\varepsilon(v,U):=\frac{1}{2}\Bigg(\sum_{i\in U_\delta}\delta^n\bigg(\frac{(v^i-1)^2}{\varepsilon}+\hspace{-.3cm}\sum_{\begin{smallmatrix}k=1\\i+\delta e_k\in U_\delta\end{smallmatrix}}^n\hspace{-.3cm}\e\Big|\frac{v^i-v^{i+\delta e_k}}{\delta}\Big|^2\bigg)\Bigg),\]
so that finally
\begin{equation}\label{loc:funct}
E_\varepsilon(u,v,U):=
\begin{cases}
F_\varepsilon(u,v,U)+G_\varepsilon(v,U) &\text{if}\ u,v\in\A_\e(\Omega),\ 0\leq v\leq 1,\\
+\infty &\text{otherwise in}\ L^1(\Omega)\times L^1(\Omega).
\end{cases}
\end{equation}
\noindent Sometimes it will be useful to distinguish between points $i\in U_\delta $ such that all their nearest neighbours belong to $U$ and points $i\in U_\delta$ such that $i\pm\delta e_k\not\in U$ for some $1\leq k\leq n$. Then, for a given $U\in\A(\Omega)$ we set
\[\overset{\circ}{U}_\delta:=\{i\in U_\delta:\ j\in U\ \text{for every}\ j\in\delta\Z^n\ \text{s.t.}\ |i-j|=\delta\},\quad\text{and}\quad\partial U_\delta:=U_\delta\setminus\overset{\circ}{U}_\delta.\]

\medskip

\noindent With the identification above, we will describe the $\Gamma$-limits of energies $E_\e$ with respect to the strong $L^1(\O)\times L^1(\O)$-topology, in the spirit of recent discrete-to-continuum analyses (see e.g.~\cite{AC,BPiat3,AliGe,BCic} for some general results in different limit functional settings and \cite{Braides02,B-Handbook} for some introductory material).
 
In all that follows we use the standard notation for the $\G$-liminf and $\G$-limsup of the functionals $E_\e$ (see \cite{Braides02} Section 1.2)\ie for every $(u,v)\in L^1(\Omega)\times L^1(\Omega)$ and every $U\in \mathscr A(\Omega)$ we set
\begin{equation}\label{c:li-ls}
E'_\ell(u,v,U):=\Gamma \hbox{-}\liminf_{\varepsilon\to 0} E_\varepsilon(u,v,U) \quad \text{and}\quad E''_\ell(u,v,U):=\Gamma \hbox{-}\limsup_{\varepsilon\to 0} E_\varepsilon(u,v,U).
\end{equation}
When $U=\Omega$ we simply write $E'_\ell(u,v)$ and $E''_\ell(u,v)$ in place of $E'_\ell(u,v,\Omega)$ and $E''_\ell(u,v,\Omega)$, respectively.

\medskip

The following $\G$-convergence theorem is the main result of this paper.

\begin{theorem}[$\G$-convergence]\label{t:main-result}
Let $\ell$ be as in \eqref{def:ell} and let $E_\e$ be as in \eqref{def:funct}. Then,

\smallskip

$(i)$ {\rm (Subcritical regime)} If $\ell=0$ the functionals $E_\e$ $\Gamma$-converge to 
$E_0$ defined as
$$
E_0(u,v):=\begin{cases}
\ds\int_\O |\nabla u|^2\dx+\HH^{n-1}(S_u\cap \O) & \text{if} \; u\in GSBV^2(\O), \, v =1 \; \text{a.e. in}\; \O,
\cr
+\infty & \text{otherwise in}\; L^1(\O)\times L^1(\O); 
\end{cases}
$$

\smallskip

$(ii)$ {\rm (Critical regime)} If $\ell\in(0,+\infty)$ there exists a subsequence $(\e_j)$ such that the functionals $E_{\e_j}$ $\Gamma$-converge to 
$E_\ell$ defined as
$$
E_\ell(u,v):=\begin{cases}
\ds\int_\O |\nabla u|^2\dx + \int_{S_u\cap \O} \phi_\ell([u],\nu_u)\, d\mathcal H^{n-1} & \text{if} \; u\in GSBV^2(\O), \, v =1 \; \text{a.e. in}\; \O,
\cr
+\infty & \text{otherwise in}\; L^1(\O)\times L^1(\O), 
\end{cases}
$$
for some Borel function $\phi_\ell \colon \R\times S^{n-1} \to [0,+\infty)$ possibly depending on the subsequence $(\e_j)$.  
If moreover $n=2$ the function $\phi_\ell$ does not depend on the subsequence $(\e_j)$. Furthermore, for every $(t,\nu)\in \R\times S^{n-1}$ we have $\phi_\ell(t,\nu)=\varphi_\ell(\nu)$ where $\varphi_\ell \colon S^{n-1} \to [0,+\infty)$ is given by 
\begin{multline*}
\varphi_\ell(\nu):=\lim_{T\to+\infty}\frac{1}{2T}\inf\Bigg\{\ell\hspace*{-.5em}\sum_{i\in TQ^\nu\cap\Z^2}\hspace*{-.5em}(v^i-1)^2+\frac{1}{2\ell}\hspace*{-1em}\sum_{\begin{smallmatrix}i,j\in TQ^\nu\cap\Z^2\\|i-j|=1\end{smallmatrix}}\hspace*{-1em}|v^i-v^j|^2\colon v\in \mathcal A_1(TQ^\nu), \\
\exists\ \text{channel $\mathscr C$ in}\ TQ^\nu\cap\Z^2\colon \text{$v=0$ on $\C$},\ v=1\ \text{otherwise near}\ \partial TQ^\nu\bigg\},
\end{multline*}
(see Definition \ref{def:channel} for a precise definition of channel);   

\smallskip

$(iii)$ {\rm (Supercritical regime)} If $\ell=+\infty$ the functionals $E_\e$ $\Gamma$-converge to 
$E_\infty$ defined as
$$
E_\infty(u,v):=\begin{cases}
\ds\int_\O |\nabla u|^2\dx & \text{if} \; u\in W^{1,2}(\O), \, v =1 \; \text{a.e. in}\; \O,
\cr
+\infty & \text{otherwise in}\; L^1(\O)\times L^1(\O). 
\end{cases}
$$

\end{theorem}

\noindent The proof of Theorem \ref{t:main-result} will be divided into a number of intermediate steps and carried out in Sections \ref{s:subcritical}, \ref{sec:critical}, and \ref{sec:superc}.


\section{Domain of the $\G$-limit and compactness}\label{sec:domain}

In this section we prove a compactness result for the functionals $E_\e$. This result is first obtained for $n=1$ and then extended to the case $n \geq 2$ by means of a slicing-procedure (see \cite{Braides02} Section 15).  

\medskip

The main result of this section is the following.

\begin{theorem}[Domain of the $\Gamma$-limit]\label{t:domain-G-limit}  Let $(u_\e,v_\e) \subset L^1(\O)\times L^1(\O)$ be such that 
$$
(u_\e,v_\e) \to (u,v) \quad \text{in}\quad L^1(\O)\times L^1(\O) \quad \text{and} \quad \sup_{\e>0} E_\e(u_\e,v_\e) < +\infty
$$
and let $\ell$ be as in \eqref{def:ell}.


$(i)$ {\rm (Subcritical and critical regime)} If $\ell\in[0,+\infty)$ then $u\in GSBV^2(\O)$ and $v=1$ a.e. in $\O$.

\smallskip

$(ii)$ {\rm (Supercritical regime)} If $\ell=+\infty$ then $u\in W^{1,2}(\O)$ and $v=1$ a.e. in $\O$.

\end{theorem}

The proof of Theorem \ref{t:domain-G-limit} will be carried out in Proposition \ref{liminf:1d} and Proposition \ref{compactness} below. 

\subsection{The one-dimensional case}
In this subsection we deal with the case $n=1$. 

In what follows we only consider the case $\Omega=I:=(a,b)$ with $a,b\in\R$, $a<b$. The case of a general open set can be treated by repeating the proof below in each connected component of $\Omega$.
\begin{prop}\label{liminf:1d}
Let $(u_\e,v_\e) \subset L^1(I)\times L^1(I)$ be such that 
$$
(u_\e,v_\e) \to (u,v) \quad \text{in}\quad L^1(I)\times L^1(I) \quad \text{and} \quad \sup_{\e>0} E_\e(u_\e,v_\e) < +\infty
$$
and let $\ell$ be as in \eqref{def:ell}.


$(i)$ {\rm (Subcritical and critical regime)} If $\ell\in[0,+\infty)$ then $u\in SBV^2(I)$ and $v=1$ a.e. in $I$. Moreover,
\[E'_\ell(u,v)\geq\int_I(u')^2dt+\#(S_u).\]
\smallskip

$(ii)$ {\rm (Supercritical regime)} If $\ell=+\infty$ then $u\in W^{1,2}(I)$ and $v=1$ a.e. in $I$. Moreover, 
\[E'_\infty(u,v)\geq\int_I(u')^2dt.\]
\end{prop}
\begin{proof} The proof will be divided into two steps.

\medskip

\noindent  {\bf Step 1:} proof of $(i)$\ie the case $\ell\in[0,+\infty)$. 

\smallskip

Let $(u_\e,v_\e)\subset L^1(I)\times L^1(I)$ be as in the statement. We claim that $E_\e(u_\e,v_\e)$ can be bounded from below by $AT_\e(\tilde{u}_\e,\tilde{v}_\e)$ for suitable functions $\tilde{u}_\e$ and $\tilde{v}_\e$ with $(\tilde{u}_\e, \tilde{v}_\e) \to (u,v)$ in $L^1(I)\times L^1(I)$. Then the conclusion follows appealing to the classical Ambrosio and Tortorelli convergence result \cite[Theorem 2.1]{AT2}. 

For our purposes it is convenient to rewrite $E_\e$ as follows
\begin{equation*}
E_\e(u_\e,v_\e)=\sum_{\begin{smallmatrix}i\in I_\delta\\i+\delta\in I\end{smallmatrix}}\delta\frac{(v_\e^i)^2+(v_\e^{i+\delta})^2}{2}\left|\frac{u_\e^i-u_\e^{i+\delta}}{\delta}\right|^2\\
+\frac{1}{2}\bigg(\sum_{i\in I_\delta}\delta\frac{(v_\e^i-1)^2}{\e}+\sum_{\begin{smallmatrix}i\in I_\delta\\i+\delta\in I\end{smallmatrix}}\e\delta\left|\frac{v_\e^i-v_\e^{i+\delta}}{\delta}\right|^2\bigg).
\end{equation*}
We define moreover $\tilde{u}_\e,\tilde{v}_\e$ as the piecewise affine interpolations of $u_\e$, $v_\e$ on $I_\delta$, respectively\ie
\begin{align*}
\tilde{u}_\e(t)&:=u_\e^i+\frac{u_\e^{i+\delta}-u_\e^i}{\delta}(t-i)\quad\text{if}\ t\in[i,i+\delta),\quad i,i+\delta\in I_\delta,\\
\tilde{v}_\e(t)&:=v_\e^i+\frac{v_\e^{i+\delta}-v_\e^i}{\delta}(t-i)\quad\text{if}\ t\in[i,i+\delta),\quad i,i+\delta\in I_\delta.
\end{align*}
We note that $(\tilde{u}_\e, \tilde{v}_\e) \to (u,v)$ in $L^1(I)\times L^1(I)$. 

Let $\eta>0$ be fixed; for $\e$ sufficiently small we have 
\[(a+\eta,b-\eta)\subset\bigcup_{i\in\overset{\circ}{I}_\delta}[i,i+\delta),\]
therefore
\begin{equation}\label{est:01}
\sum_{i\in\overset{\circ}{I}_\delta}\e\delta\left|\frac{v_\e^i-v_\e^{i+\delta}}{\delta}\right|^2=\sum_{i\in\overset{\circ}{I}_\delta}\int_i^{i+\delta}\e(\tilde{v}_\e')^2dt \geq\int_{a+\eta}^{b-\eta}\e(\tilde{v}_\e')^2dt,
\end{equation}
for $\e$ small. Moreover, in view of the definition of $\tilde v_\e$ and the convexity of $z\to(z-1)^2$, for every $i\in\overset{\circ}{I}_\delta$ we get
\begin{align*}
\int_i^{i+\delta}\frac{(\tilde{v}_\e-1)^2}{\e}dt &=\frac{1}{\e}\int_i^{i+\delta}\left(\left(1-\frac{t-i}{\delta}\right)v_\e^i+\frac{t-i}{\delta}v_\e^{i+\delta}-1\right)^2dt\\
&\leq\frac{(v_\e^i-1)^2}{\e}\int_i^{i+\delta}\left(1-\frac{t-i}{\delta}\right)dt+\frac{(v_\e^{i+\delta}-1)^2}{\e}\int_i^{i+\delta}\frac{t-i}{\delta}dt\\
&=\frac{\delta}{2\e}\left((v_\e^i-1)^2+(v_\e^{i+\delta}-1)^2\right),
\end{align*}
from which we get
\begin{equation}\label{est:02}
\sum_{i\in \overset{\circ}{I}_\delta}\delta\frac{(v_\e^i-1)^2}{\e}\geq\int_{a+\eta}^{b-\eta}\frac{(\tilde{v}_\e-1)^2}{\e}dt,
\end{equation}
for $\e$ small. Finally, the definition of $\tilde u_\e$ together with the convexity of $z\to z^2$ yield 
\begin{align*}
\int_i^{i+\delta}(\tilde{v}_\e)^2(\tilde{u}_\e')^2dt &=\left|\frac{u_\e^i-u_\e^{i+\delta}}{\delta}\right|^2\int_i^{i+\delta}\left(\left(1-\frac{t-i}{\delta}\right)v_\e^i+\frac{t-i}{\delta}v_\e^{i+\delta}\right)^2dt\\
&\leq\left|\frac{u_\e^i-u_\e^{i+\delta}}{\delta}\right|^2\left((v_\e^i)^2\int_i^{i+\delta}\left(1-\frac{t-i}{\delta}\right)dt+(v_\e^{i+\delta})^2\int_i^{i+\delta}\frac{t-i}{\delta}dt\right)\\
&=\delta\frac{(v_\e^i)^2+(v_\e^{i+\delta})^2}{2}\left|\frac{u_\e^i-u_\e^{i+\delta}}{\delta}\right|^2,
\end{align*}
for every $i\in\overset{\circ}{I}_\delta$, and thus
\begin{equation}\label{est:03}
\sum_{i\in\overset{\circ}{I}_\delta}\delta\frac{(v_\e^i)^2+(v_\e^{i+\delta})^2}{2}\left|\frac{u_\e^i-u_\e^{i+\delta}}{\delta}\right|^2\geq\int_{a+\eta}^{b-\eta}(\tilde{v}_\e)^2(\tilde{u}_\e')^2dt,
\end{equation}
for $\e$ small. Eventually, gathering \eqref{est:01}-\eqref{est:03} we deduce 
\begin{equation}\label{est:ATE}
E_\e(u_\e,v_\e) \geq AT_\e(\tilde{u}_\e,\tilde{v}_\e,(a+\eta,b-\eta)),
\end{equation}
where $AT_\e$ denotes the Ambrosio-Tortorelli functional\ie
\begin{equation*}
AT_\varepsilon(u,v):=\int_\Omega v^2|\nabla u|^2d\,x +\frac{1}{2}\int_\Omega\frac{(v-1)^2}{\varepsilon}+\varepsilon|\nabla v|^2\,dx,
\end{equation*}
for every $(u,v)\in W^{1,2}(\Omega)\times W^{1,2}(\Omega)$ with  $0\leq v\leq 1$.
Hence the thesis follows first appealing to \cite[Theorem 2.1]{AT2} and then by letting $\eta\to0$.

\medskip

\noindent  {\bf Step 2:} proof of $(ii)$\ie the case $\ell=+\infty$. 

\smallskip

Let $(u_\e,v_\e)\subset L^1(I)\times L^1(I)$ be as in the statement, then in particular

$$
\sup_{\e>0} \sum_{i\in I_\delta} \frac{\delta}{\e}(v_\e^i-1)^2 < +\infty. 
$$
Hence there exists a constant $c>0$ such that for every $i\in I_\delta$ and for every $\e>0$
$$
(v_\e^i-1)^2 \leq c\,\frac{\e}{\delta}.
$$
Let $\eta\in (0,1)$ be arbitrary; since by assumption $\e/\delta \to 0$  there exists $\e_0=\e_0(\eta)>0$ such that $|v_\e^i-1| <\eta$ for every $i\in I_\delta$ and for every $\e\in(0,\e_0)$. Then, up to choosing $\e$ small enough, we have

\begin{equation}\label{supercrit:03}
\frac{1}{2}\sum_{i\in\overset{\circ}{I}_\delta}\delta(v_\e^i)^2\left|\frac{u_\e^i-u_\e^{i\pm\delta}}{\delta}\right|^2\geq (1-\eta)^2\int_{a+\eta}^{b-\eta}(\tilde{u}_\e')^2dt.
\end{equation}
Since $\tilde{u}_\e\to u$ in $L^1(I)$, in view of the bound on the energy, from \eqref{supercrit:03} we may deduce that $\tilde{u}_\e\rightharpoonup u$ in $W^{1,2}(a+\eta,b-\eta)$ so that in particular $u\in W^{1,2}(a+\eta,b-\eta)$. Moreover, \eqref{supercrit:03} entails
\begin{align*}
\liminf_{\e\to 0}E_\e(u_\e,v_\e)&\geq\liminf_{\e\to 0}\frac{1}{2}\sum_{i\in\overset{\circ}{I}_\delta}\delta(v_\e^i)^2\left|\frac{u_\e^i-u_\e^{i\pm\delta}}{\delta}\right|^2\geq(1-\eta)^2\liminf_{\e\to 0}\int_{a+\eta}^{b-\eta}(\tilde{u}_\e')^2dt\\
&\geq(1-\eta)^2\int_{a+\eta}^{b-\eta}(u')^2dt,
\end{align*}
so that the desired lower bound follows by letting $\eta\to 0$.
\end{proof}
\begin{rem}\label{rem:liminf}
{\rm Let $(u_\e,v_\e)\subset L^1(I)\times L^1(I)$ be a sequence such that $(u_\e,v_\e)\to(u,v)$ in $L^1(I)\times L^1(I)$ and $\sup_\e E_\e(u_\e,v_\e)<+\infty$; let moreover $\ell\in [0,+\infty)$. In view of \eqref{est:01}-\eqref{est:03}, arguing as in \cite[Lemma 2.1]{AT2} we note that the two inequalities
\[\liminf_{\e\to 0}F_\e(u_\e,v_\e)\geq\int_a^b(u')^2dt, \qquad \liminf_{\e\to 0}G_\e(v_\e)\geq\#(S_u)\]
also hold.} 
\end{rem}

\subsection{The $n$-dimensional case} 

In this section we deal with the case $n\geq 2$. The following proposition will be obtained by combining the one-dimensional result in Proposition \ref{liminf:1d} and a slicing procedure in the coordinate directions. 

To this end it is convenient to introduce the following notation. For every $k\in \{1,\ldots,n\}$ set $\Pi^k:=\{x\in\R^n \colon  x_k=0\}$ and let $p^k:\R^n\to\Pi^k$ be the orthogonal projection onto $\Pi^k$. For all $y\in\Pi^k$ let
\begin{equation}\label{def:slice1}
\Omega_{k,y}:=\{t\in\R: y+te_k\in\Omega\}
\end{equation}
and
\begin{equation}\label{def:slice2}
\O_k:=\{y\in\Pi^k:\ \Omega_{k,y}\neq\emptyset\}.
\end{equation}
For every $w\colon \O \to \R$, $t\in\Omega_{k,y}$, and $y\in \O_k$ we set 
\begin{equation}\label{def:slice-fun}
w^{k,y}(t):=w(y+te_k). 
\end{equation}
\begin{prop}\label{compactness}
Let $(u_\e,v_\e) \subset L^1(\O)\times L^1(\O)$ be such that 
$$
(u_\e,v_\e) \to (u,v) \quad \text{in}\quad L^1(\O)\times L^1(\O) \quad \text{and} \quad \sup_{\e>0} E_\e(u_\e,v_\e) < +\infty
$$
and let $\ell$ be as in \eqref{def:ell}.


$(i)$ {\rm (Subcritical and critical regime)} If $\ell\in[0,+\infty)$ then $u\in GSBV^2(\O)$ and $v=1$ a.e. in $\O$. Moreover,
\begin{equation}\label{eq:n02}
E'_\ell(u,v)\geq\int_\Omega|\nabla u|^2dx+\int_{S_u\cap\Omega}|\nu_u|_\infty d\HH^{n-1}.
\end{equation}
\smallskip

$(ii)$ {\rm (Supercritical regime)} If $\ell=+\infty$ then $u\in W^{1,2}(\O)$ and $v=1$ a.e. in $\O$. Moreover, 
\begin{equation}\label{s:liminf-super}
E'_\infty(u,v)\geq\int_\O |\nabla u|^2\,dx.
\end{equation}
\end{prop}
\begin{proof}
The proof will be divided into two steps.

\medskip

\noindent  {\bf Step 1:} proof of $(i)$\ie the case $\ell\in[0,+\infty)$. 

\smallskip

Let $(u_\e,v_\e)\subset L^1(\Omega)\times L^1(\Omega)$ be a sequence converging to $(u,v)$ in $L^1(\Omega)\times L^1(\Omega)$ and such that $\sup_\e E_\e(u_\e,v_\e)<+\infty$. Note that $v_\e\to 1$ in $L^2(\Omega)$, so that $v=1$ a.e. in $\Omega$. 

We now show that $u\in GSBV^2(\Omega)$. To this end let $k\in\{1,\cdots,n\}$ be fixed and for every $y\in \Omega_k$ consider the two sequences of functions $(u_\e^{k,y})$, $(v_\e^{k,y})$ defined on $\Omega_{k,y}$ as in \eqref{def:slice-fun} with $w$ replaced by $u_\e$ and $v_\e$, respectively. 
Let $\eta>0$ be fixed; set $\Omega^\eta:=\{x\in\Omega:\ \dist(x,\R^n\setminus\Omega)>\eta\}$ and let $\Omega^\eta_{k,y}$ be as in \eqref{def:slice1} with $\Omega$ replaced by $\Omega^\eta$. Moreover let $\Omega^\eta_{k}$ be defined according to \eqref{def:slice2}. 

Set $\Pi^k_\delta:=\Pi^k \cap \delta \Z^{n-1}$; a direct computation yields
\begin{align}\label{est:n01}
E_\e &(u_\e,v_\e)\nonumber\\
&\geq\delta^{n-1}\,\bigg(\frac{1}{2}\sum_{\begin{smallmatrix}i\in \O_\delta\\i\pm\delta e_k\in \O_\delta\end{smallmatrix}}\delta(v_\e^i)^2\left|\frac{u_\e^i-u_\e^{i\pm\delta e_k}}{\delta}\right|^2+\frac{1}{2}\bigg(\sum_{i\in\Omega_\delta} \delta\frac{(v_\e^i-1)^2}{\e}+\hspace{-.3cm}\sum_{\begin{smallmatrix}i\in \O_\delta\\i+\delta e_k\in \O_\delta\end{smallmatrix}}\e \delta\left|\frac{v_\e^i-v_\e^{i+\delta e_k}}{\delta}\right|^2\bigg)\bigg)\nonumber\\
&= \delta^{n-1}\,\sum_{j\in \Pi^k_\delta}\Bigg(\frac{1}{2}\sum_{\begin{smallmatrix}i\in \O_{k,j}\cap \delta\Z\\i\pm\delta e_k\in \O_{k,j}\end{smallmatrix}}\delta(v_\e^{k,j}(i))^2\left|\frac{u_\e^{k,j}(i)-u_\e^{k,j}({i\pm\delta e_k})}{\delta}\right|^2
\\
& \hspace{2.5cm} +\frac{1}{2}\bigg(\sum_{i\in\Omega_{k,j}\cap \delta\Z}\delta\frac{(v_\e^{k,j}(i)-1)^2}{\e}+\sum_{\begin{smallmatrix}i\in \O_{k,j}\cap \delta\Z\\i+\delta e_k\in \O_{k,j}\end{smallmatrix}}\e \delta\left|\frac{v_\e^{k,j}(i)-v_\e^{k,j}({i+\delta e_k})}{\delta}\right|^2\bigg)\Bigg)\nonumber
\\\nonumber
& \geq \int_{\O^\eta_k} \Big(F^k_\e(u_\e^{k,y},v_\e^{k,y},\O_{k,y})+G^k_\e(v_\e^{k,y},\O_{k,y})\Big)\,d\HH^{n-1}(y), 
\end{align}
with
$$
F_\e^k(u_\e^{k,y},v_\e^{k,y},\Omega_{k,y}):=\frac{1}{2}\sum_{\begin{smallmatrix}i\in \O_{k,y}\cap \delta\Z\\i\pm\delta e_k\in \O_{k,y}\end{smallmatrix}}\delta(v_\e^{k,y}(i))^2\left|\frac{u_\e^{k,y}(i)-u_\e^{k,y}({i\pm\delta e_k})}{\delta}\right|^2
$$
and
$$
G_\e^k(v_\e^{k,y},\Omega_{k,y}):=\frac{1}{2}\bigg(\sum_{i\in\Omega_{k,j}\cap \delta\Z}\delta\frac{(v_\e^{k,j}(i)-1)^2}{\e}+\sum_{\begin{smallmatrix}i\in \O_{k,j}\cap \delta\Z\\i+\delta e_k\in \O_{k,y}\end{smallmatrix}}\e\delta\left|\frac{v_\e^{k,y}(i)-v_\e^{k,y}({i+\delta e_k})}{\delta}\right|^2\bigg),
$$
where $u_\e$ and $v_\e$ are identified with their piecewise-constant interpolations. 

Then invoking Fatou's Lemma gives
\[\liminf_{\e\to 0}E_\e(u_\e,v_\e)\geq\int_{\Omega_k^\eta}\liminf_{\e\to 0}\left(F_\e^k(u_\e^{k,y},v_\e^{k,y},\Omega_{k,y})+G_\e^k(v_\e^{k,y},\Omega_{k,y})\right)d\HH^{n-1}(y).\]
Since for $\HH^{n-1}$-a.e $y\in \Omega_k$ there holds $u_\e^{k,y}\to u^{k,y}$ in $L^1(\Omega_{k,y})$,  Proposition \ref{liminf:1d}-$(i)$ together with Remark \ref{rem:liminf} yield $u\in SBV^2(\Omega_{k,y}^\eta)$ for $\HH^{n-1}$-a.e. $y\in\Omega_k^\eta$ and
\begin{align}\label{est:n03}
\liminf_{\e\to 0} E_\e(u_\e,v_\e)&\geq\int_{\Omega^\eta_k}\liminf_{\e\to 0} F_\e(u_\e^{k,y},v_\e^{k,y},\Omega_{k,y})+\liminf_{\e\to 0} G_\e(v_\e^{k,y},\Omega_{k,y})d\HH^{n-1}(y)\nonumber\\
&\geq\int_{\Omega^\eta_k}\left(\int_{\Omega^\eta_{k,y}}((u^{k,y})')^2dt+\#\left(S_{u^{k,y}}\cap\Omega^\eta_{k,y}\right)\right)d\HH^{n-1}(y).
\end{align}
Since \eqref{est:n03} holds for every $k\in\{1,\ldots,n\}$, applying \cite[Theorem 4.1 and Remark 4.2]{braides98} we deduce that $u\in GSBV^2(\Omega^\eta)$.\\
\indent In order to prove the lower bound \eqref{eq:n02} we notice that for every $k\in\{1,\ldots,n\}$ we also have
\begin{align*}
\liminf_{\e\to 0} &E_\e(u_\e,v_\e)\\
&\geq\sum_{l=1}^n\int_{\Omega^\eta_l}\liminf_{\e\to 0}F_\e^l(u_\e^{l,y},v_\e^{l,y},\Omega_{l,y})d\HH^{n-1}(y)+\int_{\Omega_k^\eta}\liminf_{\e\to 0}G_\e^k(v_\e^{k,y},\Omega_{k,y})d\HH^{n-1}(y)\\
&\geq\sum_{l=1}^n\int_{\Omega^\eta_l}\left(\int_{\Omega^\eta_{l,y}}((u^{l,y}(t))')^2dt\right)d\HH^{n-1}(y)+\int_{\Omega_k^\eta}\#\left(S_{u^{k,y}}\cap\Omega^\eta_{k,y}\right)d\HH^{n-1}(y)\\
&=\sum_{l=1}^n\int_{\Omega^\eta}\left(\frac{\partial u}{\partial x_l}\right)^2dx+\int_{S_u\cap\Omega^\eta}|\left<\nu_u,e_k\right>|d\HH^{n-1}\\
&=\int_{\Omega^\eta}|\nabla u|^2dx+\int_{S_u\cap\Omega^\eta}|\left<\nu_u,e_k\right>|d\HH^{n-1}.
\end{align*}
Then taking the sup on $k\in\{1,\ldots,n\}$ we get
\[\liminf_{\e\to 0}E_\e(u_\e,v_\e)\geq\int_{\Omega^\eta}|\nabla u|^2dx+\int_{S_u\cap\Omega^\eta}|\nu_u|_\infty d\HH^{n-1}.\]
Finally, by letting $\eta\to 0$ we both deduce that $u\in GSBV^2(\Omega)$ and \eqref{eq:n02}.
\medskip

\noindent  {\bf Step 2:} proof of $(ii)$\ie the case $\ell=+\infty$. 

\smallskip

Arguing as in Step $1$ and now appealing to Proposition \ref{liminf:1d}-$(ii)$ yield both $u\in W^{1,2}(\O)$ and the lower-bound estimate \eqref{s:liminf-super}.   
\end{proof}
\begin{rem}\label{rem:liminf-nd}
{\rm From the proof of \eqref{eq:n02} in Proposition \ref{compactness}-(i) we get that if $\ell \in [0,+\infty)$ and $(u_\e,v_\e)\subset L^1(\O)\times L^1(\O)$ is such that $(u_\e,v_\e)\to(u,v)$ in $L^1(\O)\times L^1(\O)$ with $\sup_\e E_\e(u_\e,v_\e)<+\infty$ then the two inequalities
\[\liminf_{\e\to 0}F_\e(u_\e,v_\e)\geq\int_\O|\nabla u|^2\dx, \qquad \liminf_{\e\to 0}G_\e(v_\e)\geq\int_{S_u\cap \O}|\nu_u|_\infty\,d\HH^{n-1}\]
hold also true.} 
\end{rem}

\begin{rem}\label{lowerbound:local}
{\rm For later use we notice that Proposition \ref{compactness} can be localised in the following sense. Let $\ell\in[0,+\infty)$, $U\in\mathscr A_L(\Omega)$ and let $E'_\ell$ be as in \eqref{c:li-ls}. Then
\[E'_\ell(u,1,U)\geq\int_U|\nabla u|^2dx+\int_{S_u\cap U}|\nu_u|_\infty d\HH^{n-1}\quad\text{for every }\; u\in GSBV^2(\Omega).\]
}
\end{rem}

\subsection{Convergence of minimisation problems}

On account of the $\G$-convergence result Theorem \ref{t:main-result} in this subsection we establish a convergence result for a class of minimisation problems associated to $E_\e$.
Specifically, we consider a suitable perturbation of $E_\e$ which will also satisfy the needed equi-coercivity property. To this end, having in mind applications to image-segmentation problems, for a given $g\in L^\infty(\Omega)$ we define 
\[g^i:=\frac{1}{\delta^n}\int_{i+[0,\delta)^n}g(x)dx\]
and consider the functionals
\begin{equation}\label{add:fidelity}
{E}^g_\e(u,v):=E_\e(u,v)+\sum_{i\in\Omega_\delta}\delta^n|u^i-g^i|^2. 
\end{equation}
Moreover, we will only focus on the subcritical and critical regimes\ie $\ell\in [0,+\infty)$, as these are the only regimes giving rise to a nontrivial $\Gamma$-limit.  

\medskip

The following result holds true.

\begin{prop}[Equicoercivity] \label{prop:equi-coer}
Let $\ell \in [0,+\infty)$ and $g \in L^\infty(\O)$. Then the functionals $E^g_\e$ defined as in \eqref{add:fidelity}
are equi-coercive with respect to the strong $L^1(\O)\times L^1(\O)$-topology. More precisely, for every sequence $(u_\e,v_\e) \subset L^1(\O)\times L^1(\O)$ satisfying $\sup_\e E^g_\e(u_\e,v_\e)<+\infty$ there exist a subsequence (not relabelled) and a function $u\in GSBV^2(\O)$ such that $(u_\e,v_\e) \to (u,1)$ in $L^1(\O)\times L^1(\O)$. 
\end{prop}
\begin{proof}
For $n=1$ the thesis directly follows by combining the estimate \eqref{est:ATE} with the equi-coercivity of the (perturbed) Ambrosio-Tortorelli functional \cite[Theorem 1.2]{AT2}.  

For $n\geq 2$ the proof is also standard and the equi-coercivity of $E^g_\e$ follows from the one-dimensional case invoking \textit{e.g.} \cite[Theorem 6.6]{ABS98} (see also \cite[Section 3.8]{Alberti00}).  
\end{proof}

We are now ready to prove the following result on the convergence of the associated minimisation problems.

\begin{corollary}[Convergence of minimisation problems]\label{conv:min}
For every fixed $\e>0$ the minimisation problem
\begin{equation*}
m_\e:=\min\left\{{E}^g_\e(u,v):\ (u,v)\in \A_\e(\Omega)\times\A_\e(\Omega)\right\}.
\end{equation*}
admits a solution $(\hat{u}_\e,\hat{v}_\e)$. 

Let $\ell\in [0,+\infty)$; then, up to subsequences, the pair $(\hat{u}_\e,\hat{v}_\e)$ converges in $L^1(\Omega)\times L^1(\Omega)$ to $(\hat{u},1)$ with $\hat{u}$ solution to 
\begin{equation}\label{min:problem}
m_\ell:=\min\left\{E_\ell(u,1)+\int_\Omega|u-g|^2dx:\ u\in SBV^2(\Omega), \; \|u\|_{L^\infty(\O)}\leq \|g\|_{L^\infty(\O)}\right\};
\end{equation}
moreover, $m_\e\to m_\ell$ as $\e \to 0$. 
\end{corollary}
\begin{proof}
The existence of a minimising pair $(\hat{u}_\e,\hat{v}_\e)$ follows by applying the direct methods. Indeed, let $\e>0$ be fixed and let $(u_k,v_k)$ be a minimising sequence  for $E^g_\e$. Then there exists a constant $c>0$ such that
\[\sum_{i\in\Omega_\delta}\delta^n|u_k^i-g^i|^2\leq c\,,\]
for every $k\in\N$. Since for every $i\in \O_\delta$ it holds $|g^i|\leq \|g\|_{L^\infty(\O)}$, we deduce that $|u_k^i|\leq c$ for every $i\in\Omega_\delta$ and every $k\in \N$ (where now the constant $c$ possibly depends on $\e$).
Hence, up to subsequences not relabelled, $\lim_{k}u_k^i =\hat{u}_\e^i$, for some $\hat{u}_\e^i\in\R$. Since moreover, $0\leq v_k^i\leq 1$ for every $k\in\N$ and every $i\in\Omega_\delta$, up to subsequences, we also have $\lim_k v_k^i = \hat{v}_\e^i$ for some $\hat{v}_\e^i\in[0,1]$. 
Since for fixed $\e>0$ the set $\left\{(u_k^i,v_k^i): i\in\Omega_\delta\right\}$ is finite, up to choosing a diagonal sequence we can always  assume that
\[
\lim_{k\to +\infty}(u_k^i,v_k^i)= (\hat{u}_\e^i,\hat{v}_\e^i)\quad\text{for every}\quad i\in\Omega_\delta.
\]
Then, to deduce that $(\hat{u}_\e,\hat{v}_\e)$ is a minimising pair for $E^g_\e$ it suffices to notice that
\[m_\e=\liminf_{k\to+\infty}E^g_\e(u_k,v_k)\geq E^g_\e(\hat{u}_\e,\hat{v}_\e).\]
Since $E_\e$ decreases by truncations in $u$, by definition of $E^g_\e$ it is not restrictive to assume that $\|\hat{u}_\e\|_{L^\infty(\O)}\leq \|g\|_{L^\infty(\O)}$.  
Moreover, invoking Proposition \ref{prop:equi-coer} gives the existence of a subsequence (not relabelled) $(\hat{u}_\e,\hat{v}_\e)$ and of a function $\hat u \in GSBV^2(\Omega)$ such that $(\hat{u}_\e,\hat{v}_\e)\to(\hat{u},1)$ in $L^1(\Omega)\times L^1(\Omega)$; further, the Dominated Convergence Theorem also yields $\hat u_\e \to \hat u$ in $L^2(\O)$.  Since clearly $\|\hat{u}\|_{L^\infty(\O)}\leq \|g\|_{L^\infty(\O)}$, we actually deduce that $\hat u \in SBV^2(\Omega)$. 
Then it only remains to show that $\hat{u}$ is a solution to \eqref{min:problem}. To this end, let $g_\e\in\A_\e(\Omega)$ be the piecewise-constant function defined by 
\[g_\e(x):=g^i\quad\text{for every}\; x\in i+[0,\delta)^n,\quad\text{for every}\; i\in \Omega_\delta.\]
Then $g_\e\to g$ a.e.\ in $\O$; since moreover $\|g_\e\|_{L^\infty(\O)}\leq\|g\|_{L^\infty(\O)}$ the Dominated Convergence Theorem guarantees that $g_\e\to g$ in $L^2(\Omega)$. 
Therefore Theorem \ref{t:main-result} gives
\begin{equation}\label{c:conv-min-1}
m_\ell \leq E_\ell(\hat u,1) + \int_{\O}|\hat u -g|^2\dx \leq \liminf_{\e \to 0}E_\e^g(\hat u_\e,\hat v_\e)= \liminf_{\e \to 0} m_\e. 
\end{equation}
On the other hand, for every $w\in GSBV^2(\O) \cap L^\infty(\O)$ with $\|w\|_{L^\infty(\O)}\leq \|g\|_{L^\infty(\O)}$ Theorem \ref{t:main-result} provides us with a sequence $(u_\e,v_\e)$ such that $(u_\e,v_\e) \to (w,1)$ in $L^1(\Omega)\times L^1(\Omega)$ and
$$
\limsup_{\e \to 0} E_\e(u_\e,v_\e) \leq E_\ell(w,1). 
$$ 
Then if $(\bar u_\e)$ is the sequence obtained by truncating $(u_\e)$ at level $\|w\|_{L^\infty(\O)}$, we clearly have $\bar u_\e \to w$ in $L^2(\O)$ and  
$$
\limsup_{\e \to 0} E_\e(\bar u_\e,v_\e) \leq E_\ell(w,1), 
$$ 
so that 
$$
\limsup_{\e \to 0} m_\e \leq \limsup_{\e \to 0} E^g_\e(\bar u_\e, v_\e) \leq E_\ell(w,1)+\int_{\O}|w -g|^2\dx.
$$
Hence by the arbitrariness of $w$ we get
\begin{equation}\label{c:conv-min-2}
\limsup_{\e \to 0} m_\e \leq m_\ell.
\end{equation}
Eventually, gathering \eqref{c:conv-min-1} and \eqref{c:conv-min-2} yields
$$
\lim_{\e\to 0}m_\e=m_\ell=E_\ell(\hat u,1) + \int_{\O}|\hat u -g|^2\dx
$$
and thus the thesis. 
\end{proof}
%

%
%

\section{Proof of the $\G$-convergence result in the subcritical regime $\ell=0$}\label{s:subcritical} 

In this section we study the $\Gamma$-convergence of $E_\e$ in the subcritical regime\ie when $\ell=0$.
This regime corresponds to the case where the mesh-size $\delta$ is much smaller than the approximation parameter $\e$. 
We show that under such assumptions the discreteness of the problem does not play a role in the limit behaviour of the functionals $E_\e$ whose  $\Gamma$-limit is actually given by the Mumford-Shah functional, as in the continuous case.

We recall the following definition
$$
E_0(u,v):=M\!S(u,v)=\begin{cases}
\ds\int_\O |\nabla u|^2\dx+\HH^{n-1}(S_u\cap \O) & \text{if} \; u\in GSBV^2(\O), \, v =1 \; \text{a.e. in}\; \O,
\cr
+\infty & \text{otherwise in}\; L^1(\O)\times L^1(\O). 
\end{cases}
$$

\medskip

We deal with the lower-bound and upper-bound inequalities separately. 

\subsection{Lower-bound inequality} In this subsection we establish the liminf inequality for $E_\e$ when $\ell=0$. The main result of this subsection is as follows. 

\begin{prop}[Lower-bound for $\ell=0$]\label{prop:liminf}
Let $E_\e$ be as in \eqref{def:funct}  and $\ell=0$. Then for every $(u,v)\in L^1(\Omega)\times L^1(\Omega)$ and every $(u_\e,v_\e)\subset L^1(\Omega)\times L^1(\Omega)$ with $(u_\e,v_\e) \to (u,v)$ in $L^1(\Omega)\times L^1(\Omega)$ we have
\[
\liminf_{\e \to 0} E_\e(u_\e,v_\e)\geq E_0(u,v).\]
\end{prop}
\begin{proof}
Up to subsequences we can always assume that $\sup_\e E_\e(u_\e,v_\e)<+\infty$, otherwise there is nothing to prove. Then Proposition \ref{compactness}-(i) gives $u\in GSBV^2(\Omega)$ and $v=1$ a.e. in $\Omega$. Thus, it remains to show that
\begin{equation}\label{liminf}
\liminf_{\e \to 0} E_\e(u_\e,v_\e)\geq\int_\Omega|\nabla u|^2dx+\HH^{n-1}(S_u\cap\Omega)
\end{equation}
holds true for every $u\in GSBV^2(\Omega)$. We first prove \eqref{liminf} for $u\in SBV^2(\Omega)$. To this end we use the Fonseca and M\"uller blow-up procedure \cite{FM}. 
For every $\e>0$ we define the discrete Radon measure
\[\mu_\e:=\sum_{i\in\overset{\circ}{\Omega}_\delta}\frac{\delta^n}{2}\left((v_\e^i)^2\sum_{k=1}^n\left|\frac{u_\e^i-u_\e^{i\pm\delta e_k}}{\delta}\right|^2+\frac{(v_\e^i-1)^2}{\e}+\sum_{k=1}^n\e\left|\frac{v_\e^i-v_\e^{i+\delta e_k}}{\delta}\right|^2\right)\mathbf{1}_i,\]
where $\mathbf{1}_i$ denotes the Dirac delta in $i$. Since 
$$
\sup_{\e>0}\mu_\e(\Omega)\leq \sup_{\e>0} E_\e(u_\e,v_\e)<+\infty
$$ 
we deduce that, up to subsequences, $\mu_\e\overset{*}{\rightharpoonup}\mu$ weakly* in the sense of measures, for some positive finite Radon measure $\mu$. Then, appealing to the Radon-Nikod\`ym Theorem we can write $\mu$ as the sum of three mutually orthogonal measures
\[\mu=\mu_a\LL^n\llcorner\Omega+\mu_J\HH^{n-1}\llcorner (S_u\cap\Omega)+\mu_s.\]
We claim that
\begin{equation}\label{est:abs}
\mu_a(x_0)\geq |\nabla u(x_0)|^2\quad\text{for}\ \LL^n\text{-a.e.}\ x_0\in\Omega
\end{equation}
and
\begin{equation}\label{est:jump}
\mu_J(x_0)\geq 1\quad\text{for}\ \HH^{n-1}\text{-a.e.}\ x_0\in S_u.
\end{equation}
Let us assume for the moment that we can prove \eqref{est:abs} and \eqref{est:jump}. Then, by choosing an increasing sequence of cut-off functions $(\varphi_k) \subset C_c^\infty(\Omega)$ such that $0\leq\varphi_k\leq 1$ and $\sup_k\varphi_k=1$ we get 
\begin{align*}
\liminf_{\e\to 0} E_\e(u_\e,v_\e) &\geq\liminf_{\e\to 0} \mu_\e(\Omega)\geq\liminf_{\e\to 0}\int_\Omega\varphi_kd\mu_\e=\int_\Omega\varphi_kd\mu\\
&\geq\int_\Omega\varphi_kd\mu_a+\int_{S_u\cap\Omega}\varphi_kd\mu_J\\
&\geq\int_\Omega|\nabla u|^2\varphi_k dx+\int_{S_u\cap\Omega}\varphi_k d\HH^1,
\end{align*}
hence the conclusion follows by letting $k\to+\infty$ and appealing to the Monotone Convergence Theorem.\\
\indent Let us now prove \eqref{est:abs} and \eqref{est:jump}. 

\medskip

\noindent {\bf Step 1:} proof of \eqref{est:abs}.

\smallskip

By virtue of the Besicovitch Derivation Theorem and the Calder\`on-Zygmund Lemma for $\LL^n$-a.e. $x_0\in\Omega$ we have 
\begin{equation}\label{leb:point}
\mu_a(x_0)=\lim_{\rho\to 0^+}\frac{\mu(Q_\rho(x_0))}{\LL^n(Q_\rho(x_0))}=\lim_{\rho\to 0^+}\frac{\mu(Q_\rho(x_0))}{\rho^n}
\end{equation}
and
\begin{equation}\label{appr:diff}
\lim_{\rho\to 0^+}\frac{1}{\rho^{n+1}}\int_{Q_\rho(x_0)}|u(x)-u(x_0)-\left<\nabla u(x_0),x-x_0\right>|dx=0.
\end{equation}
Let $x_0\in \O$ be fixed and such that both \eqref{leb:point} and \eqref{appr:diff} hold true.  
Since $\mu$ is a finite Radon measure, there holds $\mu(Q_\rho(x_0))=\mu(\overline{Q}_\rho(x_0))$ except for a countable family of $\rho$'s. Moreover, for $\rho$ sufficiently small the upper semicontinuous function $\chi_{\overline{Q}_\rho}$ has compact support in $\Omega$. Thus, appealing to \cite[Proposition 1.62(a)]{AFP} we deduce that for every $\rho_m\to0$ and every $\e_j\to0$ we have
\begin{align*}
\mu_a(x_0)&=\lim_{m\to +\infty}\frac{1}{\rho_m^n}\mu(Q_{\rho_m}(x_0))=\lim_{m\to +\infty}\frac{1}{\rho_m^n}\mu(\overline{Q}_{\rho_m}(x_0))=\lim_{m\to +\infty}\frac{1}{\rho_m^n}\int_\Omega\chi_{\overline{Q}_{\rho_m}(x_0)}d\mu\\
&\geq\lim_{m\to +\infty}\limsup_{j\to +\infty}\frac{1}{\rho_m^n}\int_\Omega\chi_{\overline{Q}_{\rho_m}(x_0)}d\mu_{\e_j}\geq\lim_{m\to +\infty}\limsup_{j\to +\infty}\frac{1}{\rho_m^n}\mu_{\e_j}(Q_{\rho_m}(x_0)).
\end{align*}
We now want to estimate $\frac{1}{\rho_m^n}\mu_{\e_j}(Q_{\rho_m}(x_0))$. To this end, we notice that for every $j$ and for every $m$ we can find $x_0^j \in\delta_j\Z^n$ and $\rho_{m,j}>0$ such that $x_0^j \to x_0$, $\rho_{m,j} \to \rho_m$, as $j\to +\infty$ and 
$$
\delta_j\Z^n\cap Q_{\rho_{m,j}}(x^j_0) = \delta_j\Z^n\cap Q_{\rho_m}(x_0)
$$
so that
\begin{align*}
&\frac{1}{\rho_m^n}\mu_{\e_j}(Q_{\rho_m}(x_0))
\\
&=\frac{1}{\rho_m^n}\hspace*{-0.5em}\sum_{i\in \delta_j\Z^n\cap Q_{\rho_{m,j}}(x^j_0)}\hspace*{-0.5em}\frac{\delta_j^n}{2}\left((v_{\e_j}^i)^2\sum_{k=1}^n\left|\frac{u_{\e_j}^i-u_{\e_j}^{i\pm\delta_j e_k}}{\delta_j}\right|^2+\frac{(v_{\e_j}^i-1)^2}{\e_j}+\sum_{k=1}^n\e_j\left|\frac{v_{\e_j}^i-v_{\e_j}^{i+\delta_j e_k}}{\delta_j}\right|^2\right)\\
&\geq\hspace*{-1em}\sum_{l\in\frac{\delta_j}{\rho_{m,j}}\Z^n\cap Q}\hspace*{-1em}\frac{1}{2}\bigg(\frac{\delta_j}{\rho_m}\bigg)^n\Bigg((v_{j,m}^l)^2\sum_{k=1}^n\bigg|\frac{u_{j,m}^l-u_{j,m}^{l\pm\frac{\delta_j}{\rho_{m,j}}e_k}}{\delta_j}\bigg|^2\Bigg)
\end{align*}
where 
\begin{equation}\label{def:u-jm}
u_{j,m}^l:=u_{\e_j}^{x_0^j+\rho_{m,j} l}, v_{j,m}^l:=v_{\e_j}^{x_0^j+\rho_{m,j} l}\quad \text{for every}\quad l\in\frac{\delta_j}{\rho_{m,j}}\Z^n\cap Q.
\end{equation}
Now we define a new sequence $w_{j,m}$ on the lattice $\frac{\delta_j}{\rho_{m,j}}\Z^n$ by setting
\[w_{j,m}^l:=\frac{u_{j,m}^l-u(x_0)}{\rho_{m,j}}\quad\text{for every }\ l\in\frac{\delta_j}{\rho_{m,j}}\Z^n\cap Q.\]
Since $u_{j,m}\to u(x_0+\rho_m\,\cdot)$ in $L^1(Q)$ as $j\to +\infty$, appealing to \eqref{appr:diff} we deduce that by letting first $j\to +\infty$ and then $m\to +\infty$ we get $w_{j,m}\to w_0$ in $L^1(Q)$, where $w_0(x):=\langle \nabla u(x_0),x\rangle$ for every $x\in Q$. Moreover by definition 
\[\frac{u_{j,m}^l-u_{j,m}^{l\pm\frac{\delta_j}{\rho_{m,j}}e_k}}{\delta_j}=\frac{w_{j,m}^l-w_{j,m}^{l\pm\frac{\delta_j}{\rho_{m,j}}e_k}}{\frac{\delta_j}{\rho_{m,j}}}\quad\text{for every }\ l\in\frac{\delta_j}{\rho_{m,j}}\Z^n\cap Q.\]
Hence
\begin{align*}
\frac{1}{\rho_m^n}\mu_{\e_j}(Q_{\rho_m}(x_0))&\geq \hspace*{-0.5em}\sum_{l\in\frac{\delta_j}{\rho_{m,j}}\Z^n\cap Q}\hspace*{-0.5em}\frac{1}{2}\Big(\frac{\delta_j}{\rho_m}\Big)^n\Bigg((v_{j,m}^l)^2\sum_{k=1}^n\bigg|\frac{w_{j,m}^l-w_{j,m}^{l\pm\frac{\delta_j}{\rho_{m,j}}e_k}}{\frac{\delta_j}{\rho_{m,j}}}\bigg|^2\Bigg)\\
& = \bigg(\frac{\rho_{m,j}}{\rho_m}\bigg)^n\, F_{\frac{\e_j}{\rho_{m,j}}}(w_{j,m},v_{j,m},Q),
\end{align*}
then, a standard diagonalisation argument yields the existence of a sequence $m_j \to +\infty $ as $j\to +\infty$ such that $\sigma_j:=\frac{\e_j}{\rho_{{m_j,j}}}\to 0$, $(w_{j,m_j},v_{j,m_j})\to(w_0,1)$ in $L^1(Q)\times L^1(Q)$, as $j\to+\infty$, and
\begin{align*}
\mu_a(x_0)
&\geq\liminf_{j\to+\infty}F_{\sigma_j}(w_{j,m_j},v_{j,m_j},Q).
\\
&\geq\int_Q|\nabla w_0|^2dx=|\nabla u(x_0)|^2,
\end{align*}
where the second inequality follows  by Remark \ref{rem:liminf-nd} in view of the equiboundedness of the total energy $E_{\sigma_j}(w_{j,m_j},v_{j,m_j},Q)$.

\medskip

\noindent {\bf Step 2:} proof of \eqref{est:jump}.

\smallskip

To prove \eqref{est:jump} let $x_0\in S_u\cap\Omega$ be such that
\begin{equation}\label{leb:point2}
\mu_J(x_0)=\lim_{\rho\to 0^+}\frac{\mu\left(Q_\rho^\nu(x_0)\right)}{\HH^{n-1}\left(Q_\rho^\nu(x_0)\cap S_u\right)}=\lim_{\rho\to 0^+}\frac{\mu(Q_\rho^\nu(x_0))}{\rho^{n-1}}
\end{equation}
and
\begin{equation}\label{approx:dis}
\lim_{\rho\to 0^+}\frac{1}{\rho^n}\int_{(Q_\rho^\nu(x_0))^\pm}|u(x)-u^\pm(x_0)|dx=0,
\end{equation}
where $\nu:=\nu_u(x_0)$ and $(Q_\rho^\nu(x_0))^\pm:=\{x\in Q_\rho^\nu(x_0):\ \pm\left<x-x_0,\nu\right>>0\}$. We notice that \eqref{leb:point2} and \eqref{approx:dis} hold true for $\HH^{n-1}$-a.e. $x_0\in S_u\cap\Omega$.  

Then, arguing as in Step 1 we deduce that for every $\rho_m\to0$ and every $\e_j\to0$ we have
\begin{align*}
\mu_J(x_0)\geq\lim_{m\to +\infty}\limsup_{j\to +\infty}\frac{1}{\rho_m^{n-1}}\mu_{\e_j}(Q^\nu_{\rho_m}(x_0)),
\end{align*}
hence we are now led to estimate $\frac{1}{\rho_m^{n-1}}\mu_{\e_j}(Q_{\rho_m}^\nu(x_0))$ from below. To this end, arguing as in Step 1, by a change of variables we have
\begin{align*}
\frac{1}{\rho_m^{n-1}}\mu_{\e_j}(Q_{\rho_m}^\nu(x_0))\geq \bigg(\frac{\rho_{m,j}}{\rho_m}\bigg)^{n-1}\, E_{\frac{\e_j}{\rho_{m,j}}}(u_{j,m},v_{j,m},Q^\nu),
\end{align*}
where $u_{j,m}$ and $v_{j,m}$ are as in \eqref{def:u-jm}. Set
\[u_0(x):=
\begin{cases}
u^+(x_0) &\text{if}\ \left<x,\nu\right>\geq 0\\
u^-(x_0) &\text{if}\ \left<x,\nu\right><0;
\end{cases}\]
in view of \eqref{approx:dis} we deduce that $u_{j,m}\to u_0$ in $L^1(Q^\nu)$ if we first let $j\to +\infty$ and then $m\to +\infty$.
Then appealing to a diagonalisation argument we may find a sequence $m_j \to +\infty$ as $j\to +\infty$ such that
$\sigma_j:=\frac{\e_j}{\rho_{{m_j,j}}}\to 0$, $(u_j,v_j):=(u_{j,m_j},v_{j,m_j})\to(u_0,1)$ in $L^1(Q^\nu)\times L^1(Q^\nu)$, as $j\to+\infty$, and
\begin{align*}
\mu_J(x_0)\geq\liminf_{j\to+\infty}E_{\sigma_j}(u_{j},v_{j},Q^\nu).
\end{align*}
Therefore to prove \eqref{est:jump} we now need to show that 
\begin{equation}\label{diag:seq}
\liminf_{j\to+\infty}E_{\sigma_j}(u_{j},v_{j},Q^\nu)\geq 1.
\end{equation}
For the sake of clarity in what follows we only consider the case $n=2$; the proof for $n>2$ can be obtained by means of analogous constructions and arguments. 

Upon possibly extracting a subsequence, we assume that the liminf in \eqref{diag:seq} is actually a limit. 

We now define suitable continuous counterparts of $u_j$ and $v_j$. To this end, set $\tau_j:=\delta_j/\rho_{{m_j,j}}$ and consider the triangulation $\T_j$ of $Q^\nu$ defined as follows: For every $i\in\tau_j\Z^2\cap Q^\nu$ set
\[T_i^+:=\conv\{i,i+\tau_j e_1,i+\tau_j e_2\}\quad\text{and}\quad T_i^-:=\conv\{i,i-\tau_j e_1,i-\tau_j e_2\},\]
and $\T_j:=\{T_i^+,T_i^-:i\in\tau_j\Z^2\cap Q^\nu\}$. Then we denote by $\tilde{u}_j$, $\tilde{v}_j$ the piecewise affine interpolations of $u_j$ and $v_j$ on $\T_j$, respectively. Moreover, we also consider the piecewise-constant function
\[\hat{v}_j(x):=v_j^i\quad\text{if}\ x\in \overset{\circ}{T_i^+}\cup T_i^-.\]
We clearly have $\tilde{u}_j\to u_0$, $\hat{v}_j\to 1$, and $\tilde{v}_j\to 1$  in $L^1(Q^\nu)$. 

Let $\eta>0$ be fixed; we now claim that 
\begin{equation}\label{est:ATn}
E_{\sigma_j}(u_j,v_j,Q^\nu)\geq\int_{Q^\nu_{1-\eta}}(\hat{v}_j)^2|\nabla\tilde{u}_j|^2dx+\frac{1}{2}\int_{Q^\nu_{1-\eta}}\frac{(\tilde{v}_j-1)^2}{\sigma_j}+\sigma_j|\nabla\tilde{v}_j|^2\,dx
\end{equation} 
for $j$ large. 

For $j$ sufficiently large there holds
\begin{equation}\label{est:surface1}
\int_{Q^\nu_{1-\eta}}\sigma_j|\nabla\tilde{v}_j|^2dx\leq\sum_{i\in\tau_j\Z^2\cap Q^\nu}\sum_{k=1}^2\sigma_j\tau_j^2\left|\frac{v_j^i-v_j^{i+\tau_j e_k}}{\tau_j}\right|^2.
\end{equation}
Moreover, for every $i\in\tau_j\Z^2\cap Q^\nu$ there holds
\[\tilde{v}_j(x)=\lambda_0(x)v_j^i+\lambda_1(x)v_j^{i+\tau_h e_1}+\lambda_2(x)v_j^{i+\tau_he_2}\quad\forall\ x\in T_i^+,\]
for some $\lambda_0(x),\lambda_1(x),\lambda_2(x)\in[0,1]$ satisfying $\lambda_0(x)+\lambda_1(x)+\lambda_2(x)=1$ for every $x\in T_i^+$ and
\[\int_{T_i^+}\lambda_k(x)dx=\frac{1}{3}\mathcal L^2(T_i^+)=\frac{1}{6}\tau_j^2, \quad  \text{for every}\quad k=0,1,2.\]
Then, the convexity of $z\to(z-1)^2$ yields
\begin{align*}
&\int_{T_i^+}\frac{(\tilde{v}_j-1)^2}{\sigma_j}dx=\frac{1}{\sigma_j}\int_{T_i^+}\left(\lambda_0(x)v_j^i+\lambda_1(x)v_j^{i+\tau_j e_1}+\lambda_2(x)v_j^{i+\tau_j e_2}-1\right)^2dx\\
&\leq\frac{1}{\sigma_j}\left((v_j^i-1)^2\int_{T_i^+}\lambda_0(x)dx+(v_j^{i+\tau_j e_1}-1)^2\int_{T_i^+}\lambda_1(x)dx+(v_j^{i+\tau_j e_2}-1)^2\int_{T_i^+}\lambda_2(x)dx\right)\\
&=\frac{1}{6}\tau_j^2\left(\frac{(v_j^i-1)^2}{\sigma_j}+\frac{(v_j^{i+\tau_j e_1}-1)^2}{\sigma_j}+\frac{(v_j^{i+\tau_j e_2}-1)^2}{\sigma_j}\right),
\end{align*}
and analogously for $T_i^-$. Therefore summing up on all triangles $T_i^+,T_i^-\in\T_j$ yields
\begin{equation}\label{est:surface2}
\int_{Q^\nu_{1-\eta}}\frac{(\tilde{v}_j-1)^2}{\sigma_j}dx\leq\sum_{i\in\tau_j\Z^2\cap Q^\nu}\tau_j^2\frac{(v_j^i-1)^2}{\sigma_j},
\end{equation}
for $j$ sufficiently large. 

Further, for every $i\in\tau_j\Z^2\cap Q$ there holds
\[\int_{T_i^+}(\hat{v}_j)^2|\nabla\tilde{u}_j|^2dx=\frac{\tau_j^2}{2}(v_j^i)^2\left(\left|\frac{u_j^i-u_j^{i+\tau_j e_1}}{\tau_j}\right|^2+\left|\frac{u_j^i-u_j^{i+\tau_j e_2}}{\tau_j}\right|^2\right)\]
and analogously on $T_i^-$ so that we may deduce
\begin{equation}\label{est:surface3}
\int_{Q^\nu_{1-\eta}}(\hat{v}_j)^2|\nabla\tilde{u}_j|^2dx\leq\frac{1}{2}\sum_{i\in \tau_j\Z^2\cap Q^\nu}\tau_j^2(v_j^i)^2\sum_{k=1}^2\left|\frac{u_j^i-u_j^{i\pm\tau_j e_k}}{\tau_j}\right|^2
\end{equation}
for $j$ sufficiently large. Finally, gathering \eqref{est:surface1}-\eqref{est:surface3} entails \eqref{est:ATn}. 

Now let $\Pi_\nu:=\{x\in\R^2 \colon \langle x, \nu \rangle=0\}$, $y\in Q^\nu\cap \Pi_\nu$ and set 
$$
\tilde{u}_j^{\nu,y}(t):=\tilde{u}_j(y+t\nu),\quad \hat{v}_j^{\nu,y}(t):=\hat{v}_j(y+t\nu).
$$ 
Clearly, $\tilde{u}_j^{\nu,y}\in W^{1,2}(\frac{-1+\eta}{2},\frac{1-\eta}{2})$ for $\HH^1$-a.e. $y\in Q^\nu_{1-\eta}\cap \Pi_\nu$, moreover
there holds
\begin{equation}\label{eq:n06}
\tilde{u}_h^{\nu,y}\to u_0^{\nu,y}\quad\text{in}\ L^1(-1/2,1/2)\quad \text{with}\quad S_{u_0^{\nu,y}}=\{0\},
\end{equation}
for $\HH^1$-a.e. $y\in Q^\nu\cap \Pi_\nu$. By Fubini's Theorem we have
\begin{align*}
\int_{Q^\nu_{1-\eta}}(\hat{v}_j)^2|\nabla\tilde{u}_j|^2dx &=\int_{Q^\nu_{1-\eta}\cap \Pi_\nu}\left(\int_{\frac{-1+\eta}{2}}^{\frac{1-\eta}{2}}(\hat{v}_j(y+t\nu))^2|\nabla\tilde{u}_j(y+t\nu)|^2dt\right)d\HH^1(y)\\
&\geq\int_{Q^\nu_{1-\eta}\cap \Pi_\nu}\left(\int_{\frac{-1+\eta}{2}}^{\frac{1-\eta}{2}}(\hat{v}_j^{\nu,y})^2((\tilde{u}_j^{\nu,y})')^2dt\right)d\HH^1(y),
\end{align*}
thus, from the bound on the energy we deduce the existence of a set $N\subset \Pi_\nu$ with $\HH^1(N)=0$ such that
\begin{equation}\label{est:n05}
\sup_j\int_{\frac{-1+\eta}{2}}^{\frac{1-\eta}{2}}(\hat{v}_j^{\nu,y})^2((\tilde{u}_j^{\nu,y})')^2dt<+\infty
\end{equation}
for every $y\in (Q^\nu_{1-\eta}\cap \Pi_\nu)\setminus N$. Further, it is not restrictive to assume that $\tilde{u}_j^{\nu,y}\in W^{1,2}(\frac{-1+\eta}{2},\frac{1-\eta}{2})$ for every $y\in (Q^\nu_{1-\eta}\cap \Pi_\nu)\setminus N$.
Therefore, in view of \eqref{eq:n06}-\eqref{est:n05}, appealing to a classical one-dimensional argument (see \textit{e.g.} the proof of \cite[Theorem 3.15]{braides98}) for every $y\in (Q^\nu_{1-\eta}\cap \Pi_\nu)\setminus N$ we can find a sequence $(s_j^y)\subset (\frac{-1+\eta}{2},\frac{1-\eta}{2})$ satisfying  
\begin{equation}\label{c:sj}
\hat{v}_j^{\nu,y}(s_j^y)\to 0\; \text{ as }\; j \to +\infty. 
\end{equation}

\noindent Now let $\tilde{v}_j^{\nu,y}$ be the one-dimensional slice of $\tilde v_j$ in the direction $\nu$\ie $\tilde{v}_j^{\nu,y}(t):=\tilde{v}_j(y+t\nu)$. For every $y\in  (Q^\nu_{1-\eta}\cap \Pi_\nu)\setminus N$ let  $s_j^y$ be as in \eqref{c:sj} and consider $\tilde v_j^{\nu,y}(s_j^y)$. Let moreover $d>0$ be fixed; we want to exhibit a set $N_j^d \subset \Pi_\nu$ with $\HH^1(N_j^d)\to 0$ as $j\to+\infty$ with the following property: for every $y\in  (Q^\nu_{1-\eta}\cap \Pi_\nu)\setminus (N\cup N^d_j)$ there exists $j_0:=j_0(d,y)\in \N$ satisfying
$$
\tilde{v}_j^{\nu,y}(s_j^y)\leq d \quad\text{for every $j\geq j_0$}.
$$
To this end, for every $i\in\tau_j\Z^2\cap Q^\nu$ set
\begin{equation}\label{def:mi}
M_j^i:=\max\{|v_j^i-v_j^l|:\ j\in\tau_j\Z^2\cap Q^\nu,\ |i-l|=\tau_j\};
\end{equation}
set moreover
\begin{equation}\label{def:Jhd}
\mathcal I_j^d:=\left\{i\in \tau_j\Z^2\cap Q^\nu \colon M_j^i\geq\frac{d}{2}\right\}.
\end{equation}
From the energy bound we deduce the existence of a constant $c>0$ such that
\[c\geq\sum_{i\in \mathcal I_j^d}\sum_{\begin{smallmatrix}j\in\tau_j\Z^2\cap Q^\nu\\|i-l|=\tau_j\end{smallmatrix}}\sigma_j\tau_j^2\bigg|\frac{v_j^i-v_j^l}{\tau_j}\bigg|^2\geq\sum_{i\in \mathcal I_j^d}\sigma_j(M_j^i)^2\geq\#(\mathcal I_j^d)\frac{d^2}{4}\sigma_j\]
for every $j$. Hence, there exists a constant $c(d)>0$ such that
\begin{equation}\label{est:Jh}
\#(\mathcal I_j^d)\leq \frac{c(d)}{\sigma_j}\quad\text{for every $j$}.
\end{equation}
Let $p^\nu:\R^n\to \Pi_\nu$ be the orthogonal projection onto the hyperplane $\Pi_\nu$ and set 
$$
N_j^d:=\bigcup_{i\in \mathcal I_j^d}p^\nu(\overset{\circ}{T_i^+}\cup T_i^-); 
$$
then in view of \eqref{est:Jh} we have
\begin{equation}\label{c:Njd}
\HH^1\big(N_j^d\big)\leq 2\sqrt{2}\tau_j\,\#(\mathcal I_j^d)\leq 2\sqrt{2}\,c(d)\,\frac{\tau_j}{\sigma_j}\to 0\quad\text{as}\ j\to+\infty,
\end{equation}
where the convergence to zero comes from the identity ${\tau_j}/{\sigma_j}={\delta_j}/{\e_j}$.

Now let $j\in \N$ be large, let $y\in (Q^\nu_{1-\eta}\cap \Pi_\nu)\setminus (N\cup N^d_j)$, and consider the corresponding $s^y_j$ as in 
\eqref{c:sj}. By definition of $\hat v_j$ we deduce the existence of $i_0:=i_0(y) \in (\tau_j\Z^2\cap Q^\nu)\setminus \mathcal I^d_j$ such that 
$y+s_j^y\nu\in T_{i_0}^+\cup T_{i_0}^-$ and $v_j^{i_0}\to 0$ as $j\to +\infty$. Therefore for every $d>0$ and every $y\in (Q^\nu_{1-\eta}\cap \Pi_\nu)\setminus (N\cup N^d_j)$ there exists $j_0:=j_0(d,y)\in \N$ such that $v_j^{i_0} < d/2$ for every $j\geq j_0$.  
Moreover, since $i_0\in (\tau_j\Z^2\cap Q^\nu)\setminus \mathcal I_h^d$ we also have
$$
v_h^{i_0\pm\tau_j e_k}<v_j^{i_0}+\frac{d}{2}<d\,,  \quad  \text{for $k=1,2$\; and for every $j\geq j_0$.}
$$
Therefore, since $\tilde{v}_j(y+s_j^y\nu)$ is a convex combination of either the triple $(v_j^{i_0},v_j^{i_0+\tau_j e_1},v_j^{i_0+\tau_j e_2})$ or of the triple $(v_j^{i_0},v_j^{i_0-\tau_j e_1},v_j^{i_0-\tau_j e_2})$ we finally get 
$$
\tilde{v}_j(y+s_j^y\nu)<d\,, \quad  \text{for every $j\geq j_0$.}
$$
Since on the other hand (up to a possible extraction) $\tilde{v}_j^{\nu,y}\to 1$ a.e., we can find $r_j^y,\tilde{r}_j^y\in (\frac{-1+\eta}{2},\frac{1-\eta}{2})$ such that $r_j^y<s_j^y<\tilde{r}_j^y$ and 
$$
\tilde{v}_j^{\nu,y}(r_j^y) > 1-d\,, \quad \tilde{v}_j^{\nu,y}(\tilde{r}_j^y) > 1-d
$$ 
for $j$ sufficiently large. 

Hence, for every fixed $y\in (Q^\nu_{1-\eta}\cap \Pi_\nu)\setminus (N\cup N^d_j)$ using the so-called ``Modica-Mortola trick'' as follows we get that
\begin{align}\label{est:n07}
\frac{1}{2}\int_{\frac{-1+\eta}{2}}^{\frac{1-\eta}{2}}\frac{(\tilde{v}_j^{\nu,y}-1)^2}{\sigma_j} +\sigma_j((\tilde{v}_j^{\nu,y})')^2dt 
&\geq\int_{r_j^y}^{s_j^y}(1-\tilde{v}_j^{\nu,y})|(\tilde{v}_j^{\nu,y})'|dt+\int_{s_j^y}^{\tilde{r}_j^y}(1-\tilde{v}_j^{\nu,y})|(\tilde{v}_j^{\nu,y})'|dt\nonumber\\
&\geq 2\int_d^{1-d}(1-z)dz=(1-2d)^2,
\end{align}
for every $j\geq j_0$. 
Moreover, by \eqref{c:Njd} we deduce that (up to subsequences) 
$$
\chi_{(Q^\nu_{1-\eta}\cap \Pi_\nu)\setminus (N \cup N_j^d)}\to 1 \quad \text{$\HH^{1}$-a.e. in $Q^\nu_{1-\eta}\cap \Pi_\nu$}
$$ 
so that
\begin{equation}\label{est:n06}
\liminf_{j \to +\infty}\left(\frac{1}{2}\int_{\frac{-1+\eta}{2}}^{\frac{1-\eta}{2}}\frac{(\tilde{v}_j^{\nu,y}-1)^2}{\sigma_j}+\sigma_h((\tilde{v}_j^{\nu,y})')^2dt\right)\chi_{(Q^\nu_{1-\eta}\cap \Pi_\nu)\setminus (N \cup N_j^d)}\geq(1-2d)^2
\end{equation}
for $\HH^{1}$-a.e. $y$ in $Q^\nu_{1-\eta}\cap \Pi_\nu$.
Thus, in view of \eqref{est:ATn}, the Fatou's Lemma together with \eqref{est:n06} give
\begin{align*}
\lim_{j\to+\infty} &E_{\sigma_j}(u_j,v_j,Q^\nu)\geq\liminf_{j\to+\infty}\int_{Q^\nu_{1-\eta}}\frac{(\tilde{v}_j-1)^2}{\sigma_j}+\sigma_j|\nabla\tilde{v}_j|^2dx\\
&\geq\int_{Q^\nu_{1-\eta}\cap \Pi^\nu}\liminf_{j\to+\infty}\left(\int_{\frac{-1+\eta}{2}}^{\frac{1-\eta}{2}}\frac{(\tilde{v}_j^{\nu,y}-1)^2}{\sigma_j}+\sigma_j((\tilde{v}_j^{\nu,y})')^2dt\right)\chi_{(Q^\nu_{1-\eta}\cap \Pi_\nu)\setminus (N \cup N_j^d)}d\HH^1(y)\\
&\geq (1-2d)^2\HH^1(Q^\nu_{1-\eta}\cap \Pi_\nu)=(1-2d)^2(1-\eta),
\end{align*}
so that we deduce \eqref{est:jump} by first letting $d\to 0$ and then $\eta\to 0$.

\medskip

\noindent {\bf Step 3:} extension to the case $u\in GSBV^2(\O)$.

\smallskip

Let $u\in GSBV^2(\Omega)$ and for $m\in\N$ let $u^m$ be the truncation of $u$ at level $m$. Clearly, $u^m\to u$ in $L^1(\Omega)$ as $m\to+\infty$. Therefore, since $E_0=M\!S$ is lower semicontinuous with respect to the strong $L^1(\Omega)$-topology and $E'_0(\cdot,v)$) decreases by truncations, we deduce 
\[E'_0(u,1)\geq\liminf_{m\to+\infty}E'_0(u^m,1)\geq\liminf_{m\to +\infty}M\!S(u^m,1)\geq M\!S(u,1)=E_0(u,1).\]
and hence the thesis.
\end{proof}
\subsection{Upper-bound inequality} In this subsection we prove the limsup inequality for $E_\e$ when $\ell=0$. To this end we start by recalling some well-known facts about the so-called optimal profile problem for the Ambrosio-Tortorelli functional.  
We define
\begin{equation}\label{min:prob}
\mathbf{m}:=\min\left\{\int_0^{+\infty}(f-1)^2+(f')^2dt:\ f\in W_{\rm loc}^{1,2}(0,+\infty),\ f(0)=0,\ \lim_{t\to+\infty}f(t)=1\right\},
\end{equation}
a straightforward computation shows that $\mathbf{m}=1$ and that this minimum value is attained at the function
$f(t)=1-e^{-t}$.

For our purposes it is also convenient to notice that $1=\mathbf{m}=\mathbf{\tilde{m}}$ where
\begin{equation}\label{def:op:reg}
\mathbf{\tilde{m}}:=\inf_{T>0}\inf\Bigg\{\int_0^{T}(f-1)^2+(f')^2dt\colon  f\in C^2([0,T]),\,
f(0)=0,\, f(T)=1, f'(T)=f''(T)=0 \Bigg\}.
\end{equation}
We are now ready to prove the following proposition.
\begin{prop}[Upper-bound for $\ell=0$]\label{prop:limsup}
Let $E_\e$ be as in \eqref{def:funct}  and $\ell=0$. Then for every $(u,v)\in L^1(\Omega)\times L^1(\Omega)$ there exists $(u_\e,v_\e)\subset L^1(\Omega)\times L^1(\Omega)$ with $(u_\e,v_\e) \to (u,v)$ in $L^1(\Omega)\times L^1(\Omega)$ such that
\[
\limsup_{\e \to 0} E_\e(u_\e,v_\e)\leq E_0(u,v).\]
\end{prop}
\begin{proof}
We can consider only those target functions $u\in GSBV^2(\Omega)$ and $v=1$ a.e.\ in $\Omega$, otherwise there is nothing to prove. 

By virtue of \cite[Theorem 3.9 and Corollary 3.11]{cortesani97}, using a standard density and diagonalisation argument it suffices to approximate those functions $u$ which belong to the space $\W(\Omega)$ defined as the space of all the $SBV(\Omega)$-functions satisfying the following conditions:
\begin{enumerate}
\item $S_{u}$ is essentially closed\ie $\HH^{n-1}(\overline{S_{u}}\setminus S_{u})=0$,
\item $\overline{S_{u}}$ is the intersection of $\Omega$ with the union of a finite number of pairwise disjoint closed and convex sets each contained in an $(n-1)$ dimensional hyperplane, and whose (relative) boundaries are $C^\infty$,
\item $u\in W^{l,\infty}(\Omega\setminus\overline{S_{u}})$ for all $l\in\N$.
\end{enumerate}
%
%
We prove the limsup inequality only in the case $\overline{S_u}=\Omega\cap K$ with $K\subset \Pi_\nu$, $K$ closed and convex, the proof in the general case being analogous.


Let $p^\nu$ denote as usual the orthogonal projection onto $\Pi_\nu$ and for $x\in \R^n$ set $d(x):=\dist(x,\Pi_\nu)$. For every $h>0$ define $K_h:=\{x\in\Pi_\nu \colon \dist(x,K)\leq h\}$. Let $\eta>0$ be fixed; by \eqref{def:op:reg} there exist $T_\eta>0$ and $f_\eta\in C^2([0,T_\eta])$ such that $f_\eta(0)=0$, $f_\eta(T_\eta)=1$, $f'_\eta(T_\eta)=f''_\eta(T_\eta)=0$, and
\[\int_0^{T_\eta}(f_\eta-1)^2+(f_\eta')^2dt\leq 1+\eta.\]
Clearly, up to setting $f_\eta(t)=1$ for every $t\geq T_\eta$ we can always assume that $f_\eta\in C^2([0,+\infty))$. 

Let $T>T_\eta$ and choose $\xi_\e>0$ such that $\xi_\e/\e\to 0$ as $\e\to 0$. We set
\begin{align*}
A_\e &:=\{x\in\R^n:\ p^\nu(x)\in K_{\e+\sqrt{n}\delta},\ d(x)\leq\xi_\e+\sqrt{n}\delta\},\\
B_\e &:=\{x\in\R^n:\ p^\nu(x)\in K_{2\e+\sqrt{n}\delta},\ d(x)\leq\xi_\e+\sqrt{n}\delta+\e T\},\\
C_\e &:=\{x\in\R^n:\ p^\nu(x)\in K_{\e/2},\ d(x)\leq\xi_\e/2\},\\
D_\e &:=\{x\in\R^n:\ p^\nu(x)\in K_\e,\ d(x)\leq\xi_\e\},
\end{align*}
and according to \eqref{def:mesh:int} we denote by $A_{\e,\delta}$, $B_{\e,\delta}$, $C_{\e,\delta}$, $D_{\e,\delta}$ the corresponding discretised sets. Let $\varphi_\e$ be a smooth cut-off function between $C_\e$ and $D_\e$ and set
\[u_\e(x):=u(x)(1-\varphi_\e(x)).\]
Since $u\in W^{l,\infty}(\Omega\setminus\overline{S_u})$ for every $l\in\N$, we can assume in particular that $u\in C^2(\Omega\setminus D_\e)$ so that $u_\e\in C^2(\Omega)$. We notice that $u_\e\to u$ in $L^1(\Omega)$ by the Lebesgue Dominated Convergence Theorem. Moreover, choose a smooth cut-off function $\gamma_\e$ between $K_{\e+\sqrt{n}\delta}$ and $K_{2\e+\sqrt{n}\delta}$ and define 
\[v_\e(x):=\gamma_\e(p^\nu(x))h_\e(d(x))+(1-\gamma_\e(p^\nu(x))),\]
where $h_\e:[0,+\infty)\to \R$ is given by
\begin{equation}\label{def:heps}
h_\e(t):=
\begin{cases}
0 &\text{if}\ t<\xi_\e+\sqrt{n}\delta,\\
f_\eta\left(\frac{t-\xi_\e-\sqrt{n}\delta}{\e}\right) &\text{if}\ \xi_\e+\sqrt{n}\delta\leq t<\xi_\e+\sqrt{n}\delta+\e T,\\
1 &\text{if}\ t\geq\xi_\e+\sqrt{n}\delta+\e T.
\end{cases}
\end{equation}
By construction $v_\e\in W^{1,\infty}(\Omega)\cap C^0(\Omega)\cap C^2(\Omega\setminus A_\e)$ and $v_\e\to 1$ in $L^1(\Omega)$. We then define the recovery sequence $(\bar u_{\e}, \bar v_{\e})\subset\A_\e(\Omega)\times\A_\e(\Omega)$ by setting
\[\bar u_{\e}^i:=u_\e(i),\quad \bar v_\e^i:=0\vee (v_\e(i)\wedge 1)\quad\text{for every}\quad i\in\Omega_\delta.\]
We clearly have $\bar v_{\e}\to 1$ in $L^1(\Omega)$ as $\e\to 0$. Moreover,
\begin{align}\nonumber
\|\bar u_{\e}-u_\e\|_{L^1(\Omega)} &\leq\sum_{i\in\Omega_\delta}\int_{i+[0,\delta)^n}|u_\e(i)-u_\e(x)|dx\\\nonumber
&=\hspace*{-1em}\sum_{\begin{smallmatrix}i\in\Omega_\delta\\i+[0,\delta)^n\subset(\Omega\setminus D_\e)\end{smallmatrix}}\hspace*{-2em}\int_{i+[0,\delta)^n}|u(i)-u(x)|dx+\hspace*{-1em}\sum_{\begin{smallmatrix}i\in\Omega_\delta\\i+[0,\delta)^n\cap D_\e\neq\emptyset\end{smallmatrix}}\hspace*{-2em}\int_{i+[0,\delta)^n}|u_\e(i)-u_\e(x)|dx\\\label{c:conv-rec}
&\leq c\left(\|\nabla u\|_{L^\infty}\delta+\|u\|_{L^\infty}\xi_\e\right),
\end{align}
where to estimate the first term in the second line we have used the mean-value Theorem while to estimate the second term we used the fact that 
$$
\#(\{i\in\Omega_\delta:\ i+[0,\delta)^n\cap D_\e\neq\emptyset\})=\mathcal{O}\Big(\frac{\xi_\e}{\delta^n}\Big).
$$ 
Hence, the convergence $\bar u_{\e}\to u$ in $L^1(\Omega)$ follows from \eqref{c:conv-rec} and the fact that $u_\e\to u$ in $L^1(\Omega)$.

Thus it remains to prove that
\[\limsup_{\e\to 0}E_\e(\bar u_{\e},\bar v_{\e})\leq \int_\Omega|\nabla u|^2dx+\HH^{n-1}(S_u\cap\Omega).\]
We clearly have
\begin{align}\label{est:n08}
\limsup_{\e\to 0}E_\e(\bar u_{\e},\bar v_{\e}) &\leq\limsup_{\e\to 0}F_\e(\bar u_{\e},\bar v_{\e})+\limsup_{\e\to 0}G_\e(\bar v_{\e}).
\end{align}
We now estimate the two terms in the right-hand-side of \eqref{est:n08} separately. We start with $F_\e(\bar u_{\e},\bar v_{\e})$. 
Since $\bar v_\e^i=v_\e(i)=0$ for all $i\in A_{\e,\delta}$, we have
\begin{equation}\label{eq:n09}
F_\e(\bar u_{\e},\bar v_{\e})=\frac{1}{2}\sum_{i\in\Omega_\delta\setminus A_\e}\delta^n (v_\e^i)^2\sum_{\begin{smallmatrix}k=1\\i\pm\delta e_k\in\Omega_\delta\end{smallmatrix}}^n\left|\frac{u_\e(i)-u_\e({i\pm\delta e_k})}{\delta}\right|^2.
\end{equation}
Let $i\in\Omega_\delta\setminus A_\e$. By construction $i\pm\delta e_k\in\Omega_\delta\setminus D_\e$ for every $k\in \{1,\ldots, n\}$ and since $u_\e=u$ on $\Omega\setminus D_\e$, using Jensen's inequality we deduce that
\begin{equation*}
\left|\frac{u_\e({i\pm\delta e_k})-u_\e(i)}{\delta}\right|^2=\left|\frac{1}{\delta}\int_0^\delta\left<\nabla u(i\pm te_k),e_k\right>dt\right|^2\leq\frac{1}{\delta}\int_0^\delta|\left<\nabla u(i\pm t e_k),e_k\right>|^2dt,
\end{equation*}
for every $k\in \{1,\ldots,n\}$. Therefore, thanks to the regularity of $u$ and to the fact that $0\leq \bar v_{\e}\leq 1$, the mean-value Theorem gives
\begin{align}\label{eq:n10}
F_\e(\bar u_{\e},\bar v_{\e}) &\leq\sum_{i\in\Omega_\delta\setminus A_\e}\sum_{k=1}^n\int_{i+[0,\delta)^{n-1}}\left(\int_0^\delta|\left<\nabla u(i\pm t e_k),e_k\right>|^2dt\right)d\HH^{n-1}(y)\nonumber\\
&=\sum_{i\in\Omega_\delta\setminus A_\e}\sum_{k=1}^n\int_{i+[0,\delta)^{n-1}}\left(\int_0^\delta|\left<\nabla u(y\pm t e_k),e_k\right>|^2dt\right)d\HH^{n-1}(y)+\mathcal{O}(\delta)\nonumber\\
&\leq\int_{\Omega}|\nabla u(x)|^2dx+\mathcal{O}(\delta),
\end{align}
so that
\begin{equation}\label{est:n09}
\limsup_{\e\to 0}F_\e(\bar u_{\e},\bar v_{\e})\leq\int_\Omega |\nabla u|^2dx.
\end{equation}
We now turn to estimate the term $G_\e(\bar v_{\e})$. We have
\begin{align}\label{eq:n10b}
G_\e(\bar v_\e) \leq G(v_\e) &=\frac{1}{2}\sum_{i\in\Omega_\delta\setminus B_\e}\delta^n\left(\frac{(v_\e^i-1)^2}{\e}+\e\sum_{k=1}^n\left|\frac{v_\e^i-v_\e^{i+\delta e_k}}{\delta}\right|^2\right)\nonumber\\
&+\frac{1}{2}\sum_{i\in B_{\e,\delta}\setminus A_\e}\delta^n\left(\frac{(v_\e^i-1)^2}{\e}+\e\sum_{k=1}^n\left|\frac{v_\e^i-v_\e^{i+\delta e_k}}{\delta}\right|^2\right)\nonumber\\
&+\frac{1}{2}\sum_{i\in A_{\e,\delta}}\delta^n\left(\frac{(v_\e^i-1)^2}{\e}+\e\sum_{k=1}^n\left|\frac{v_\e^i-v_\e^{i+\delta e_k}}{\delta}\right|^2\right).
\end{align}
We start noticing that
\begin{equation}\label{eq:n09a}
\sum_{i\in\Omega_\delta\setminus B_\e}\delta^n\left(\frac{(v_\e^i-1)^2}{\e}+\e\sum_{k=1}^n\left|\frac{v_\e^i-v_\e^{i+\delta e_k}}{\delta}\right|^2\right)=0.
\end{equation}
Indeed, for $i\in\Omega_\delta\setminus B_\e$ we have $d(i)\geq\xi_\e+\sqrt{n}\delta+\e T$. Then, since $T>T_\eta$, for $\e$ sufficiently small we deduce that $d(i+\delta e_k)\geq d(i)-\delta\geq\xi_\e+\sqrt{n}\delta+\e T_\eta$ for every $k\in\{1,\ldots, n\}$ so that by definition of $h_\e$ we get $v_\e^i=v_\e^{i+\delta e_k}=1$, hence \eqref{eq:n09a}.

We now claim that
\begin{equation}\label{est:n10}
\sum_{i\in A_{\e,\delta}}\delta^n\left(\frac{(v_\e^i-1)^2}{\e}+\e\sum_{k=1}^n\left|\frac{v_\e^i-v_\e^{i+\delta e_k}}{\delta}\right|^2\right)\leq c\,\frac{\xi_\e+\delta}{\e}\HH^{n-1}(K_{\e+\sqrt{n}\delta})\;\to 0,
\end{equation}
as $\e\to 0$. To prove the claim we observe that $v_\e^i=v_\e^{i+\delta e_k}=0$ for every $i\in A_{\e,\delta}$ and every $k\in\{1,\ldots,n\}$ such that $i+\delta e_k\in A_{\e,\delta}$. Hence we have
\begin{align}\nonumber
& \sum_{i\in A_{\e,\delta}}\delta^n\bigg(\frac{(v_\e^i-1)^2}{\e}+\e\sum_{\begin{smallmatrix}k=1\\i+\delta e_k\in A_{\e,\delta}\end{smallmatrix}}^n\left|\frac{v_\e^i-v_\e^{i+\delta e_k}}{\delta}\right|^2\bigg)
\\\nonumber
& =\sum_{i\in A_{\e,\delta}}\delta^n\bigg(\frac{1}{\e}\bigg) = \#(A_{\e,\delta})\,\frac{\delta^n}{\e}
\\\label{c:invece-serve}
&\leq c\,\frac{\xi_\e+\delta}{\e}\HH^{n-1}(K_{\e+\sqrt{n}\delta}).  
\end{align}
Then, it only remains to estimate the energy for those $i\in A_{\e,\delta}$ and $k\in\{1,\ldots,n\}$ such that $i+\delta e_k\in B_{\e,\delta}\setminus A_{\e,\delta}$\ie to estimate the term
\begin{equation*}
\sum_{i\in A_{\e,\delta}}\delta^n\sum_{\begin{smallmatrix}k=1\\i+\delta e_k\in B_{\e,\delta}\setminus A_{\e,\delta}\end{smallmatrix}}^n\e\left|\frac{v_\e^i-v_\e^{i+\delta e_k}}{\delta}\right|^2.
\end{equation*}
To this end, we observe that in general, for every $i\in\overset{\circ}{\Omega}_\delta$ and every $k\in\{1,\ldots,n\}$, thanks to the regularity of $v_\e$, by Jensen's inequality we have
\begin{equation}\label{est:n11}
\left|\frac{v_\e^i-v_\e^{i+\delta e_k}}{\delta}\right|^2\leq\left|\frac{v_\e(i)-v_\e(i+\delta e_k)}{\delta}\right|^2\leq\frac{1}{\delta}\int_0^\delta|\left<\nabla v_\e(i+te_k),e_k\right>|^2dt.
\end{equation}
Since
\[\nabla v_\e(x)=\left(h_\e(d(x))-1\right)\nabla\gamma_\e(p^\nu(x))D_p^\nu(x)\pm\gamma_\e(p^\nu(x))h_\e'(d(x))\nu,\]
where $D_p^\nu(x)$ denotes the Jacobian of $p^\nu$ evaluated at $x$, using that $h_\e\in W^{1,\infty}(0,+\infty)$ satisfies $\|h_\e'\|_{L^\infty}\leq\frac{C}{\e}$, while $\|\nabla\gamma_\e\|_{L^\infty}\leq\frac{C}{\e}$ and $\|D_p^\nu\|_{L^\infty}\leq 1$, from \eqref{est:n11} we obtain
\begin{equation}\label{eq:n12a}
\e\left|\frac{v_\e^i-v_\e^{i+\delta e_k}}{\delta}\right|^2=\mathcal{O}\left(\frac{1}{\e}\right)\quad\text{for every}\; i\in\overset{\circ}{\Omega}_\delta,\quad\text{for every}\; k\in \{1,\ldots, n\}
\end{equation}
thus, consequently 
\[\sum_{i\in A_{\e,\delta}}\delta^n\sum_{\begin{smallmatrix}k=1\\i+\delta e_k\in B_{\e,\delta}\setminus A_{\e,\delta}\end{smallmatrix}}^n\e\left|\frac{v_\e^i-v_\e^{i+\delta e_k}}{\delta}\right|^2\leq c\#(\partial A_{\e,\delta})\frac{\delta^n}{\e}\leq c\frac{\delta}{\e}\HH^{n-1}\left(K_{\e+\sqrt{n}\delta}\right),\]
which together with \eqref{c:invece-serve} gives \eqref{est:n10}.

We finally come to estimate
\[\frac{1}{2}\sum_{i\in B_{\e,\delta}\setminus A_{\e,\delta}}\delta^n\left(\frac{(v_\e^i-1)^2}{\e}+\e\sum_{k=1}^n\left|\frac{v_\e^i-v_\e^{i+\delta e_k}}{\delta}\right|^2\right).\]
To this end, it is convenient to write $B_\e\setminus A_\e$ as the union of the pairwise disjoint sets $M_\e$, $V_\e$, $W_\e$ defined as:
\begin{align*}
M_\e &:=\{x\in\R^n:\ p^\nu(x)\in K_{\e+\sqrt{n}\delta},\ \xi_\e+\sqrt{n}\delta< d(x)\leq\xi_\e+\sqrt{n}\delta+\e T\},\\
V_\e &:=\{x\in\R^n:\ p^\nu(x)\in K_{2\e+\sqrt{n}\delta}\setminus K_{\e+\sqrt{n}\delta},\ d(x)\leq\xi_\e+\sqrt{n}\delta\},\\
W_\e &:=\{x\in\R^n:\ p^\nu(x)\in K_{2\e+\sqrt{n}\delta}\setminus K_{\e+\sqrt{n}\delta},\ \xi_\e+\sqrt{n}\delta< d(x)\leq\xi_\e+\sqrt{n}\delta+\e T\}. 
\end{align*}
Further, we denote with $M_{\e,\delta}$, $V_{\e,\delta}$, $W_{\e,\delta}$ their discrete counterparts as in \eqref{def:mesh:int}. We now estimate the energy along the recovery sequence in the three sets as above, separately. 
%
To this end we start noticing that $\#(V_{\e,\delta})=\mathcal{O}\left(\frac{\e\xi_\e}{\delta^n}\right)$ and $\#(W_{\e,\delta})=\mathcal{O}\left(\frac{\e^2}{\delta^n}\right)$. Thus, appealing to \eqref{eq:n12a} we deduce that
\[\frac{1}{2}\sum_{i\in V_{\e,\delta}}\left(\delta^n\frac{(v_\e^i-1)^2}{\e}+\sum_{k=1}^n\e\left|\frac{v_\e^i-v_\e^{i+\delta e_k}}{\delta}\right|^2\right)=\mathcal{O}(\xi_\e)\]
and
\[\frac{1}{2}\sum_{i\in W_{\e,\delta}}\left(\delta^n\frac{(v_\e^i-1)^2}{\e}+\sum_{k=1}^n\e\left|\frac{v_\e^i-v_\e^{i+\delta e_k}}{\delta}\right|^2\right)=\mathcal{O}(\e).\]
Finally, again using \eqref{eq:n12a} we get
\begin{align*}
\sum_{i\in M_{\e,\delta}}\left(\delta^n\frac{(v_\e^i-1)^2}{\e}+\sum_{k=1}^n\e\left|\frac{v_\e^i-v_\e^{i+\delta e_k}}{\delta}\right|^2\right) &\leq\sum_{i\in\overset{\circ}{M}_{\e,\delta}}\left(\delta^n\frac{(v_\e^i-1)^2}{\e}+\sum_{k=1}^n\e\left|\frac{v_\e^i-v_\e^{i+\delta e_k}}{\delta}\right|^2\right)\\
&+c\frac{\delta^n}{\varepsilon}\#(\partial M_{\e,\delta});
\end{align*}
moreover we notice that
$$
\#(\partial M_{\e,\delta})\leq \frac{c}{\delta^{n-1}}\left(\HH^{n-1}(K_{\e+\sqrt{n}\delta})+\e\right).
$$
Let now $i\in\overset{\circ}{M}_{\e,\delta}$; then $i+[0,\delta)^n\subset M_\e$. Hence, by definition of $v_\e$ we have
\[v_\e(x)=f_\eta\left(\frac{d(x)-\xi_\e-\sqrt{n}\delta}{\e}\right)\quad\text{and}\quad\nabla v_\e(x)=\pm\frac{1}{\e} f_\eta'\left(\frac{d(x)-\xi_\e-\sqrt{n}\delta}{\e}\right)\nu\quad\text{in}\ i+[0,\delta)^n.\]
Since $f_\eta\in C^2([0,T])$, appealing to the Mean-Value Theorem we deduce that for every $x\in  i +[0,\delta)^{n}$
$$
\Big|f_\eta\Big(\frac{d(x)-\xi_\e-\sqrt{n}\delta}{\e}\Big)-f_\eta\Big(\frac{d(i)-\xi_\e-\sqrt{n}\delta}{\e}\Big)\Big| \leq \frac{\delta}{\e}\, \|f'_\eta\|_{L^\infty(0,T)},
$$
while for every $y\in i +[0,\delta)^{n-1}$, every $t\in (0,\delta)$, and every $k\in\{1,\ldots,n\}$ 
$$
\bigg|\Big\langle\bigg(f'_\eta\Big(\frac{d(y+te_k)-\xi_\e-\sqrt{n}\delta}{\e}\Big)-f'_\eta\Big(\frac{d(i+t e_k)-\xi_\e-\sqrt{n}\delta}{\e}\Big)\bigg)\nu,e_k\Big\rangle\bigg| \leq \frac{\delta}{\e}\, \|f''_\eta\|_{L^\infty(0,T)}.
$$
So that it follows 
\begin{align*}
&\frac{1}{2}\sum_{i\in M_{\e,\delta}}\delta^n\Bigg(\frac{(v_\e^i-1)^2}{\e}+\sum_{k=1}^n\e\left|\frac{v_\e^i-v_\e^{i+\delta e_k}}{\delta}\right|^2\Bigg)\\
&\leq\frac{1}{2}\sum_{i\in\overset{\circ}{M}_{\e,\delta}}\Bigg(\int_{i+[0,\delta)^n} \frac{1}{\e}\left(f_\eta\left(\frac{d(x)-\xi_\e-\sqrt{n}\delta}{\e}\right)-1\right)^2dx\\
&+\sum_{k=1}^n\frac{1}{\e}\int_{i+[0,\delta)^{n-1}}\hspace*{-0.5em}\bigg(\int_0^\delta\left|\left<f_\eta'\Big(\frac{d(y+te_k)-\xi_\e-\sqrt{n}\delta}{\e}\Big)\nu,e_k\right>\right|^2dt\bigg)d\HH^{n-1}(y)\Bigg)\\
&+c\frac{\delta}{\e}\left(\HH^{n-1}(K_{\e+\sqrt{n}\delta})+\e\right)\\
&\leq\frac{1}{2}\int_{M_\e} \frac{1}{\e}\bigg(\Big(f_\eta\Big(\frac{d(x)-\xi_\e-\sqrt{n}\delta}{\e}\Big)-1\bigg)^2\hspace*{-0.5em}+\left|f_\eta'\left(\frac{d(x)-\xi_\e-\sqrt{n}\delta}{\e}\right)\nu\right|^2\bigg)dx
\\
&+c\frac{\delta}{\e}\left(\HH^{n-1}(K_{\e+\sqrt{n}\delta})+\e\right)\\
&=\int_{K_{\e+\sqrt{n}\delta}}\Bigg(\frac{1}{\e}\int_{\xi_\e+\sqrt{n}\delta}^{\xi_\e+\sqrt{n}\delta+\e T}\left(f_\eta\left(\frac{t-\xi_\e-\sqrt{n}\delta}{\e}\right)-1\right)^2\\
&+\left(f_\eta'\left(\frac{t-\xi_\e-\sqrt{n}\delta}{\e}\right)\right)^2dt\Bigg)d\HH^{n-1}(y)+c\frac{\delta}{\e}\left(\HH^{n-1}(K_{\e+\sqrt{n}\delta})+\e\right)\\
&=\int_{K_{\e+\sqrt{n}\delta}}\Bigg(\int_0^T(f_\eta(t)-1)^2+(f_\eta'(t))^2dt\Bigg)d\HH^{n-1}(y)+c\frac{\delta}{\e}\left(\HH^{n-1}(K_{\e+\sqrt{n}\delta})+\e\right)\\
&\leq \left(1+\eta+c\frac{\delta}{\e}\right)\HH^{n-1}(K_{\e+\sqrt{n}\delta})+c\delta.
\end{align*}
Thus, gathering \eqref{est:n08}, \eqref{est:n09}-\eqref{est:n10} gives
\[\limsup_{\e \to 0} E_\e(\bar u_{\e},\bar v_{\e})\leq\int_\Omega|\nabla u|^2dx+(1+\eta)\HH^{n-1}(S_u\cap\Omega).\]
and hence the thesis.
\end{proof}

\begin{rem}\label{upperbound:local}
{\rm For later use we notice that the upper-bound inequality in Proposition \ref{prop:limsup} can be easily generalised in the following way: let $\ell\in(0,+\infty)$ then there exists $\alpha_\ell \in (0,+\infty)$ such that 
%
%
%
\[E_\ell''(u,1,U)\leq\int_U|\nabla u|^2dx+(1+\alpha_\ell)\HH^{n-1}(S_u\cap U),\]
for every $u\in GSBV^2(\Omega)$ and for every $U\in\mathscr A_L(\Omega)$.
}
\end{rem}

\section{Proof of the $\G$-convergence result in the critical regime $\ell \in (0,+\infty)$}\label{sec:critical} 

In this section we study the $\Gamma$-limit of the functionals $E_\e$ in the case where $\varepsilon$ and $\delta$ are of the same order.
We show that in this case the presence of the underlying lattice affects the $\Gamma$-limit, which in particular turns out
to be anisotropic. 

The treatment of this case presents some additional difficulties with respect to the subcritical case investigated in Section \ref{s:subcritical}. In particular, when the space-dimension $n$ is larger than two we can only show that the $\Gamma$-convergence of $E_\e$ takes place up to a subsequence and that the $\Gamma$-limit can be represented as a free-discontinuity functional of the form    
\begin{equation}\label{int:rep01}
E_\ell(u)=\int_\Omega|\nabla u|^2dx+\int_{S_u\cap \O}\phi_\ell([u],\nu_u)d\HH^{n-1},
\end{equation} 
for some $\phi_\ell$ which is not explicit and possibly depends on the amplitude of the jump $[u]$. We observe that this dependence cannot be excluded \emph{a priori} because in general we cannot exclude the possibility for the term $F_\e$ to enter in the definition of $\phi_\ell$ (as opposite to the case $\ell=0$ where the surface energy is only determined by $G_\e$). 
When $n=2$, however, we are able to characterise $\phi_\ell$ showing, in particular, that it does not depend on $[u]$ and on the subsequence so that in this case the $\Gamma$-convergence of the functionals $E_\e$ takes place for the whole sequence.

\medskip

The convergence result for $E_\e$ in this critical regime will be achieved by means of the so-called localisation method of $\Gamma$-convergence (see \emph{e.g.} \cite[Chapters 14-20]{DalMaso93}, \cite[Chapters 16]{Braides02}). The latter consists of two main steps: proving the existence of a subsequence $(E_{\varepsilon_j})$ $\Gamma$-converging to some abstract functional $E$ when localised to all open subsets of $\Omega$ and, subsequently, showing that $E$ can be represented in an integral form as in \eqref{int:rep01}. 

In order to apply the localisation method we need to consider the functionals $E_\varepsilon$ as functions defined on triples $(u,v,U)\in L^1(\Omega)\times L^1(\Omega)\times\mathscr A(\Omega)$ using the localised definition of the energies as in \eqref{loc:funct}. Moreover, for every $\varepsilon_j\to 0$ we also need to consider $E'_\ell,E''_\ell:L^1(\Omega)\times L^1(\Omega)\times\mathscr A(\Omega)\to [0,+\infty]$ the localised versions of the $\Gamma$-liminf and $\Gamma$-limsup functionals\ie
\begin{align*}
E'_\ell(u,v,U):=\Gamma\hbox{-}\liminf_{j\to+\infty}E_{\varepsilon_j}(u,v,U),\quad E''_\ell(u,v,U):=\Gamma\hbox{-}\limsup_{j\to +\infty}E_{\varepsilon_j}(u,v,U).
\end{align*}
Finally, we define the inner regular envelopes of $E'_\ell$ and $E''_\ell$, respectively\ie
\begin{align*}
(E'_\ell)_{-}(\cdot,\cdot,U)&:=\sup\{E'_\ell(\cdot,\cdot,V):\ V\subset\subset U,\ V\in\mathscr A(\Omega)\},\\
(E''_\ell)_{-}(\cdot,\cdot,U)&:=\sup\{E''_\ell(\cdot,\cdot,V):\ V\subset\subset U,\ V\in\mathscr A(\Omega)\}.
\end{align*}
\begin{rem}\label{rem:crit01}
{\rm
The functionals $E'_\ell$ and $E''_\ell$ are increasing \cite[Proposition 6.7]{DalMaso93}, lower semicontinuous \cite[Proposition 6.8]{DalMaso93}, and local \cite[Proposition 16.15]{DalMaso93}. Note that they both decrease by truncation. Moreover, $E'_\ell$ is also superadditive \cite[Proposition 16.12]{DalMaso93}. 
%
%
Further, the functionals $(E'_\ell)_{-}$ and $(E''_\ell)_{-}$ are inner regular by definition, increasing, lower semicontinuous \cite[Remark 15.10]{DalMaso93}, local \cite[Remark 15.25]{DalMaso93}, and $(E'_\ell)_{-}$ is superadditive \cite[Proposition 16.12]{DalMaso93}.}
\end{rem}
The next proposition shows that the functionals $E_\e$ satisfy the so-called fundamental estimate, uniformly in $\e$. 
\begin{prop}[Fundamental estimate]\label{fund:est}
For every $\varepsilon>0$, $\eta>0$, $U,U',V\in\mathscr A(\Omega)$ with $U\subset\subset U'$,  and for every $(u,v),(\tilde{u},\tilde{v})\in L^1(\Omega)\times L^1(\Omega)$, with $0\leq v,\tilde{v}\leq 1$ there exists $(\hat{u},\hat{v})\in L^1(\Omega)\times L^1(\Omega)$ such that
\[E_\varepsilon(\hat{u},\hat{v},U\cup V)\leq(1+\eta)\left(E_\varepsilon(u,v,U')+E_\varepsilon(\tilde{u},\tilde{v},V)\right)+\sigma_\varepsilon(u,v,\tilde{u},\tilde{v},U,U',V),\]
where $\sigma_\varepsilon:L^1(\Omega)^4\times\mathscr A(\Omega)^3\to [0,+\infty)$ depends only on $\varepsilon$ and $E_\varepsilon$ and is such that
\begin{equation}\label{cond:sigma}
\lim_{\varepsilon\to 0}\sigma_\varepsilon(u_\varepsilon,v_\varepsilon,\tilde{u}_\varepsilon,\tilde{v}_\varepsilon,U,U',V)=0,
\end{equation}
for all $U,U',V\in\mathscr A(\Omega)$ with $U\subset\subset U'$ and every $(u_\varepsilon,v_\varepsilon)$ and $(\tilde{u}_\varepsilon,\tilde{v}_\varepsilon)$ which have the same limit as $\varepsilon\to 0$ in $L^2((U'\setminus U)\cap V)\times L^2((U'\setminus U)\cap V)$ and satisfy
\[\sup_\e\left(E_\e(u_\e,v_\e,U')+E_\e(\tilde{u}_\e,\tilde{v}_\e,V)\right)<+\infty.\]
\end{prop}
\begin{proof}
Fix $\e>0$ and $\eta>0$ and let $U,U',V\in\mathscr A(\Omega)$ with $U\subset\subset U'$. Let $N\in\N$ and choose $U_1,\ldots, U_{N+1}\in\mathscr A(\Omega)$ such that
\[U_0:=U\subset\subset U_1\subset\subset\ldots\subset\subset U_{N+1}\subset\subset U'.\]
For every $ l\in \{2,\ldots, N\}$ let $\varphi_l$ be a cut-off function between $U_{l-1}$ and $U_l$ and set
\[M:=\max_{2\leq l\leq N}\|\nabla \varphi_l\|_{L^\infty}.\]
Let $(u,v)$ and $(\tilde{u},\tilde{v})$ be as in the statement and define $w\in\A_\e(\Omega)$ by setting
\[w^i:=\min\{v^i,\tilde{v}^i\},\]
for every $i\in\Omega\cap\delta\Z^n$. Notice that by definition of $w^i$ we have that for every $i\in\Omega\cap\delta\Z^n$
\begin{equation}\label{est:w1}
(w^i-1)^2\leq(v^i-1)^2+(\tilde{v}^i-1)^2
\end{equation} 
and
\begin{equation}\label{est:w2}
\left|\frac{w^i-w^{i+\delta e_k}}{\delta}\right|^2\leq\left|\frac{v^i-v^{i+\delta e_k}}{\delta}\right|^2+\left|\frac{\tilde{v}^i-\tilde{v}^{i+\delta e_k}}{\delta}\right|^2\quad\text{for every}\ k\in\{1,\ldots, n\}.
\end{equation}
For every $l\in\{4,\ldots, N-2\}$ we define the functions
\[\hat{u}_l^i:=\varphi_l(i)u^i+(1-\varphi_l(i))\tilde{u}^i,\quad i\in\Omega\cap\delta\Z^n,\]
and
\[\hat{v}_l^i:=
\begin{cases}
\varphi_{l-2}(i)v^i+(1-\varphi_{l-2}(i))w^i &\text{if}\ i\in U_{l-2}\cap\delta\Z^n,\\
w^i &\text{if}\ i\in (U_{l+1}\setminus\overline{U}_{l-2})\cap\delta\Z^n,\\
\varphi_{l+2}(i)w^i+(1-\varphi_{l+2}(i))\tilde{v}^i &\text{if}\ i\in(\Omega\setminus\overline{U}_{l+1})\cap\delta\Z^n.
\end{cases}\]
Then we have
\begin{align}\label{fund:est:01}
E_\e(\hat{u}_l,\hat{v}_l,U\cup V) &\leq\tilde{E}_\e(u,v,U_{l-4})+ \tilde{E}_\e(u,\hat{v}_l,(U_{l-3}\setminus\overline{U}_{l-4})\cap V)+\tilde{E}_\e(u,\hat{v}_l,(U_{l-2}\setminus\overline{U}_{l-3})\cap V)\nonumber\\
&+\tilde{E}_\e(\hat{u}_l,\hat{v}_l,(U_{l+1}\setminus\overline{U}_{l-2})\cap V)+\tilde{E}_\e(\tilde{u},\hat{v}_l,(U_{l+2}\setminus\overline{U}_{l+1})\cap V)\nonumber\\
&+\tilde{E}_\e(\tilde{u},\hat{v}_l(U_{l+3}\setminus\overline{U}_{l+2})\cap V)+\tilde{E}_\e(\tilde{u},\tilde{v},V\setminus\overline{U}_{l+3}),
\end{align}
where to shorten notation for every $(u,v)\in\A_\e(\Omega)\times\A_\e(\Omega)$ and $U\in\A(\Omega)$ we set
\[\tilde{E}_\e(u,v,U):=\frac{1}{2}\sum_{i\in U_\delta}\delta^n\bigg((v^i)^2\hspace*{-1.5em}\sum_{\begin{smallmatrix}k=1\\i\pm\delta e_k\in U\cup V\end{smallmatrix}}^n\hspace*{-1em}\left|\frac{u^i-u^{i\pm\delta e_k}}{\delta}\right|^2+\frac{(v^i-1)^2}{\e}+\e\hspace*{-1.5em}\sum_{\begin{smallmatrix}k=1\\i+\delta e_k\in U\cup V\end{smallmatrix}}^n\hspace*{-1em}\left|\frac{v^i-v^{i+\delta e_k}}{\delta}\right|^2\bigg).\]
Analogously, we set
\[\tilde{F}_\e(u,v,U):=\frac{1}{2}\sum_{i\in U_\delta}\delta^n(v^i)^2\hspace*{-1.5em}\sum_{\begin{smallmatrix}k=1\\i\pm\delta e_k\in U\cup V\end{smallmatrix}}^n\hspace*{-1em}\left|\frac{u^i-u^{i\pm\delta e_k}}{\delta}\right|^2\]
and
\[\tilde{G}_\e(v,U):=\frac{1}{2}\sum_{i\in U_\delta}\delta^n\bigg(\frac{(v^i-1)^2}{\e}+\e\hspace*{-1.5em}\sum_{\begin{smallmatrix}k=1\\i+\delta e_k\in U\cup V\end{smallmatrix}}^n\hspace*{-1em}\left|\frac{v^i-v^{i+\delta e_k}}{\delta}\right|^2\bigg).\]
We now estimate the different terms on the right hand side of \eqref{fund:est:01} separately. Taking into account the definition of $\hat{v}_l$ we have
\begin{align*}
\tilde{E}_\e(u,\hat{v}_l,(U_{l-3}\setminus\overline{U}_{l-4})\cap V) &=\tilde{F}_\e(u,v,(U_{l-3}\setminus\overline{U}_{l-4})\cap V)+G_\e(v,(U_{l-3}\setminus\overline{U}_{l-4})\cap V)\\
&+\hspace*{-2em}\sum_{i\in (U_{l-3}\setminus\overline{U}_{l-4})\cap V_\delta}\hspace*{-2em}\delta^n\e\sum_{\begin{smallmatrix}k=1\\i+\delta e_k\in(U_{l-2}\setminus\overline{U}_{l-3})\cap V\end{smallmatrix}}^n\hspace*{-2.5em}\left|\frac{\hat{v}_l^i-\hat{v}_l^{i+\delta e_k}}{\delta}\right|^2.
\end{align*}
On $U_{l-2}$ we have
\begin{align}\label{fund:est02}
\left|\frac{\hat{v}_l^i-\hat{v}_l^{i+\delta e_k}}{\delta}\right|^2 &\leq 3\Bigg(\varphi_{l-2}(i+\delta e_k)^2\left|\frac{v^i-v^{i+\delta e_k}}{\delta}\right|^2+(1-\varphi_{l-2}(i))^2\left|\frac{w^i-w^{i+\delta e_k}}{\delta}\right|^2\nonumber\\
&\hspace{3em}+\left|\frac{\varphi_{l-2}(i)-\varphi_{l-2}(i+\delta e_k)}{\delta}\right|^2|v^i-w^i|^2\Bigg)\nonumber\\
&\leq 3\left(\left|\frac{v^i-v^{i+\delta e_k}}{\delta}\right|^2+\left|\frac{w^i-w^{i+\delta e_k}}{\delta}\right|^2+2M^2\right).
\end{align}
Thus, thanks to \eqref{est:w2} we may also deduce that
\begin{align}\label{fund:est03}
\tilde{E}_\e(u,\hat{v}_l,(U_{l-3}\setminus\overline{U}_{l-4})\cap V) &\leq \tilde{E}_\e(u,v,(U_{l-3}\setminus\overline{U}_{l-4})\cap V)\nonumber\\
&+6\left(\tilde{G}_\e(v,(U_{l-3}\setminus\overline{U}_{l-4})\cap V)+\tilde{G}_\e(\tilde{v},(U_{l-3}\setminus\overline{U}_{l-4}\cap V)\right)\nonumber\\
&+3M^2n\delta^n\e\#((U_{l-3}\setminus\overline{U}_{l-4})\cap V_\delta)).
\end{align}
Moreover, on $U_{l-2}$ we have
\[(\hat{v}_l^i-1)^2\leq 2\left((v^i-1)^2+(w^i-1)^2\right)\]
and $(\hat{v}_l^i)^2\leq (v^i)^2$, which together with \eqref{est:w1},\eqref{est:w2}, and \eqref{fund:est02} give
\begin{align}\label{fund:est04}
\tilde{E}_\e(u,\hat{v}_l,(U_{l-2}\setminus\overline{U}_{l-3})\cap V) &\leq\tilde{F}_\e(u,v,(U_{l-2}\setminus\overline{U}_{l-3})\cap V)\nonumber\\
&+6\left(\tilde{G}_\e(v,(U_{l-2}\setminus\overline{U}_{l-3})\cap V)+\tilde{G}_\e(\tilde{v},(U_{l-2}\setminus\overline{U}_{l-3}\cap V)\right)\nonumber\\
&+3M^2n\delta^n\e\#((U_{l-2}\setminus\overline{U}_{l-3})\cap V_\delta)).
\end{align}
Analogously, we get
\begin{align}\label{fund:est05}
\tilde{E}_\e(\tilde{u},\hat{v}_l,(U_{l+2}\setminus\overline{U}_{l+1})\cap V) &\leq\tilde{F}_\e(\tilde{u},\tilde{v},(U_{l+2}\setminus\overline{U}_{l+1})\cap V)\nonumber\\
&+6\left(\tilde{G}_\e(v,(U_{l+2}\setminus\overline{U}_{l+1})\cap V)+\tilde{G}_\e(\tilde{v},(U_{l+2}\setminus\overline{U}_{l+1}\cap V)\right)\nonumber\\
&+3M^2n\delta^n\e\#((U_{l+2}\setminus\overline{U}_{l+1})\cap V_\delta))
\end{align}
and
\begin{align}\label{fund:est06}
\tilde{E}_\e(\tilde{u},\hat{v}_l,(U_{l+3}\setminus\overline{U}_{l+2})\cap V) &\leq \tilde{E}_\e(\tilde{u},\tilde{v},(U_{l+3}\setminus\overline{U}_{l+2})\cap V)\nonumber\\
&+6\left(\tilde{G}_\e(v,(U_{l+3}\setminus\overline{U}_{l+2})\cap V)+\tilde{G}_\e(\tilde{v},(U_{l+3}\setminus\overline{U}_{l+2}\cap V)\right)\nonumber\\
&+3M^2n\delta^n\e\#((U_{l+3}\setminus\overline{U}_{l+2})\cap V_\delta)).
\end{align}
Then it remains to estimate
\[\tilde{E}_\e(\hat{u}_l,\hat{v}_l,(U_{l+1}\setminus\overline{U}_{l-2})\cap V)=\tilde{F}_\e(\hat{u}_l,w,(U_{l+1}\setminus\overline{U}_{l-2})\cap V))+\tilde{G}_\e(\hat{v}_l,(U_{l+1}\setminus\overline{U}_{l-2})\cap V)).\]
To this end, we observe that by definition of $\hat{v}_l$ we have
\begin{align*}
\tilde{G}_\e(\hat{v}_l,(U_{l+1}\setminus\overline{U}_{l-2})\cap V)) &=G_\e(w,(U_{l+1}\setminus\overline{U}_{l-2})\cap V))\\
&+\frac{1}{2}\hspace*{-2em}\sum_{i\in (U_{l-1}\setminus\overline{U}_{l-2})\cap V_\delta}\hspace*{-2em}\delta^n\e\sum_{\begin{smallmatrix}k=1\\i+\delta e_k\in U_{l-2}\cap V\end{smallmatrix}}^n\hspace*{-1.5em}\left|\frac{\hat{v}_l^i-\hat{v}_l^{i+\delta e_k}}{\delta}\right|^2\\
&+\frac{1}{2}\hspace*{-2em}\sum_{i\in (U_{l+1}\setminus\overline{U}_l)\cap V_\delta}\hspace*{-2em}\delta^n\e\sum_{\begin{smallmatrix}k=1\\i+\delta e_k\in(U_{l+2}\setminus\overline{U}_{l+1})\cap V\end{smallmatrix}}^n\hspace*{-2em}\left|\frac{\hat{v}_l^i-\hat{v}_l^{i+\delta e_k}}{\delta}\right|^2.
\end{align*}
From \eqref{est:w1}-\eqref{est:w2} we deduce that
\[G_\e(w,(U_{l+1}\setminus\overline{U}_{l-2})\cap V))\leq G_\e(v,(U_{l+1}\setminus\overline{U}_{l-2})\cap V))+G_\e(\tilde{v},U_{l+1}\setminus\overline{U}_{l-2})\cap V)),\]
while the same computations as in \eqref{fund:est02} yield
\begin{align*}
\frac{1}{2}\hspace*{-2em} \sum_{i\in (U_{l-1}\setminus\overline{U}_{l-2})\cap V_\delta}\hspace*{-2em} \delta^n\e &\sum_{\begin{smallmatrix}k=1\\i+\delta e_k\in U_{l-2}\cap V\end{smallmatrix}}^n\hspace*{-1.5em}\left|\frac{\hat{v}_l^i-\hat{v}_l^{i+\delta e_k}}{\delta}\right|^2\\
&\leq\frac{3}{2}\hspace*{-2em}\sum_{i\in (U_{l-1}\setminus\overline{U}_{l-2})\cap V_\delta}\hspace*{-2em}\delta^n\e\sum_{\begin{smallmatrix}k=1\\i+\delta e_k\in U_{l-2}\cap V\end{smallmatrix}}^n\hspace*{-2em}\left(\left|\frac{v^i-v^{i+\delta e_k}}{\delta}\right|^2+\left|\frac{w^i-w^{i+\delta e_k}}{\delta}\right|^2+2M^2\right)
\end{align*}
and
\begin{align*}
\frac{1}{2}\hspace*{-2em}\sum_{i\in (U_{l+1}\setminus\overline{U}_l)\cap V_\delta}\hspace*{-2em}\delta^n\e &\sum_{\begin{smallmatrix}k=1\\i+\delta e_k\in(U_{l+2}\setminus\overline{U}_{l+1})\cap V\end{smallmatrix}}^n\hspace*{-2em}\left|\frac{\hat{v}_l^i-\hat{v}_l^{i+\delta e_k}}{\delta}\right|^2\\
&\leq\frac{3}{2}\hspace*{-2em}\sum_{i\in (U_{l+1}\setminus\overline{U}_l)\cap V_\delta}\hspace*{-2em}\delta^n\e\sum_{\begin{smallmatrix}k=1\\i+\delta e_k\in(U_{l+2}\setminus\overline{U}_{l+1})\cap V\end{smallmatrix}}^n\hspace*{-2em}\left(\left|\frac{\tilde{v}^i-\tilde{v}^{i+\delta e_k}}{\delta}\right|^2+\left|\frac{w^i-w^{i+\delta e_k}}{\delta}\right|^2+2M^2\right).
\end{align*}
Hence, again using \eqref{est:w1} we get
\begin{align}\label{fund:est07}
\tilde{G}_\e(\hat{v}_l,(U_{l+1}\setminus\overline{U}_{l-2})\cap V)) &\leq 6\left(\tilde{G}_\e(v,(U_{l+1}\setminus\overline{U}_{l-2})\cap V))+\tilde{G}_\e(\tilde{v},(U_{l+1}\setminus\overline{U}_{l-2})\cap V))\right)\nonumber\\
&+3M^2n\delta^n\e\left(\#((U_{l-1}\setminus\overline{U}_{l-2})\cap V_\delta)+\#((U_{l+1}\setminus\overline{U}_l)\cap V_\delta)\right).
\end{align}
Finally, arguing as in \eqref{fund:est02} for every $i\in (U_{l+1}\setminus\overline{U}_{l-2})\cap V_\delta$ and for every $k\in \{1,\ldots, n\}$ we get
\begin{align*}
\left|\frac{\hat{u}_l^i-\hat{u}_l^{i\pm\delta e_k}}{\delta}\right|^2\leq 3\left(\left|\frac{u^i-u^{i\pm\delta e_k}}{\delta}\right|^2+\left|\frac{\tilde{u}^i-\tilde{u}^{i\pm\delta e_k}}{\delta}\right|^2+M^2|u^i-\tilde{u}^i|^2\right),
\end{align*}
and thus
\begin{align}\label{fund:est08}
\tilde{F}_\e(\hat{u}_l,w,(U_{l+1}\setminus\overline{U}_{l-2})\cap V) &\leq\frac{3}{2}\hspace*{-2em}\sum_{i\in(U_{l+1}\setminus\overline{U}_{l-2})\cap V_\delta}\hspace*{-2em}\delta^n(w^i)^2\sum_{k=1}^n\left(\left|\frac{u^i-u^{i\pm\delta e_k}}{\delta}\right|^2+\left|\frac{\tilde{u}^i-\tilde{u}^{i\pm\delta e_k}}{\delta}\right|^2\right)\nonumber\\
&\hspace{2em}+3M^2n\hspace*{-2em}\sum_{i\in (U_{l+1}\setminus\overline{U}_{l-2})\cap V_\delta}\hspace*{-2em}\delta^n(w^i)^2|u^i-\tilde{u}^i|^2\nonumber\\
&\leq 3\left(\tilde{F}_\e(u,v,(U_{l+1}\setminus\overline{U}_{l-2})\cap V)+\tilde{F}_\e(\tilde{u},\tilde{v},(U_{l+1}\setminus\overline{U}_{l-2})\cap V)\right)\nonumber\\
&\hspace{2em}+3M^2n\hspace*{-2em}\sum_{i\in (U_{l+1}\setminus\overline{U}_{l-2})\cap V_\delta}\hspace*{-2em}\delta^n|u^i-\tilde{u}^i|^2,
\end{align}
where in the last estimate we have used the fact that $w\leq v$, $w\leq\tilde{v}$ and $w\leq 1$. \\
\indent Gathering \eqref{fund:est:01} and \eqref{fund:est03}-\eqref{fund:est08}, summing up over $l$ we get
\begin{align*}
\sum_{l=4}^{N-2}E_\e(\hat{u}_l,\hat{v}_l,U\cup V)&\leq\left(N-6+42\right)\left(E_\e(u,v,U')+E_\e(\tilde{u},\tilde{v},V)\right)\\
&+21M^2n\delta^n\e\#((U'\setminus U)\cap V_\delta))\\
&+9M^2n\hspace*{-1.5em}\sum_{i\in (U'\setminus U)\cap V_\delta}\hspace*{-1.5em}\delta^n|u^i-\tilde{u}^i|^2.
\end{align*}
Hence we can find an index $\hat{l}\in\{4,\ldots,N-2\}$ such that
\begin{align*}
E_\e(\hat{u}_{\hat{l}},\hat{v}_{\hat{l}},U\cup V) &\leq\frac{1}{N-6}\sum_{l=4}^{N-2}E_\e(\hat{u}_l,\hat{v}_l,U\cup V)\\
&\leq\left(1+\frac{42}{N-6}\right)\left(E_\e(u,v,U')+E_\e(\tilde{u},\tilde{v},V)\right)\\
&\hspace{2em}+\frac{21M^2n}{N-6}\delta^n\e\#((U'\setminus U)\cap V_\delta))\\
&\hspace{2em}+\frac{9M^2n}{N-6}\sum_{i\in (U'\setminus U)\cap V_\delta}\hspace*{-1.5em}\delta^n|u^i-\tilde{u}^i|^2.
\end{align*}
We now choose $N$ sufficiently large such that $\frac{42}{N-6}\leq\eta$. Then the pair $(\hat{u},\hat{v}):=(\hat{u}_{\hat{l}},\hat{v}_{\hat{l}})$ satisfies
\[E_\varepsilon(\hat{u},\hat{v},U\cup V)\leq(1+\eta)\left(E_\varepsilon(u,v,U')+E_\varepsilon(\tilde{u},\tilde{v},V)\right)+\sigma_\varepsilon(u,v,\tilde{u},\tilde{v},U,U',V),\]
where
\begin{align*}
\sigma_\e(u,v,\tilde{u},\tilde{v},U,U',V):=\frac{21M^2n}{N-6}\delta^n\e\#((U'\setminus U)\cap V_\delta))+\frac{9M^2n}{2(N-6)}\sum_{i\in (U'\setminus U)\cap V_\delta}\delta^n|u^i-\tilde{u}^i|^2.
\end{align*}
Eventually, note that $\sigma_\e$ satisfies \eqref{cond:sigma}.
\end{proof}
On account of Proposition \ref{fund:est} we are now in a position to prove the following compactness result for the sequence $E_\e$.  

\begin{theorem}[Compactness by $\Gamma$-convergence and properties of the $\Gamma$-limit]\label{gamma:compactness}
Let $E_\varepsilon$ be as in \eqref{def:funct}. Then, for every sequence of positive numbers converging to $0$ there exist a subsequence $(\varepsilon_j)$ and a functional $E_\ell:L^1(\Omega)\times L^1(\Omega)\times\mathscr A(\Omega)\to [0,+\infty]$ such that for all $U\in\mathscr A(\Omega)$
\begin{equation}\label{eq:crit:01}
E_\ell(\cdot,1,U)=(E'_\ell)_{-}(\cdot,1,U)=(E''_\ell)_{-}(\cdot,1,U)\quad\text{on}\ GSBV^2(\Omega).
\end{equation}
Moreover, $E_\ell$ satisfies the following properties:
\begin{enumerate}
\item For every $U\in\mathscr A(\Omega)$, the functional $E_\ell(\cdot,1,U)$ is local and lower semicontinuous with respect to the strong $L^1(\Omega)$-topology;
\item for every $u\in GSBV^2(\Omega)$ and every $U\in\mathscr A(\Omega)$ there holds
\[\int_U|\nabla u|^2dx+\frac{1}{\sqrt{n}}\HH^{n-1}(S_u\cap U)\leq E_\ell(u,1,U)\leq\int_U|\nabla u|^2dx+(1+\alpha_\ell)\HH^{n-1}(S_u\cap U),\]
where $\alpha_\ell\in(0,+\infty)$ is as in Remark \ref{upperbound:local};
\item for every $u\in GSBV^2(\Omega)$, the set function $E_\ell(u,1,\cdot)$ is the restriction to $\mathscr A(\Omega)$ of a Radon measure;
\item for every $U\in \mathscr A_L(\Omega)$ there holds
\[E_\ell(\cdot,1,U)=E'_\ell(\cdot,1,U)=E''_\ell(\cdot,1,U)\quad\text{on}\ GSBV^2(\Omega).\]
\item $E_\ell$ is invariant under translations $x$ and in $u$.
\end{enumerate}
\end{theorem}
\begin{proof}
The existence of the subsequence $(\e_j)$ and of the functional $E_\ell$ satisfying \eqref{eq:crit:01} directly follows from \cite[Proposition 16.9]{DalMaso93}. Further, Remark \ref{rem:crit01} gives (1). The estimates as in (2) are a consequence of Remarks \ref{lowerbound:local}  and \ref{upperbound:local}, and of the inner regularity of $M\!S(u,1,\cdot)$, where $M\!S$ denotes the Mumford-Shah functional. 
The proof of (3) and (4) is standard and appeals to the fundamental estimate Proposition \ref{fund:est} and to the De Giorgi and Letta measure property criterion (see \emph{e.g.} \cite[Proposition 14.23]{DalMaso93}). Finally, the proof of (5) is a consequence of the fact that $E_\e$ is invariant under translations in $x$ and $u$. 
\end{proof}
For later use we now prove that the abstract $\Gamma$-limit $E_\ell$ provided by Theorem \ref{gamma:compactness} is stable under addition of suitable boundary conditions.
To this end, for every $(t,\nu)\in\R\times S^{n-1}$ we localise $E_\ell$ to the oriented unit cube $Q^\nu$ and set
\[E^t_\ell(u,1,Q^\nu):=
\begin{cases}
E_\ell(u,1,Q^\nu) &\text{if}\ u=u_t^\nu\ \text{near}\ \partial Q^\nu,\ \|u\|_{L^\infty}\leq t,\\
+\infty &\text{otherwise in}\ L^1(Q^\nu),
\end{cases}\]
where $u_t^\nu:=t\chi_{\{\left<x,\nu\right>>0\}}$. Moreover, we define
\[E^t_{\e_j}(u,v,Q^\nu):=
\begin{cases}
E_{\e_j}(u,v,Q^\nu) &\text{if}\ (u,v)=(\hat{u}_{t,\nu,j},\hat{v}_{\nu,j})\ \text{near}\ \partial Q^\nu,\ \|u\|_{L^\infty}\leq t,\\
+\infty &\text{otherwise in}\ L^1(Q^\nu)\times L^1(Q^\nu),
\end{cases}\]
where
\begin{equation}\label{bound:cond01}
\hat{u}_{t,\nu,j}^i:=
\begin{cases}
t &\text{if}\ \left<i,\nu\right>>0,\\
0 &\text{if}\ \left<i,\nu\right>\leq 0,
\end{cases}
\end{equation}
and
\begin{equation}\label{bound:cond02}
\hat{v}_{\nu,j}^i:=
\begin{cases}
0 &\text{if}\quad i\in S_j^\nu,\\
1 &\text{otherwise},
\end{cases}
\end{equation}
for every $i\in Q^\nu\cap\delta_j\Z^n$, with
\begin{equation}\label{def:sj}
S_j^\nu:=\{i\in Q^\nu:\ \exists\ l\in Q^\nu\cap\delta_j\Z^n\ \text{such that}\ |i-l|=\delta_j\ \text{and}\quad {\rm sign}\left<i,\nu\right>\neq{\rm sign}\left<l,\nu\right>\}.
\end{equation}
%

\medskip

The following result holds true. 
\begin{theorem}[$\Gamma$-convergence with boundary data]\label{prop:boundary}
Let $(\e_j)$ and $E_\ell$ be as in the statement of Theorem \ref{gamma:compactness}, and let $(t,\nu)\in\R\times S^{n-1}$; then for every and $u\in L^1(Q^\nu)$ there holds
\[\Gamma\hbox{-}\lim_{j\to+\infty}E^t_{\e_j}(u,1,Q^\nu)=E^t_\ell(u,1,Q^\nu).\]
\end{theorem}
\begin{proof}

The liminf inequality is straightforward. Indeed, let $u\in L^1(Q^\nu)$ and let $(u_j,v_j)\subset L^1(Q^\nu)\times L^1(Q^\nu)$ be such that $(u_j,v_j)\to (u,1)$ in $L^1(Q^\nu)\times L^1(Q^\nu)$ and $\sup_jE^t_{\e_j}(u_j,v_j,Q^\nu)<+\infty$. Then, in particular, $u_j=\hat{u}_{t,\nu,j}$ in a neighbourhood of $\partial Q^\nu$, hence $u=u_t^\nu$ in a neighbourhood of $\partial Q^\nu$. Since moreover $\|u_j\|_{L^\infty}\leq t$ for every $j\in\N$ we also deduce that $\|u\|_{L^\infty}\leq t$.

\medskip

To prove the limsup inequality let $u\in GSBV^2(Q^\nu)\cap L^\infty(Q^\nu)$ be such that $u=u_t^\nu$ in $Q^\nu\setminus\overline{Q_\rho^\nu}$ for some $\rho\in (0,1)$ and $\|u\|_{L^\infty}\leq t$. Let $(u_j,v_j)\subset L^1(Q^\nu)\times L^1(Q^\nu)$ be a sequence converging to $(u,1)$ in $L^1(Q^\nu)\times L^1(Q^\nu)$ and such that
\[\limsup_{j\to+\infty}E_{\e_j}(u_j,v_j,Q^\nu)\leq E(u,1,Q^\nu).\]
Since $E_{\e_j}$ decreases by truncation we can always assume that $\|u_j\|_{L^\infty}\leq\|u\|_{L^\infty}\leq t$ so that, in particular, $u_j\to u$ in $L^2(Q^\nu)$. We now modify the sequence $(u_j,v_j)$ in order to attain the boundary condition. To do so we make use of  the fundamental estimate Proposition \ref{fund:est}. We then fix $\rho'\in (\rho,1)$ and set $U:=Q_\rho^\nu$, $U':=Q_{\rho'}^\nu$ and $V:=Q^\nu\setminus\overline{Q_{\rho}^\nu}$. We notice that clearly $u_j\to u_t^\nu$ in $L^2(V)$ and $v_j\to 1$ in $L^2(U')$.
We then let $\eta>0$ be fixed and arbitrary. Hence, invoking Proposition \ref{fund:est}, with $U$ and $V$ chosen as above, we can find a sequence $(\tilde{u}_j,\tilde{v}_j)\subset L^1(Q^\nu)\times L^1(Q^\nu)$, such that $(\tilde{u}_j,\tilde{v}_j)=(u_j,v_j)$ in $Q_\rho^\nu$, $(\tilde{u}_j,\tilde{v}_j)=(\hat{u}_{t,\nu,j},\hat{v}_{\nu,j})$ in $Q^\nu\setminus\overline{Q_{\rho'}^\nu}$, $(\tilde{u}_j,\tilde{v}_j)\to (u,1)$ in $L^2(Q^\nu)\times L^2(Q^\nu)$, $\|u_j\|_{L^\infty}\leq t$, and also satisfying 
\[\limsup_{j\to+\infty}E_{\e_j}(\tilde{u}_j,\tilde{v}_j,Q^\nu)\leq (1+\eta)\Big(\limsup_{j\to+\infty}E_{\e_j}(u_j,v_j,Q_{\rho'}^\nu)+\limsup_{j\to +\infty}E_{\e_j}(\hat{u}_{t,\nu,j},\hat{v}_{\nu,j},Q^\nu\setminus\overline{Q_\rho^\nu})\Big).\]
Moreover, we notice that $F_{\e_j}(\hat{u}_{t,\nu,j},\hat{v}_{\nu,j},Q^\nu\setminus\overline{Q_\rho^\nu})=0$, while 
\[\limsup_{j\to +\infty}G_{\e_j}(\hat{v}_{\nu,j},Q^\nu\setminus\overline{Q_\rho^\nu})\leq c\,\HH^{n-1}((Q^\nu\setminus\overline{Q_\rho^\nu})\cap\Pi_\nu),\]
for some $c>0$.
Therefore we obtain
\[\limsup_{j\to +\infty}E_{\e_j}(\tilde{u}_j,\tilde{v}_j,Q^\nu)\leq (1+\eta)\left(E(u,1,Q^\nu)+c\,\HH^{n-1}((Q^\nu\setminus\overline{Q_\rho^\nu})\cap \Pi_\nu)\right).\]
Then we conclude by the arbitrariness of $\eta>0$ and $\rho\in (0,1)$.
\end{proof}
We are now ready to prove the following integral-representation result for the $\Gamma$-limit $E_\ell$. 
\begin{theorem}[Integral representation]\label{int:rep}
Let $E_\ell$ be as in Theorem \ref{gamma:compactness}, then there exists a Borel function $\phi_\ell:\R\times S^{n-1}\to [0,+\infty)$ such that
\[E_\ell(u,v)=\begin{cases}
\ds \int_\O|\nabla u|^2dx+\int_{S_u\cap \O}\phi_\ell([u],\nu_u)d\HH^{n-1} & \text{if}\; u\in GSBV^2(\O)\; \text{and}\; v=1\, \text{a.e.} \; \text{in}\; \O,
\cr
+\infty & \text{otherwise in}\; L^1(\O)\times L^1(\O).
\end{cases}
\]
Moreover $\phi_\ell$ is given by the following asymptotic formula
\begin{align}\label{form:phi}
\phi_\ell(t,\nu)=\limsup_{\rho\to 0^+}\frac{1}{\rho^{n-1}}\lim_{j\to+\infty}\inf\{E_{\e_j}(u,v,Q_\rho^\nu) \colon (u,v)=(\hat{u}_{t,\nu,j},\hat{v}_{\nu,j})\; \text{near}\; \partial Q_\rho^\nu\},
\end{align}
where $\hat{u}_{t,\nu,j}$ and $\hat{v}_{\nu,j}$ are as in \eqref{bound:cond01} and \eqref{bound:cond02}, respectively.
\end{theorem}
\begin{proof}

We first observe that Theorem \ref{t:domain-G-limit}-(i) together with Remark \ref{upperbound:local} ensure that the domain of $E_\ell$ coincides with $GSBV^2(\O)\times \{1\}$.
Moreover, in view of Theorem \ref{gamma:compactness} $E_\ell$ satisfies all the assumptions of \cite[Theorem 1]{BFLM} except for the lower-bound estimate which can though be recovered by using a standard perturbation argument. That is, let $\sigma>0$ and for every $u\in SBV^2(\Omega)$ set
\[E^\sigma_\ell(u,1,U):=E_\ell(u,1,U)+\sigma\int_{S_u\cap U}|[u]|d\HH^{n-1}.\]
Then, for every $\sigma >0$, $E^\sigma_\ell$ satisfies all the assumptions of Theorem \cite[Theorem 1]{BFLM} which ensures the existence of two Borel functions $f_\ell^\sigma:\Omega\times\R\times\R^n\to [0,+\infty)$ and $g_\ell^\sigma:\Omega\times\R\times\R\times S^{n-1}\to [0,+\infty)$ such that
\[E^\sigma_\ell(u,1,U)=\int_Uf_\ell^\sigma(x,u,\nabla u)dx+\int_{S_u\cap U}g_\ell^\sigma(x,u^+,u^-,\nu_u)d\HH^{n-1},\]
for every $u\in SBV^2(\Omega)$, $U\in\mathscr A(\Omega)$. Since $E_\ell$ is invariant under translations in $u$ and in $x$, by (2) and (3) in \cite[Theorem 1]{BFLM} we have that $f_\ell^\sigma$ does not depend on $x$ and $u$
and that $g_\ell^\sigma$ is independent of $x$ and depends on $u^+$ and $u^-$ only through their difference $[u]$\ie $g_\ell^\sigma(x,a,b,\nu)=\phi_\ell^\sigma(a-b,\nu)$, for some $\phi_\ell^\sigma \colon\R\times S^{n-1}\to [0,+\infty)$. Moreover, appealing to Theorem \ref{gamma:compactness}-(2) gives $f_\ell^\sigma(\xi)=|\xi|^2$ for every $\sigma >0$ and every $\xi \in \R^n$.
Further, formula (3) in \cite[Theorem 1]{BFLM} implies that $\phi_\ell^\sigma$ is decreasing as $\sigma$ decreases. Thus, setting $\phi_\ell(t,\nu):=\lim_{\sigma\to 0^+}\phi_\ell^\sigma (t,\nu)=\inf_{\sigma>0}\phi_\ell^\sigma(t,\nu)$, letting $\sigma\to 0^+$, by the pointwise convergence of $E_\ell^\sigma$ to $E_\ell$ and the Monotone Convergence Theorem we get
\begin{equation}\label{eq:crit03}
E_\ell(u,1,U)=\int_U|\nabla u|^2dx+\int_{S_u\cap U}\phi_\ell([u],\nu_u)d\HH^{n-1},
\end{equation}
for every $u\in SBV^2(\Omega)$, $U\in\mathscr A_L(\Omega)$. Hence, choosing $U=\O$ gives the desired integral representation on $SBV^2(\O)$.

\medskip

We now show that the integral representation \eqref{eq:crit03} also extends to the case $u\in GSBV^2(\Omega)$. 
To this end, for every $u\in GSBV^2(\Omega)$ set
\[\tilde{E}_\ell(u,1):=\int_\Omega|\nabla u|^2dx+\int_{S_u\cap\Omega}\phi_\ell([u],\nu)d\HH^{n-1}.\]
Since $\tilde{E}_\ell$ coincides with ${E}_\ell$ on $GSBV^2(\Omega)\cap L^\infty(\Omega)$ we can deduce that $\tilde{E}_\ell$ is lower semicontinuous with respect to the strong $L^1(\Omega)$-convergence on $GSBV^2(\Omega)\cap L^\infty(\Omega)$. Hence, in particular, $\tilde{E}_\ell$ is lower semicontinuous on the space of finite partitions. Then, necessarily, $\phi_\ell$ is subadditive in the first variable; moreover, for every $t\in\R$ the $1$-homogeneous extension of $\phi_\ell(t,\cdot)$ to $\R^n$ is convex (see \cite{AB902}). Therefore, in view of  the lower bound $\phi_\ell \geq 1/\sqrt{n}$ we can apply \cite[Theorem 5.22]{AFP} to deduce that $\tilde{E}_\ell$ is lower semicontinuous with respect to the strong $L^1(\Omega)$-convergence on the whole $GSBV^2(\Omega)$. 
Since both $E_\ell$ and $\tilde{E}_\ell$ decrease by truncation, by virtue of their lower semicontinuity we can actually deduce that they are continuous by truncations. Therefore,
if $u^m$ denotes the truncation of $u\in GSBV^2(\Omega)$ at level $m\in \N$ we have
$$
E_\ell(u,1)=\lim_{m\to +\infty} E_\ell(u^m,1)=\lim_{m\to +\infty} \tilde E_\ell(u^m,1)=\tilde E_\ell(u,1),
$$
and thus the desired extension. 

\medskip

We now prove \eqref{form:phi}. We start observing that combining the definition of $\phi_\ell$ with formula (3) of \cite[Theorem 1]{BFLM}, for every $(t,\nu)\in\R\times S^{n-1}$ we have
\begin{multline*}
\phi_\ell(t,\nu)=\inf_{\sigma>0}\limsup_{\rho\to 0^+}\frac{1}{\rho^{n-1}}\inf\bigg\{E_\ell(u,1,Q_\rho^\nu)+\sigma\int_{S_u\cap Q_\rho^\nu}|[u]|d\HH^{n-1}\colon \\ u\in SBV^2(Q_\rho^\nu), u=u_t^\nu\; \text{near}\; \partial Q_\rho^\nu\Big\}.
\end{multline*}
For $(t,\nu)\in\R\times S^{n-1}$ set
\[\tilde{\phi}_\ell(t,\nu):=\limsup_{\rho\to 0^+}\frac{1}{\rho^{n-1}}\inf\left\{E_\ell(u,1,Q_\rho^\nu)\colon u\in SBV^2(Q_\rho^\nu),  u=u_t^\nu\; \text{near}\; \partial Q_\rho^\nu\right\};
\]
we now claim that 
$$
\phi_\ell(t,\nu)=\tilde{\phi}_\ell(t,\nu) \quad \text{for every\; $(t,\nu)\in \R\times S^{n-1}$.}
$$ 
Since we trivially have ${\phi_\ell}\geq\tilde\phi_\ell$, we only need to prove the opposite inequality.
To this end, we notice that up to a truncation argument, in the definition of $\phi_\ell(t,\nu)$ and $\tilde{\phi}_\ell(t,\nu)$ we can restrict to test functions $u$ satisfying $\|u\|_{L^\infty}\leq t$. 

Let $\rho>0$ be fixed; for every $(t,\nu)\in \R\times S^{n-1}$ set
\[\tilde{\phi}_{\ell, \rho}(t,\nu):=\frac{1}{\rho^{n-1}}\inf\left\{E_\ell(u,1,Q_\rho^\nu)\colon u\in SBV^2(Q_\rho^\nu),  u=u_t^\nu\; \text{near}\; \partial Q_\rho^\nu,\ \|u\|_{L^\infty}\leq t\right\},\]
and
\begin{multline*}
\phi_{\ell,\rho}^\sigma(t,\nu):=\frac{1}{\rho^{n-1}}\inf\bigg\{E_\ell(u,1,Q_\rho^\nu)+\sigma\int_{S_u\cap Q_\rho^\nu}|[u]|d\HH^{n-1}\colon\\
u\in SBV^2(Q_\rho^\nu), u=u_t^\nu\; \text{near}\; \partial Q_\rho^\nu,\ \|u\|_{L^\infty}\leq t\Big\}.
\end{multline*}
For $\sigma>0$ and $\rho>0$ fixed let $u_{\rho,\sigma}\in SBV^2(Q_\rho^\nu)$ be such that $u_{\rho,\sigma}=u_t^\nu$ in a neighbourhood of $\partial Q_\rho^\nu$, $\|u_{\rho,\sigma}\|_{L^\infty}\leq t$ and such that
\begin{equation}\label{eq:crit05}
E_\ell(u_{\rho,\sigma},1,Q_\rho^\nu)\leq\rho^{n-1}(\tilde{\phi}_{\ell,\rho}(t,\nu)+\sigma).
\end{equation}
We then have
\begin{align}\nonumber
\phi_{\ell,\rho}^\sigma(t,\nu)&\leq\frac{1}{\rho^{n-1}}\left(E_\ell(u_{\rho,\sigma},1,Q_\rho^\nu)+\sigma\int_{S_{u_{\rho,\sigma}}}|[u_{\rho,\sigma}]|d\HH^{n-1}\right)\\\label{c:thu}
&\leq\frac{1}{\rho^{n-1}}\left(\rho^{n-1}(\tilde{\phi}_{\ell,\rho}(t,\nu)+\sigma)+2\sigma|t|\HH^{n-1}(S_{u_{\rho,\sigma}}\cap Q_\rho^\nu)\right).
\end{align}
By Theorem \ref{gamma:compactness}-(2) we also get
\[\HH^{n-1}(S_{u_{\rho,\sigma}}\cap Q_\rho^\nu)\leq\sqrt{n}\,E_\ell(u_{\rho,\sigma},1,Q_\rho^\nu),\]
hence combining \eqref{eq:crit05} and \eqref{c:thu} gives
\[\phi_{\ell,\rho}^\sigma (t,\nu)\leq(1+2\sigma|t|\sqrt{n})\,(\tilde{\phi}_{\ell,\rho}(t,\nu)+\sigma).\]
Thus, passing first to the limsup in $\rho$ and then to the limit in $\sigma$ we deduce that $\phi_\ell\leq\tilde{\phi}_\ell$ and hence the claim. 

Eventually, to deduce the asymptotic formula \eqref{form:phi} we notice that also in \eqref{form:phi} we can restrict the minimisation to those functions  $u$ also satisfying the bound $\|u\|_{L^\infty}\leq t$. Therefore,  \eqref{form:phi} follows from Theorem \ref{prop:boundary} and the fundamental property of $\Gamma$-convergence once noticed that Theorem \ref{t:domain-G-limit} and the constraint $\|u\|_{L^\infty}\leq t$ ensure the needed equi-coercivity.
 \end{proof}
\begin{rem}\label{rem:coordinate}
{\rm If $n=1$ the proof of the $\Gamma$-convergence result in this critical regime simplifies. In fact, by adapting carefully the arguments in \cite[Lemma 2.1]{AT2} to the discrete setting, one can explicitly compute the $\Gamma$-limit as
\[E_\ell(u,1)=\int_\Omega(u')^2dt+c_\ell\#(S_u),\]
where
\[c_\ell:=\min\left\{\sum_{i\in\N}\left(\ell(v^i-1)^2+\frac{1}{\ell}(v^{i+1}-v^i)^2\right):\ v^0=0,\ \lim_{i\to +\infty} v^i=1\right\}.\]
Now let $n>1$; using a slicing argument it can be also shown that
\[\phi_\ell(t,\pm e_k)=c_\ell,\quad\text{for every}\ t\in\R,\quad\text{for every}\ k\in\{1,\ldots,n\},\]
that is, in the coordinate directions, there is a one-dimensional optimal profile.

Eventually, if $\nu\in S^{n-1}$ is of the form
\[\nu=\frac{1}{\sqrt{n}}\sum_{k=1}^n a_ke_k,\quad\text{with}\ a_k\in\{-1,1\};\]
{\it i.e.}, $\nu$ is a symmetry axis of the underlying lattice, then, again by slicing, it can be shown that
\[\phi_\ell(t,\nu)=\sqrt{n}c_{\ell,n},\quad\text{for every}\ t\in\R,\]
where
\begin{equation}\label{def:celln}
c_{\ell,n}:=\min\left\{\sum_{i\in\N}\left(\frac{\ell}{n}(v^i-1)^2+\frac{1}{\ell}(v^{i+1}-v^i)^2\right):\ v^0=0,\ \lim_{i\to+\infty}v^i=1\right\}.
\end{equation}
A solution $\bar{v}$ to \eqref{def:celln} can be explicitly computed by solving the associated Euler-Lagrange equation, whose solution is
\[\bar{v}^i=1-\left(\frac{\ell}{2n}\left(\sqrt{\ell^2+4n}+\ell\right)-1\right)^{-i},\]
and corresponding energy
\[c_{\ell,n}=\frac{\ell}{n}+\frac{4n+\left(\sqrt{\ell^2+4n}+\ell\right)^2}{\ell\left(\sqrt{\ell^2+4n}+\ell\right)^2+4n\sqrt{\ell^2+4n}+4n\ell}\]
so that
\[c_\ell=c_{\ell,1}=\ell+\frac{4+\left(\sqrt{\ell^2+4}+\ell\right)^2}{\ell\left(\sqrt{\ell^2+4}+\ell\right)^2+4\sqrt{\ell^2+4}+4\ell}.\]}
\end{rem}

\subsection{Characterisation of $\phi_\ell$ for $n=2$}
In this subsection we characterise the surface energy density $\phi_\ell$ when $n=2$. In particular we prove that in this case the function $\phi_\ell$ does not depend on $t$\ie for every $(t,\nu)\in \R\times S^1$ we have $\phi_\ell(t,\nu)=\varphi_\ell(\nu)$ where $\varphi_\ell$ is given by the following formula
\begin{multline}\label{as:form}
\varphi_\ell(\nu):=\lim_{T\to+\infty}\frac{1}{2T}\inf\Bigg\{\ell\hspace*{-.5em}\sum_{i\in TQ^\nu\cap\Z^2}\hspace*{-.5em}(v^i-1)^2+\frac{1}{2\ell}\hspace*{-1em}\sum_{\begin{smallmatrix}i,j\in TQ^\nu\cap\Z^2\\|i-j|=1\end{smallmatrix}}\hspace*{-1em}|v^i-v^j|^2\colon v\in \mathcal A_1(TQ^\nu), \\
\exists\ \text{channel $\mathscr C$ in}\ TQ^\nu\cap\Z^2\colon \text{$v=0$ on $\C$},\ v=1\ \text{otherwise near}\ \partial TQ^\nu\Big\},
\end{multline}
where we refer to Definition \ref{def:channel} below for the precise definition of a channel $\C$. An immediate consequence of \eqref{as:form} is that for $n=2$ the $\Gamma$-limit does not depend on the subsequence $(\e_j)$ so that the whole sequence $(E_\e)$ $\Gamma$-converges to $E_\ell$. 


We note that the definition of $\phi_\ell$ provided by the general asymptotic formula \eqref{form:phi} involves the whole functional $E_{\e_j}$, whereas the minimisation problem in \eqref{as:form} depends only on a scaled version of the functional $G_{\e_j}$, the latter being as in \eqref{def:surface}. Then, to prove that for $n=2$ \eqref{form:phi} and \eqref{as:form} actually coincide the idea is to show that for every $(t,\nu)\in \R\times S^1$ a test pair $(u,v)$ for $\phi_\ell(t,\nu)$ can be suitably modified in a way such that $F_{\e_j}(u,v,Q^\nu_\rho)=0$ (with $F_{\e_j}$ as in \eqref{def:bulk}). Since in \eqref{form:phi} a test function $u$ has to satisfy $u^i=u_t^\nu(i)$ in a neighbourhood of $\partial Q_\rho^\nu$, intuitively one would like to choose $u\equiv u_t^\nu$ in $Q^\nu_\rho$, which would carry no energy contribution far from the line $\Pi_\nu$. On the other hand, close to $\Pi_\nu$ this choice would lead to a diverging $F_{\e_j}$. Therefore to find a pair $(u,v)$ satisfying $F_{\e_j}(u,v,Q_\rho^\nu)=0$ we are going to show that we can modify $v$ in a way such that the discrete level set $\{v=0\}$ contains a ``path'' which is ``close'' to $\Pi_\nu$ and connects the two opposite sides of $Q_\rho^\nu$ which are parallel to $\nu$. Then, along this path, the function $u$ can be discontinuous (jumping from the value $0$ to the value $t$) and at the same time, when coupled with $v$, it can satisfy the desired equality $F_{\e_j}(u,v,Q_\rho^\nu)=0$. 

Note that this last argument can be generalised to any space dimension, so that an analog of formula \eqref{as:form} always 
provides an upper bound (see Step 3 in the proof of Theorem \ref{characterization:phi}), upon suitably defining an $n$-dimensional analog of a channel, which will be a ``discrete hypersurface'' disconnecting $TQ^\nu$. We do not make this remark precise further for the sake of brevity.

In what follows we make precise the heuristic idea as above. To this end we need to introduce some further notation.

\begin{definition}
A sequence $p:=\{i_1,\ldots,i_{l_p}\}\subset\delta\Z^2$ is called a {\em path of cardinality} $l_p$ if
\[|i_{k+1}-i_k|=\delta\quad\text{for every}\; k\in\{1,\ldots,l_p-1\}.\]
%
We say that two paths $p=\{i_1,\ldots,i_{l_p}\}$ and $p'=\{i_1',\ldots,i'_{l_{p'}}\}$ are {\em disjoint} if either $p \cap p' = \emptyset$ or if
for every $k\in \{1,\dots, l_p-1\}$ and for every $j\in\{1,\ldots, l_{p'}-1\}$ there holds
\[i_k=i_j'\qquad\Longrightarrow\qquad 
\begin{cases}
i_{k+1}\neq i_{j+1}'\cr 
i_{k+1}\neq i_{j-1}' \cr 
i_{k-1}\neq i_{j+1}',
\end{cases}
\]
\textit{i.e.}, if $i_k \in p \cap p'$, then the vertical and the horizontal $\delta$-segments departing from $i_k$ can only ``belong'' to one of the two paths (see Figure \ref{fig:path}). 
\end{definition}

The definition of disjoint paths as above is motivated by the following fact: if $p,p'\subset Q_\rho^\nu$ are two disjoint paths we have 
\[F_{\e}(u,v,Q_\rho^\nu)\geq F_{\e}(u,v,p)+F_{\e}(u,v,p').\]

\begin{definition}
We say that a discrete set $A\subset\delta\Z^2$ is {\em path connected} if for every pair $(i,j)\in A\times A$ there exists a path $p=\{i_1,\ldots,i_{l_p}\}\subset A$ with $i_1=i$ and $i_{l_p}=j$.

We say that $p$ is a {\em strong path} if for every $i_k\in p$ and every $k\in \{1,\ldots, l_p-1\}$ there holds
\begin{align*}
\begin{cases}
i_k+\delta e_1\in p\quad\text{or}\quad i_k-\delta e_1\in p
\cr\cr
i_k+\delta e_2\in p\quad\text{or}\quad i_k-\delta e_2\in p,
\end{cases}
\end{align*}
\textit{i.e.}, for each $i_k\in p$ (except for $i_1$ and $i_{l_p}$) at least one horizontal and one vertical nearest neighbouring points belong to $p$ (See Figure \ref{fig:strongpath}). 
\begin{figure}[h]
\begin{subfigure}[t]{0.49\textwidth}
\centering
\def\svgwidth{0.7\columnwidth}
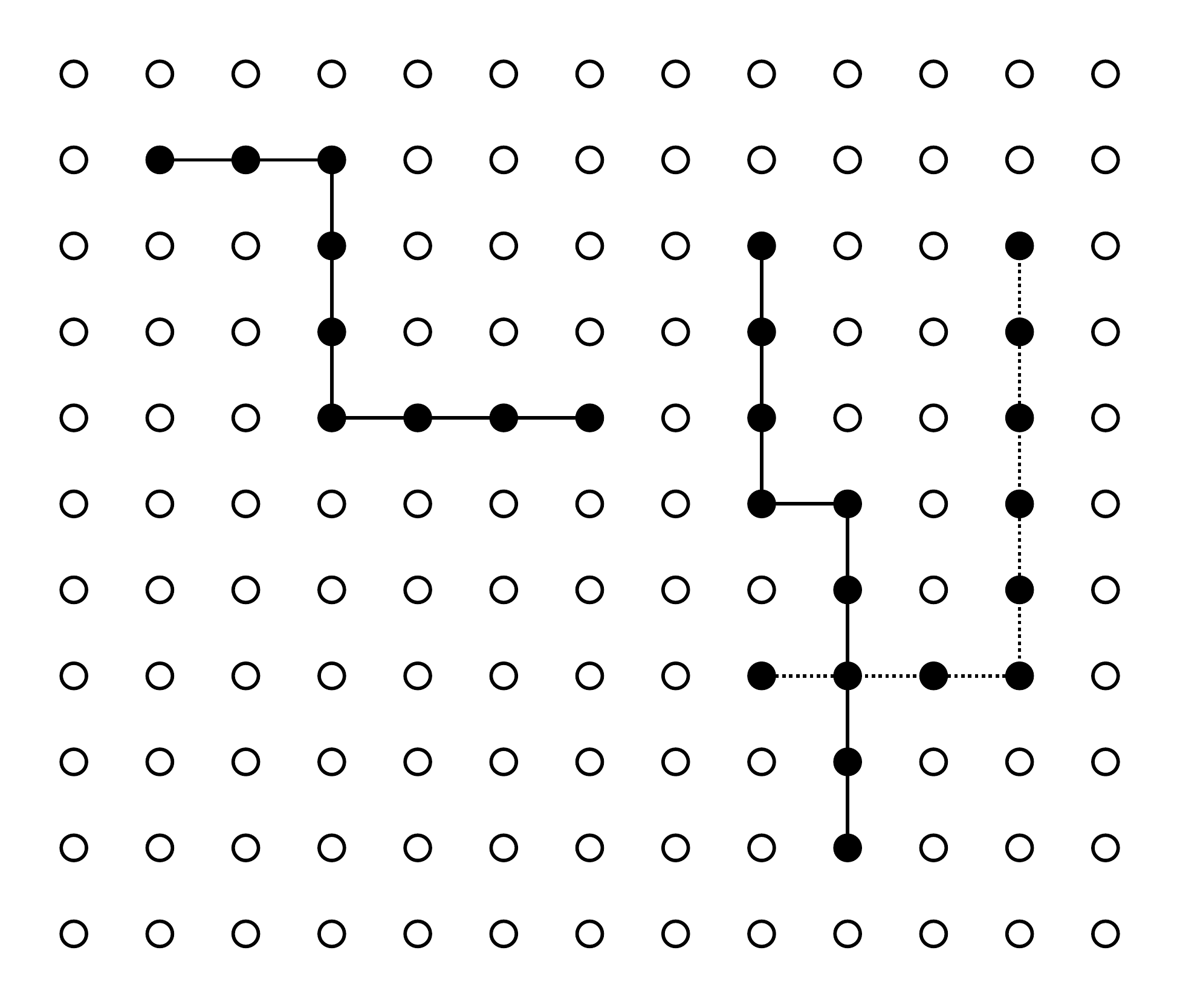
\caption{{\footnotesize Example of a path and of two disjoint paths.}}
\label{fig:path}
\end{subfigure}
\begin{subfigure}[t]{0.49\textwidth}
\centering
\def\svgwidth{0.7\columnwidth}
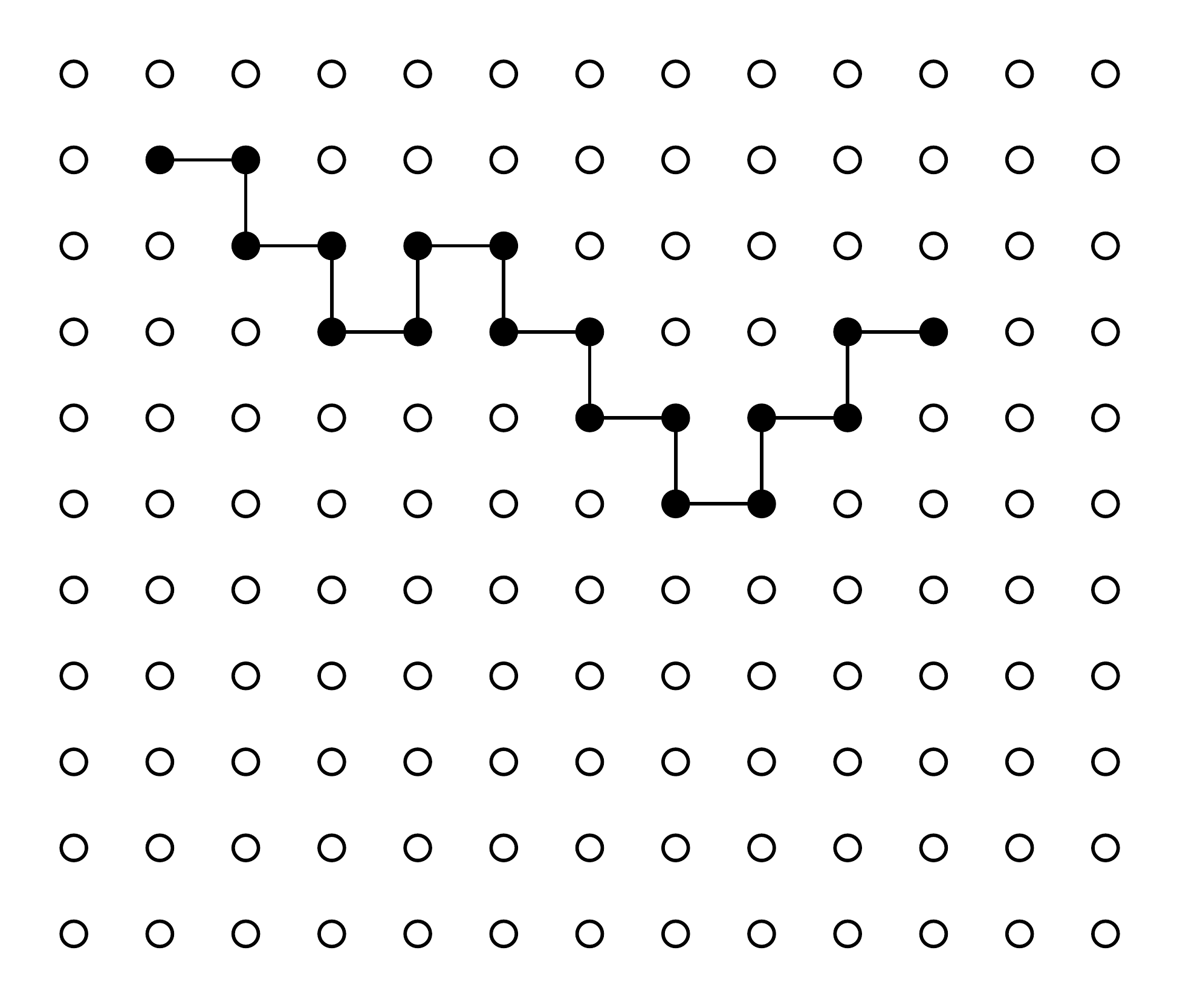
\caption{{\footnotesize Example of a strong path.}}
\label{fig:strongpath}
\end{subfigure}
\caption{{\footnotesize Examples of a path and of a strong path.}}
\end{figure}

In analogy to the definition of path-connected set, we say that a discrete set $A\subset\delta\Z^2$ is {\em strongly connected} if for every pair $(i,j)\in A\times A$ there exists a strong path $p=\{i_1,\ldots,i_{l_p}\}\subset A$ with $i_1=i$ and $i_{l_p}=j$.
\end{definition}

We note that the notion of strong path is motivated by the form of our energy and in particular by the form of the term $F_\e$. Indeed, if $p$ is a strong path connecting the two opposite sides of $Q_\rho^\nu$ which are parallel to $\nu$ and $v\in\A_\e(\Omega)$ is such that $v^i=0$ for all $i\in p$, then it is possible to construct a function $u:Q_\rho^\nu\cap\delta\Z^2\to\{0,t\}$ which is discontinuous across $p$ and satisfies $F_\e(u,v,Q_\rho^\nu)=0$ (see Figure \ref{fig:strongchannel}). 
\begin{figure}[h]
\begin{subfigure}[t]{0.49\textwidth}
\centering
\def\svgwidth{\columnwidth}
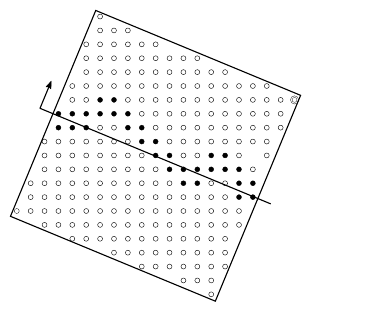
\caption{{\footnotesize Example of a strong path connecting the left and the right side of a cube $Q_\rho^\nu$ and a pair $(u,v)$ with $F(u,v,Q_\rho^\nu)=0$.}}
\label{fig:strongchannel}
\end{subfigure}
\begin{subfigure}[t]{0.49\textwidth}
\centering
\def\svgwidth{1.1\columnwidth}
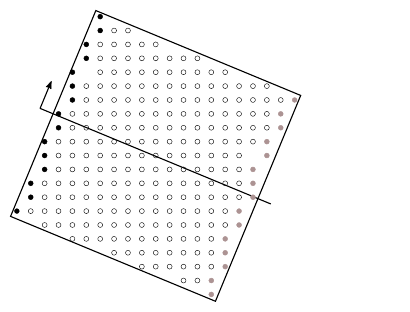
\caption{{\footnotesize $\partial_L^\delta Q_\rho^\nu$, $\partial_R^\delta Q_\rho^\nu$ and $S_\delta^\nu$.}}
\label{fig:discreteboundary}
\end{subfigure}
\caption{{\footnotesize Example of a strong path and discretisation of the ``left'' and ``right'' boundary.}}
\end{figure}
In order to make the above construction precise, we finally introduce the notion of channel. 

\begin{definition}\label{def:channel}
We say that $\C\subset Q_\rho^\nu$ is a {\em channel} in $Q_\rho^\nu\cap\delta\Z^2$ if $\C$ is a strong path connecting $S_\delta^\nu\cap\partial_L^\delta Q_\rho^\nu$ and $S_\delta^\nu\cap\partial_R^\delta Q_\rho^\nu$, where $S_\delta^\nu$ is as in \eqref{def:sj} with $\delta_j$ replaced by $\delta$ and
\begin{align*}
\partial_L^\delta Q_\rho^\nu &:=\left\{i\in Q_\rho^\nu\cap\delta\Z^2:\ \exists\ j\in\delta\Z^2\ \text{such that}\ |i-j|=\delta\ \text{and}\ \langle j,\nu^\perp\rangle\leq -\rho/2\right\},\\
\partial_R^\delta Q_\rho^\nu &:=\left\{i\in Q_\rho^\nu\cap\delta\Z^2:\ \exists\ j\in\delta\Z^2\ \text{such that}\ |i-j|=\delta\ \text{and}\ \langle j,\nu^\perp\rangle\geq \rho/2\right\},
\end{align*}
\textit{i.e.}, $\partial_L^\delta Q_\rho^\nu$ and $\partial_R^\delta Q_\rho^\nu$ are the discretised ``left'' and ``right'' boundary of $Q_\rho^\nu$, respectively (see Figure \ref{fig:discreteboundary}).
We also set
\[\partial^\delta Q_\rho^\nu:=\left\{i\in Q_\rho^\nu\cap\delta\Z^2:\ \exists\ j\in\delta\Z^2\setminus Q_\rho^\nu\ \text{such that}\ |i-j|=\delta\right\}.\]
\end{definition}

We are now ready to state the main result of this section.
\begin{theorem}\label{characterization:phi}
Let $n=2$; let $(t,\nu)\in \R\times S^1$, $\phi_\ell(t,\nu)$ be as in \eqref{form:phi}, and $\varphi_\ell(\nu)$ be as in \eqref{as:form}. Then $\varphi_\ell$ is well-defined, moreover $\phi_\ell(t,\nu)=\varphi_\ell(\nu)$ for every $(t,\nu)\in \R\times S^1$.
\end{theorem}
\begin{proof}
We divide the proof into three main steps. 

\medskip 

\noindent \textbf{Step 1:} $\varphi_\ell$ is well-defined and continuous.

\smallskip

For $T>0$ and $\nu\in S^1$ set
\begin{multline}\label{def:phiT}
\varphi^T_\ell(\nu):=\frac{1}{2T}\inf\Bigg\{\ell\hspace*{-0.5em}\sum_{i\in TQ^\nu\cap\Z^2}\hspace*{-0.5em}(v^i-1)^2+\frac{1}{2\ell}\hspace*{-0.5em}\sum_{\begin{smallmatrix}i,j\in TQ^\nu\cap\Z^2\\|i-j|=1\end{smallmatrix}}\hspace*{-1em}|v^i-v^j|^2: v\in \mathcal A_1(TQ^\nu),
\\ 
\exists\ \text{strong channel} \;
\C\subset TQ^\nu\cap\Z^2\colon \text{$v=0$ on $\C$},\ v=1\ \text{otherwise near}\ \partial TQ^\nu\Big\}.
\end{multline}
Let $\Xi$ denote the set of unit rational vectors\ie $\Xi:=\{\nu\in S^1\colon \exists \lambda \in \R\; \text{s.t.}\; \lambda \nu\in \mathbb Q^2\}$. 

Arguing as in \cite[Proposition 14.4]{BDf} we obtain that $\lim_{T\to +\infty}\varphi^T_\ell(\nu)$ exists for every $\nu\in \Xi$.  

We now claim that for every $\eta \in (0,1)$ there exists $\sigma=\sigma(\eta)>0$ such that for every $\nu,\nu'\in S^1$ satisfying $|\nu-\nu'|\leq\sigma$ there hold
\begin{equation}\label{cont:linf-phi}
\big|\liminf_{T\to+\infty}\varphi^T_\ell(\nu)-\liminf_{T\to+\infty}\varphi^T_\ell(\nu')\big|\leq c\,\eta
\end{equation}
and
\begin{equation}\label{cont:lsup-phi}
\big|\limsup_{T\to+\infty}\varphi^T_\ell(\nu)-\limsup_{T\to+\infty}\varphi^T_\ell(\nu')\big| \leq c\,\eta,
\end{equation}
for some constant $c>0$ independent of $\eta,\sigma,\nu,\nu'$; hence, in particular, both the functions $\nu \mapsto \liminf_{T}\varphi^T_\ell(\nu)$ and $\nu \mapsto\limsup_{T}\varphi^T_\ell(\nu)$ are continuous. 

To prove the claim we start observing that for fixed $\eta \in (0,1)$ we can find $\sigma\in (0,\eta)$ such that for all $\nu,\nu'\in S^1$ satisfying $|\nu-\nu'|\leq\sigma$ there holds
\begin{enumerate}
\item $(1-2\eta)TQ^\nu\subset(1-\eta)TQ^{\nu'}\subset TQ^\nu$;

\smallskip

\item $\HH^{1}\left(\Pi_{\nu'}\cap \left(TQ^\nu\setminus (1-\eta)TQ^{\nu'}\right)\right)\leq cT\eta$;

\smallskip

\item $\HH^{1}\left(\partial TQ^\nu\cap(\Pi^+_\nu\Delta \Pi^+_{\nu'})\right)\leq cT\eta$;
\end{enumerate}
for some $c>0$ independent of $\sigma,\eta,\nu,\nu'$. 

Let now $v_T$ be a test function for $\varphi^{(1-\eta)T}_\ell(\nu')$ with
\begin{equation}\label{c:test-phi}
\ell\hspace*{-2em}\sum_{i\in (1-\eta)TQ^{\nu'}\cap\Z^2}\hspace*{-2em}(v_T^i-1)^2+\frac{1}{2\ell}\hspace*{-2em}\sum_{\begin{smallmatrix}i,j\in (1-\eta)TQ^{\nu'}\cap\Z^2\\|i-j|=1\end{smallmatrix}}\hspace*{-2em}|v_T^i-v_T^j|^2\leq 2(1-\eta)T\varphi^{(1-\eta)T}_\ell(\nu')+1.
\end{equation}
We suitably modify $v_T$ to obtain a test function $\tilde v_T$ for $\varphi^T_\ell(\nu)$. By definition there exists a strong channel in $(1-\eta)TQ^{\nu'}\cap\Z^2$ along which $v_T=0$. Moreover, in view of (1)-(3) we may check that there exist a strong path $p_L \subset \Z^2$ of cardinality $l_L$ connecting
\[S_1^{\nu'}\cap\partial_L^1(1-\eta)TQ^{\nu'}\quad\text{and}\quad S_1^\nu\cap\partial_L^1TQ^\nu\]
and a strong path $p_R \subset \Z^2$ of cardinality $l_R$ connecting
\[S_1^{\nu'}\cap\partial_R^1(1-\eta)TQ^{\nu'}\quad\text{and}\quad S_1^\nu\cap\partial_R^1TQ^\nu\]
such that 
\[l_L,l_R\leq c\left(\HH^{1}\left(\Pi_{\nu'}\cap \left(TQ^\nu\setminus (1-\eta)TQ^{\nu'}\right)\right)+\HH^{1}\left(\partial TQ^\nu\cap(\Pi^+_{\nu}\Delta \Pi^+_{\nu'})\right)\right)\leq cT\eta.\]
Then we define the function $\tilde{v}_T$ as follows
\[\tilde{v}_T^i:=
\begin{cases}
v_T^i &\text{if}\ i\in (1-\eta)TQ^{\nu'},\\
0 &\text{if}\ i\in p_L\cup p_R,\\
1 &\text{otherwise in}\ TQ^\nu\setminus (1-\eta)TQ^{\nu'}.
\end{cases}\]
Clearly $\tilde{v}_T$ is admissible for $\varphi^T_\ell(\nu)$; therefore in view of \eqref{c:test-phi} we have
\begin{align*}
\varphi_\ell^T(\nu) &\leq\frac{1}{2T}\bigg(\ell\hspace*{-0.5em}\sum_{i\in TQ^\nu\cap\Z^2}\hspace*{-0.5em}(\tilde{v}_T^i-1)^2+\frac{1}{2\ell}\hspace*{-1em}\sum_{\begin{smallmatrix}i,j\in TQ^\nu\cap\Z^2\\|i-j|=1\end{smallmatrix}}\hspace*{-1em}|\tilde{v}_T^i-\tilde{v}_T^j|^2\bigg)\\
&=\frac{1}{2T}\bigg(\ell\hspace*{-0.8em}\sum_{i\in (1-\eta)TQ^{\nu'}\cap\Z^2}\hspace*{-2em}(v_T^i-1)^2+\frac{1}{2\ell}\hspace*{-2em}\sum_{\begin{smallmatrix}i,j\in (1-\eta)TQ^{\nu'}\cap\Z^2\\|i-j|=1\end{smallmatrix}}\hspace*{-2em}|v_T^i-v_T^j|^2\bigg)\\
&+\frac{1}{2T}\bigg(\ell\hspace*{-1em}\sum_{i\in (TQ^\nu\setminus (1-\eta)TQ^{\nu'})\cap\Z^2}\hspace*{-2em}(\tilde{v}_T^i-1)^2+\frac{1}{2\ell}\hspace*{-2em}\sum_{\begin{smallmatrix}i,j\in (TQ^\nu\setminus(1-\eta)TQ^{\nu'})\cap\Z^2\\|i-j|=1\end{smallmatrix}}\hspace*{-2em}|\tilde{v}_T^i-\tilde{v}_T^j|^2\bigg)\\
&\leq (1-\eta)\varphi^{(1-\eta)T}_\ell(\nu')+\frac{1}{2T}+\frac{1}{T}c(l_L+l_R)\\
&\leq \varphi^{(1-\eta)T}_\ell(\nu')+\frac{1}{2T}+c\eta.
\end{align*}
Thus, passing to the liminf as $T\to+\infty$ we get
\begin{equation*}\label{cont:phi1}
\liminf_{T\to+\infty}\varphi^T_\ell(\nu)-c\eta\leq\liminf_{T\to+\infty}\varphi^T_\ell(\nu'),
\end{equation*}
hence by exchanging the role of $\nu$ and $\nu'$ we obtain \eqref{cont:linf-phi}. Then, an analogous argument also gives \eqref{cont:lsup-phi}.

By combining \eqref{cont:linf-phi} and \eqref{cont:lsup-phi} we obtain that $\varphi_\ell(\nu)$ is well-defined for every $\nu \in S^1$. Indeed let $\nu \in S^1$, since $\Xi$ is dense in $S^1$, for every $\sigma>0$ we can find $\nu_\sigma\in \Xi$ such that $|\nu-\nu_\sigma|\leq\sigma$. Then, by virtue of \eqref{cont:linf-phi} and \eqref{cont:lsup-phi} and in view of the equality $\lim_T \varphi^T_\ell(\nu_\sigma)=\varphi_\ell(\nu_\sigma)$ we get
$$
-c\eta+\limsup_{T\to +\infty} \varphi^T_\ell(\nu) \leq \varphi_\ell (\nu_\sigma) \leq \liminf_{T\to +\infty} \varphi^T_\ell(\nu)+ c\eta,
$$
hence the existence of $\lim_T \varphi^T_\ell(\nu)$ for every $\nu \in S^1$ follows by the arbitrariness of $\eta$.

We eventually notice that the existence of $\lim_T \varphi^T_\ell(\nu)$ together with \eqref{cont:linf-phi} and \eqref{cont:lsup-phi} yields the continuity of $\nu \mapsto \varphi_\ell(\nu)$. 

\medskip

We now turn to the proof of the equality $\phi_\ell=\varphi_\ell$. This proof will be carried out in two steps. 

To simplify the notation, in what follows we denote by $E_\e$ the $\Gamma$-converging subsequence provided by Theorem \ref{gamma:compactness}. 

Let $(t,\nu)\in \R\times S^1$; for $\e,\rho>0$ set
\[\phi_{\e,\rho}(t,\nu):=\frac{1}{\rho}\inf\{E_{\e}(u,v,Q_\rho^\nu):\ (u,v)=(\hat{u}_{t,\nu,\e},\hat{v}_{\nu,\e})\ \text{in a neighbourhood of}\ \partial Q_\rho^\nu\},\]
where $\hat{u}_{t,\nu,\e}$ and $\hat{v}_{\nu,\e}$ are given respectively by \eqref{bound:cond01} and \eqref{bound:cond02} with $\e_j$ replaced by $\e$.

To simplify the proof we only consider the special case $\delta=\ell\e$, the general case $\delta/\e \to \ell$ being a straightforward consequence of this special one. 

Under the assumption $\delta=\ell\e$ we then have
\[\frac{1}{\rho}E_\e(u,v,Q_\rho^\nu)=\frac{1}{2\rho}\hspace*{-0.5em}\sum_{\begin{smallmatrix}i,j\in Q_\rho^\nu\cap\delta\Z^2\\|i-j|=\delta\end{smallmatrix}}\hspace*{-1em}(v^i)^2|u^i-u^j|^2+\frac{\delta}{2\rho}\bigg(\ell\hspace*{-0.5em}\sum_{i\in Q_\rho^\nu\cap\delta\Z^2}\hspace*{-1em}(v^i-1)^2+\frac{1}{2 \ell}\hspace*{-0.5em}\sum_{\begin{smallmatrix}i,j\in Q_\rho^\nu\cap\delta\Z^2\\|i-j|=\delta\end{smallmatrix}}\hspace*{-1em}|v^i-v^j|^2\bigg).\]

\medskip

\noindent {\bf Step 2:} $\phi_\ell(t,\nu)\geq\varphi_\ell(\nu)$ for every $(t,\nu)\in\R\times S^1$.

\smallskip 

Let $(t,\nu)\in\R\times S^{1}$ be such that $\phi_\ell(t,\nu)<+\infty$, otherwise there is nothing to prove. Since by definition
$$
\phi_\ell(t,\nu):=\limsup_{\rho \to 0^+}\lim_{\e \to 0}\phi_{\e,\rho}(t,\nu),
$$
up to a subsequence we can assume that
\begin{equation}\label{un:bound}
\sup_{\e,\rho>0}\phi_{\e,\rho}(t,\nu)<+\infty.
\end{equation}
Let $N\in\N$ be fixed and let $(u_{\e,\rho},v_{\e,\rho})$ be a test pair for $\phi_{\e,\rho}(t,\nu)$ such that
\begin{equation}\label{est:crit:03}
\frac{1}{\rho}E_\e(u_{\e,\rho},v_{\e,\rho},Q_\rho^\nu)\leq\phi_{\e,\rho}(t,\nu)+\frac{1}{N}.
\end{equation}
We now claim that we can replace $v_{\e,\rho}$ with a function $\tilde{v}_{\e,\rho}$ which is equal to zero along a channel in $Q_\rho^\nu\cap\delta\Z^2$ without essentially increasing the energy. To this end, fix $\eta \in (0,1)$ and set
\[\mathcal I_\e^\eta:=\{i\in Q_\rho^\nu\cap\delta\Z^2:\ v_{\e,\rho}^i<\eta\}.\]
We notice that thanks to the boundary conditions satisfied by $v_{\e,\rho}$ we have $\mathcal I_\e^\eta\neq \emptyset$. 
Moreover, in view of \eqref{un:bound} and \eqref{est:crit:03} we get
\begin{equation}\label{bound:I:eta}
\#(\mathcal I_\e^\eta)\leq\frac{c}{(1-\eta)^2}\frac{\rho}{\delta},
\end{equation}
for some $c>0$ independent of $\e,\rho$, and $\eta$. 

The general strategy to construct $\tilde{v}_{\e,\rho}$ is as follows: For every $i\in \mathcal I_\e^\eta$ we set $\tilde{v}_{\e,\rho}^i:=0$. Thanks to \eqref{bound:I:eta} this modification to $v_{\e,\rho}$ increases the energy only by an error proportional to $\eta$. Then, if the discrete set $\mathcal I_\e^\eta$ is strongly connected we are done. If instead $\mathcal I_\e^\eta$ consists of more than one strongly connected component we show that we can connect the strongly connected components of $\mathcal I_\e^\eta$ to one another in a way such that  if we replace $v_{\e,\rho}$ by zero along the ``channel'' obtained in this way, we essentially do not increase the energy. 

We now illustrate in detail the multi-step strategy leading to the construction of the channel as above. To this end set
\[\mathcal{C}_\e^\eta:=\{C\subset \mathcal I_\e^\eta:\ C\ \text{cluster in}\ \mathcal I_\e^\eta\},\]
where by ``cluster'' in $\mathcal I_\e^\eta$ we mean a maximal strongly connected component in $\mathcal I_\e^\eta$. Thanks to the boundary condition satisfied by $v_{\e,\rho}$ we can find clusters $C',C''\in\mathcal{C}_\e^\eta$ such that $C'\cap\partial_L^\delta Q_\rho^\nu\cap S_\delta^\nu\neq\emptyset$ and $C''\cap\partial_R^\delta Q_\rho^\nu\cap S_\delta^\nu\neq\emptyset$. We now denote by
\[\mathcal J_0:=\{i\in Q_\rho^\nu\cap\delta\Z^2:\ \exists\ j\in C'\ \text{s.t.}\ |i-j|=\delta\}\]
the set of all lattice points in $Q^\nu_\rho$ which have distance $\delta$ from $C'$.
Then, by the maximality of $C'$ there exists a path $p_0=\{i_1,\ldots,i_{l_0}\}\subset C'\cup \mathcal J_0$ connecting $\partial^\delta Q_\rho^\nu\cap\Pi_\nu^+$ and $\partial^\delta Q_\rho^\nu\cap\Pi_\nu^-$ such that
\begin{equation}\label{cond:eta}
\frac{\left(v_{\e,\rho}^{i_k}\right)^2+\big(v_{\e,\rho}^{i_{k+1}}\big)^2}{2}\geq\frac{\eta^2}{2}\quad\text{for every}\; k\in\{1,\ldots, l_0-1\}.
\end{equation}
In fact, since by assumption $\mathcal I_\e^\eta$ consists of more than one strongly connected component, $p_0$ can be chosen in a way such that for every $k\in\{1,\ldots, l_0-1\}$ either $i_k$ or $i_{k+1}$ belongs to $\mathcal J_0\setminus I_\e^\eta$, which already gives \eqref{cond:eta}. Then, due to the boundary conditions satisfied by $u_{\e,\rho}$ there holds $u_{\e,\rho}^{i_1}=t$ and $u_{\e,\rho}^{i_{l_0}}=0$. Thus, using Jensen's inequality from \eqref{cond:eta} we deduce
\begin{align}\label{est:crit:04}
\frac{1}{\rho}F_\e(u_{\e,\rho},v_{\e,\rho},p_0) &=\frac{1}{\rho}\sum_{k=1}^{l_0-1}\frac{\left(v_{\e,\rho}^{i_k}\right)^2+\big(v_{\e,\rho}^{i_{k+1}}\big)^2}{2}\left|u_{\e,\rho}^{i_k}-u_{\e,\rho}^{i_{k+1}}\right|^2\geq\frac{\eta^2}{2}\frac{t^2}{\rho (l_0-1)}.
\end{align}
Let us now define
\[\mathcal{C}_\e^{\eta,0}:=\left\{C\in\mathcal{C}_\e^{\eta}:\ C'\cup \mathcal J_0\cup C\ \text{is strongly connected}\right\}.\]
If $C''\in\mathcal{C}_\e^{\eta,0}$, we can find two points $i_0^1,i_0^2\in \mathcal J_0$ such that $C'\cup\{i_0^1,i_0^2\}\cup C''$ is the desired channel $\C$ in $Q_\rho^\nu\cap\delta\Z^2$. 
If instead this is not the case, we proceed as follows. We set
\[C_1:=C'\cup \mathcal J_0\cup\left\{i\in C:\ C\in \mathcal C_\e^{\eta,0}\right\},\]
and we notice that $C_1$ is strongly connected. Moreover, we set
\[\mathcal J_1:=\{i\in Q_\rho^\nu\cap\delta\Z^2:\ \exists\ j\in C_1\ \text{s.t.}\ |i-j|=\delta\}.\]
Since by assumption $C''\not\in \mathcal C_\e^{\eta,0}$, arguing as above we can find a path $p_1=\{i_1,\ldots i_{l_1}\}\subset C_1\cup \mathcal J_1$ connecting $\partial^\delta Q_\rho^\nu\cap\Pi^+_\nu$ and $\partial^\delta Q_\rho^\nu\cap\Pi^-_\nu$ such that \eqref{cond:eta} and \eqref{est:crit:04} are satisfied with $p_1,l_1$ in place of $p_0,l_0$. Moreover, $p_1$ can be chosen in such a way that it is disjoint from $p_0$. If now $C''$ is such that $C_1\cup \mathcal J_1\cup C''$ is strongly connected, we can find two points $i_0^1,i_0^2\in \mathcal J_0$ and two points $i_1^1,i_1^2\in \mathcal J_1$ such that $C'\cup\bigcup\mathcal{C}_\e^{\eta,0}\cup\{i_0^1,i_0^2,i_1^1,i_1^2\}\cup C''$ contains a channel $\C$ in $Q_\rho^\nu\cap\delta\Z^2$. If this is not the case we can iterate the procedure as above to define $\mathcal{C}_\e^{\eta,m}$, $C_{m+1}$, and $\mathcal J_{m+1}$ where for every $m\geq 1$ we have
\[\mathcal{C}_\e^{\eta,m}:=\left\{C\in\mathcal{C}_\e^\eta:\ C_m\cup \mathcal J_m\cup C\ \text{is strongly connected}\right\},\]
\[C_{m+1}:=C_m\cup \mathcal J_m\cup\bigcup\mathcal{C}_\e^{\eta,m},\]
and
\[\mathcal J_{m+1}:=\{i\in Q_\rho^\nu\cap\delta\Z^2:\ \exists\ j\in C_{m+1}\ \text{s.t.}\ |i-j|=\delta\}.\]
Let now $M\in\N$ be such that $C''\in\mathcal{C}_\e^{\eta,M}$. Then there exist $2M$ points $i_m^1,i_m^2\in \mathcal J_m$ for $0\leq m\leq M-1$ such that $C'\cup\{i_0^1,i_0^2,\ldots,i_{M-1}^1,i_{M-1}^2\}\cup\bigcup\mathcal{C}_\e^{\eta,M-1}\cup C''$ contains a channel $\C$ in $Q_\rho^\nu\cap\delta\Z^2$. Let us now define $\tilde{v}_{\e,\rho}$ by setting
\[\tilde{v}_{\e,\rho}^i:=
\begin{cases}
0 &\text{if}\ i\in \mathcal I_\e^\eta\ \text{or}\ i=i_m^l\ \text{for some}\; m\in \{0,\ldots, M-1\},\, l=1,2,\\
v_{\e,\rho}^i &\text{otherwise in}\ Q_\rho^\nu\cap\delta\Z^2.
\end{cases}\]
Then by definition $\tilde{v}_{\e,\rho}=0$ along a channel $\C$ in $Q_\rho^\nu\cap\delta\Z^2$. Moreover, we have
\begin{equation}\label{est:crit:05}
\frac{1}{\rho}G_\e(u_{\e,\rho},v_{\e,\rho},Q_\rho^\nu)\geq\frac{1}{\rho}G_\e(u_{\e,\rho},\tilde{v}_{\e,\rho},Q_\rho^\nu)-c\frac{\eta}{(1-\eta)^2}-cM\frac{\delta}{\rho},
\end{equation}
for some constant $c=c(\ell)>0$ independent of $\e$, $\rho$ and $\eta$. Indeed, to prove \eqref{est:crit:05} we observe that   \eqref{bound:I:eta} gives
\begin{align*}
\frac{\delta}{2\rho}\sum_{i\in \mathcal I_\e^\eta} &\Bigg(\ell (v_{\e,\rho}^i-1)^2+\frac{1}{\ell }\sum_{k=1}^2|v_{\e,\rho}^i-v_{\e,\rho}^{i+\delta e_k}|^2\Bigg)\\
&\geq\frac{\delta}{2\rho}\sum_{i\in \mathcal I_\e^\eta}\left(\ell (\tilde{v}_{\e,\rho}^i-1)^2+\frac{1}{\ell}\sum_{k=1}^2|\tilde{v}_{\e,\rho}^i-\tilde{v}_{\e,\rho}^{i+\delta e_k}|^2\right)-c\left(\ell +\frac{3}{\ell}\right)\frac{\eta}{(1-\eta)^2}.
\end{align*}
Then, since $0\leq\tilde{v}_{\e,\rho}\leq 1$ we get
\begin{align*}
\frac{1}{\rho}G_\e({v}_{\e,\rho},Q_\rho^\nu)\geq \frac{1}{\rho}G_\e(\tilde{v}_{\e,\rho},Q_\rho^\nu)-c\left(\ell+\frac{3}{\ell}\right)\frac{\eta}{(1-\eta)^2}-\left(\ell+\frac{2}{\ell}\right)M\frac{\delta}{\rho},
\end{align*}
and hence \eqref{est:crit:05}. 

Therefore, to conclude the proof of this step it only remains to estimate $M$ in \eqref{est:crit:05}. To this end, we notice that by construction there exist $M$ pairwise disjoint paths $p_m\subset C_m\cup \mathcal J_m$, $0\leq m\leq M-1$ of length $l_m$ (where we have set $C_0:=C'$) connecting $\partial^\delta Q_\rho^\nu\cap\Pi^+_\nu$ and $\partial^\delta Q_\rho^\nu\cap\Pi^-_\nu$ such that \eqref{cond:eta} and \eqref{est:crit:04} are satisfied with $p_0,l_0$ replaced by $p_m,l_m$. Further, in view of \eqref{bound:I:eta} there exists $c>0$ such that $l_m\leq c \rho/\delta$, for every $0\leq m\leq M-1$. Then, since the paths $p_m$ are pairwise disjoint, summing up \eqref{est:crit:04} over $m$ yields
\[\frac{1}{\rho}F_\e(u_{\e,\rho},v_{\e,\rho},Q_\rho^\nu)\geq\frac{1}{\rho}\sum_{m=0}^{M-1}F_\e(u_{\e,\rho},v_{\e,\rho},p_m)\geq\frac{\eta^2}{2c}\frac{t^2\delta}{\rho^2}M,\]
which in view of \eqref{un:bound} implies that there exists $c>0$ such that
\begin{equation}\label{c:stima-M}
M\leq \frac{c}{\eta^2t^2}\frac{\rho^2}{\delta}.
\end{equation}
Thus, thanks to \eqref{c:stima-M} estimate \eqref{est:crit:05} becomes
\begin{align*}
\frac{1}{\rho}G_\e(v_{\e,\rho},Q_\rho^\nu)\geq\frac{1}{\rho}G_\e(\tilde{v}_{\e,\rho},Q_\rho^\nu)-c\frac{\eta}{(1-\eta)^2}-\frac{c}{\eta^2t^2}\rho.
\end{align*}
Finally, setting $T(\e):=\frac{\rho}{\delta}$ and $w_{\e,\rho}^i:=\tilde{v}_{\e,\rho}^{\delta i}$ for all $i\in T(\e)Q^\nu\cap\Z^2$ from the above inequality we deduce that
\begin{align}\label{est:crit:06}
\frac{1}{\rho} &E_\e(u_{\e,\rho},v_{\e,\rho},Q_\rho^\nu)\nonumber\\
&\geq\frac{\delta}{2\rho}\Big(\ell\sum_{i\in Q_\rho^\nu\cap\delta\Z^2}(\tilde{v}_{\e,\rho}^i-1)^2+\frac{1}{2\ell}\sum_{\begin{smallmatrix}i,j\in Q_\rho^\nu\cap\delta\Z^2\\|i-j|=\delta\end{smallmatrix}}\left|\tilde{v}_{\e,\rho}^i-\tilde{v}_{\e,\rho}^j\right|^2\Big)-c\frac{\eta}{(1-\eta)^2}-\frac{c}{\eta^2t^2}\rho\nonumber\\
&=\frac{1}{2T(\e)}\Big(\ell\sum_{i\in T(\e)Q^\nu\cap\Z^2}(w_{\e,\rho}^i-1)^2+\frac{1}{2\ell}\sum_{\begin{smallmatrix}i,j\in T(\e)Q^\nu\cap\Z^2\\|i-j|=1\end{smallmatrix}}\left|w_{\e,\rho}^i-w_{\e,\rho}^j\right|^2\Big)-c\frac{\eta}{(1-\eta)^2}-\frac{c}{\eta^2t^2}\rho.
\end{align}
Then, since $w_{\e,\rho}$ is a competitor for $\varphi_\ell(\nu)$, gathering \eqref{est:crit:03} and \eqref{est:crit:06} and passing to the limit as $\e\to 0$, we deduce 
\begin{align*}
\lim_{\e\to 0}\phi_{\e,\rho}(t,\nu) \geq\varphi_\ell(\nu)-c\frac{\eta}{(1-\eta)^2}-\frac{c}{\eta^2t^2}\rho-\frac{1}{N}.
\end{align*}
Eventually, letting $\rho\to 0^+$ we get
\[\phi_\ell(t,\nu)\geq\varphi_\ell(\nu)-c\frac{\eta}{(1-\eta)^2}-\frac{1}{N},\]
hence the desired inequality follows by first letting $\eta\to 0$ and then $N\to+\infty$.

\medskip

\noindent{\bf Step 3:} $\phi_\ell(t,\nu)\leq\varphi_\ell(\nu)$ for every $(t,\nu)\in\R\times S^1$.

\smallskip

Since $\varphi_\ell$ is continuous by Step 1 and for every $t\in\R$ the function $\nu \mapsto \phi_\ell(t,\nu)$ is continuous by lower semicontinuity arguments, it suffices to prove the desired inequality for $\nu\in \Xi$. To simplify the exposition we prove this inequality only in the special case $\nu=e_2$. 

Let $T\in\N$ and let $v_T$ be a test function for $\varphi^T_\ell(e_2)$ such that
\begin{equation}\label{est:crit:06a}
\ell\sum_{i\in TQ\cap\Z^2}(v_T^i-1)^2+\frac{1}{2\ell}\sum_{\begin{smallmatrix}i,j\in TQ\cap\Z^2\\|i-j|=1\end{smallmatrix}}|v_T^i-v_T^j|^2\leq 2T\varphi^T_\ell(e_2)+1,
\end{equation}
where $\varphi^T_\ell(e_2)$ is as in \eqref{def:phiT}. Let now $\C$ be a channel in $TQ\cap \Z^2$ along which $v_T=0$. Then $\C$ divides $TQ$ into two parts defined as
\[TQ^+:=\{i\in TQ\cap\Z^2:\ \exists\ l\in\N\ \text{such that}\ i-le_2\in \C\}\quad\text{and}\quad TQ^-:=(TQ\cap\Z^2)\setminus TQ^+.\]
Now we define a discrete function $u_T$ in $TQ\cap \Z^2$ by setting
\[u_T^i:=
\begin{cases}
t &\text{if}\ i\in TQ^+,\\
0 &\text{if}\ i\in TQ^-.
\end{cases}\]
We extend the pair $(u_T,v_T)$ to $\Z^2$ by periodicity in direction $e_1$ and by setting $(u_T^i,v_T^i):=(u_t^{e_2}(i),1)$ outside the discrete strip $\{x\in \R^2 \colon \left<x,e_2\right><T/2\}\cap \Z^2$. For every $\e,\rho>0$ set
\[I_{\e,\rho}:=\{z\in\Z\times\{0\}:\ (\delta TQ+\delta Tz)\subset Q_\rho\}\]
We now define a pair of competitors $(u_{\e,\rho},v_{\e,\rho})$ for $\phi_{\e,\rho}(t,e_2)$ by setting
\begin{align*}
u_{\e,\rho}^i &:=
\begin{cases}
u_T^{\frac{i}{\delta}} &\text{if}\ i\in\bigcup_{z\in I_{\e,\rho}}(\delta TQ+\delta Tz),\\
u_t^{e_2}(i) &\text{otherwise in}\ Q_\rho\cap\delta\Z^2,
\end{cases}\\
v_{\e,\rho}^i &:=
\begin{cases}
v_T^{\frac{i}{\delta}} &\text{if}\ i\in\bigcup_{z\in I_{\e,\rho}}(\delta TQ+\delta Tz),\\
\hat{v}_{e_2,\e}^i &\text{otherwise in}\ Q_\rho\cap\delta\Z^2,
\end{cases}
\end{align*}
where $\hat{v}_{e_2,\e}^i$ is as in \eqref{bound:cond02} with $\e_j$ replaced by $\e$.
Then the pair $(u_{\e,\rho},v_{\e,\rho})$ is admissible for $\phi_{\e,\rho}$ and by construction we have
\begin{equation}\label{est:crit:07}
F_\e(u_{\e,\rho},v_{\e,\rho},Q_\rho)=0.
\end{equation}
Moreover, since $\#(I_{\e,\rho})\leq{\rho}/{\lfloor\delta T\rfloor}$,
in view of the periodicity of $(u_T,v_T)$ we get
\begin{align}\nonumber
\frac{\delta}{\rho}\sum_{i\in Q_\rho\cap\delta\Z^2}(v_{\e,\rho}^i-1)^2 &=\#(I_{\e,\rho})\frac{\delta}{\rho}\hspace*{-0.5em}\sum_{i\in TQ\cap\Z^2}\hspace*{-0.75em}(v_T^i-1)^2+\frac{\delta}{\rho}\sum_{i\in (Q_\rho\cap\delta\Z^2)\setminus\bigcup_{z\in I_{\e,\rho}}(\delta TQ+\delta Tz)}\hspace*{-3em}(v_{\e,\rho}^i-1)^2\\\nonumber
&\leq\frac{\delta}{\lfloor\delta T\rfloor}\hspace*{-0.25em}\sum_{i\in TQ\cap\Z^2}\hspace*{-0.75em}(v_T^i-1)^2+\frac{\delta}{\rho}\#\bigg(\Big\{i\in \Big(Q_\rho\setminus\bigcup_{z\in I_{\e,\rho}}(\delta TQ+\delta Tz)\Big)\cap S_\delta(e_2)\Big\}\bigg)\\\label{c:uno}
&\leq\frac{\delta}{\lfloor\delta T\rfloor}\hspace*{-0.25em}\sum_{i\in TQ\cap\Z^2}\hspace*{-0.75em}(v_T^i-1)^2+c\frac{\delta}{\rho}T,
\end{align}
and analogously
\begin{align}\label{c:due}
\frac{\delta}{\rho}\sum_{\begin{smallmatrix}i,j\in Q_\rho\cap\delta\Z^2\\|i-j|=\delta\end{smallmatrix}}|v_{\e,\rho}^i-v_{\e,\rho}^j|^2\leq \frac{\delta}{\lfloor\delta T\rfloor}\sum_{\begin{smallmatrix}i,j\in TQ\cap\Z^2\\|i-j|=1\end{smallmatrix}}|v_T^i-v_T^j|^2+c\frac{\delta}{\rho}T.
\end{align}
Hence gathering \eqref{c:uno} and \eqref{c:due} gives
\begin{align*}
\frac{1}{\rho}G_\e(v_{\e,\rho},Q_\rho)\leq\frac{\delta}{2\lfloor\delta T\rfloor}\bigg(\ell\sum_{i\in TQ\cap\Z^2}(v_T^i-1)^2+\frac{1}{2\ell}\sum_{\begin{smallmatrix}i,j\in TQ\cap\Z^2\\|i-j|=1\end{smallmatrix}}|v_T^i-v_T^j|^2\bigg)+c\frac{\delta}{\rho}T.
\end{align*}
Combining the latter estimate with \eqref{est:crit:07} and \eqref{est:crit:06a} implies
\begin{align*}
\phi_{\e,\rho}(t,e_2) &\leq\frac{1}{\rho}E_\e(u_{\e,\rho},v_{\e,\rho},Q_\rho)\\
&\leq\frac{\delta}{2\lfloor\delta T\rfloor}\bigg(\ell\sum_{i\in TQ\cap\Z^2}(v_T^i-1)^2+\frac{1}{2\ell}\sum_{\begin{smallmatrix}i,j\in TQ\cap\Z^2\\|i-j|=1\end{smallmatrix}}|v_T^i-v_T^j|^2\bigg)+c\frac{\delta}{\rho}T\\
&\leq\frac{\delta T}{\lfloor\delta T\rfloor}\left(\varphi^T_\ell(e_2)+\frac{1}{2T}\right)+c\frac{\delta}{\rho}T.
\end{align*}
Therefore letting first $\e$ and then $\rho$ go to zero we get
\[\phi_\ell(t,e_2)\leq\varphi^T_\ell(e_2)+\frac{1}{2T},\]
thus finally the desired inequality follows by letting $T\to+\infty$.
\end{proof}

\section{Proof of the $\G$-convergence result in the supercritical regime $\ell=+\infty$}\label{sec:superc} 
In this section we study the asymptotic behaviour of the functionals $E_\e$ when $\varepsilon$ is much smaller than $\delta$. 

We start recalling that Proposition \ref{compactness}-(ii) gives that in this case the domain of the $\Gamma$-limit is $W^{1,2}(\O)\times \{1\}$ and that for every $u\in W^{1,2}(\O)$ and for every $(u_\e,v_\e) \to (u,1)$ in $L^1(\Omega)\times L^1(\Omega)$ we have
\begin{equation}\label{c:sup-li}
\liminf_{\e \to 0}E_\e(u_\e,v_\e)\geq \int_\O |\nabla u|^2\,dx. 
\end{equation}
On the other hand, the upper-bound inequality is also straightforward. Indeed if $u\in C^\infty(\overline\O)$, then a recovery sequence is simply given by 
\[u_\e^i:=u(i),\quad v_\e^i:=1\quad\text{for every}\quad i\in\Omega_\delta.\]
In fact, $u_\e\to u$ in $L^1(\Omega)$ and $G_\e(v_\e)=0$, while Jensen's inequality together with the mean-value theorem gives
\begin{equation}\label{c:sup-ls}
\limsup_{\e\to 0}E_\e(u_\e,v_\e)=\limsup_{\e\to 0}F_\e(u_\e,v_\e)\leq\int_\Omega|\nabla u|^2dx
\end{equation}
as in the proof of \eqref{eq:n10} in Proposition \ref{prop:limsup}. Then, the general case $u\in W^{1,2}(\O)$ follows by a standard density argument. 

Therefore gathering \eqref{c:sup-li} and \eqref{c:sup-ls} proves that for $\ell=+\infty$ the functionals $E_\e$ $\Gamma$-converge to
\[E_\infty(u,v)=
\begin{cases}
\displaystyle\int_\Omega|\nabla u|^2dx &\text{if}\ u\in W^{1,2}(\Omega),\ v=1\ \text{a.e. in}\ \Omega,\\
+\infty &\text{otherwise in}\ L^1(\Omega)\times L^1(\Omega).
\end{cases}\]
The $\Gamma$-convergence result as above implies, in particular, that if  the pair $(u_\e,v_\e)$ converges in $L^1(\Omega)\times L^1(\Omega)$ to a pair $(u,1)$ for some $u\in GSBV^2(\Omega)\setminus W^{1,2}(\Omega)$ then $E_\e(u_\e,v_\e) \to +\infty$. Then, in the spirit of \cite{BT08}, the purpose of the following subsection is to study the asymptotic behaviour of a suitable scaling of $E_\e$ leading to a limit functional which is finite on the whole $GSBV^2(\O)$. 

\subsection{Formulation of an equivalent energy}
We start noticing that the analysis performed in the scaling regimes $\ell=0$ and $\ell\in (0,+\infty)$ suggests that the development of a discontinuity for $u$ is penalised by a factor proportional to ${\delta}/{\e}$, which is divergent in this supercritical regime.
This observation leads us to consider the following functionals:
\begin{align*}
H_\e(w,v)&:= \frac{1}{2}\sum_{i\in\Omega_\delta}\delta^n(v^i)^2\hspace*{-1em}\sum_{\begin{smallmatrix} k=1\\i\pm\delta e_k\in\Omega_\delta\end{smallmatrix}}^n\hspace*{-0.5em}\left|\frac{w^i-w^{i\pm\delta e_k}}{\delta}\right|^2
\\
&+\frac{1}{2}\bigg(\sum_{i\in \Omega_\delta}\delta^{n-1}(v^i-1)^2
+\sum_{i\in\Omega_\delta}\sum_{\begin{smallmatrix}k=1\\i+\delta e_k\in\Omega_\delta\end{smallmatrix}}^n\hspace*{-1em}\delta^{n-1}\frac{\e^2}{\delta^2}\left|v^i-v^{i+\delta e_k}\right|^2\bigg).
\end{align*}
We notice that $H_\e$ is obtained by $E_\e$ by scaling at the same time the energy and the variable $u$; in fact we have  
$$
H_\e(w,v)=\frac{\e}{\delta}E_\e(\sqrt{\delta/\e}\,w,v).
$$
With the following theorem we establish a $\Gamma$-convergence result for the scaled functionals $H_\e$. 
\begin{theorem}\label{thm:rescaled}
Let $\ell=+\infty$ and let $H_\e \colon L^1(\Omega)\times L^1(\Omega)\longrightarrow [0,+\infty]$ be defined as
\[H_\e(w,v):=
\begin{cases}
\dfrac{\e}{\delta}\,E_\e\Big(\sqrt{\frac{\delta}{\e}}\,w,v\Big) &\text{if}\ w,v\in\A_\e(\Omega),\\
+\infty &\text{otherwise in}\ L^1(\Omega)\times L^1(\Omega).
\end{cases}\]
Then the functionals $H_\e$ $\Gamma$-converge to $H\colon L^1(\Omega)\times L^1(\Omega)\longrightarrow [0,+\infty]$ defined as 
\begin{equation}\label{def:lim:res}
H(w,v):=
\begin{cases}
\displaystyle\int_\Omega|\nabla w|^2dx+\int_{S_w\cap\Omega}|\nu_w|_\infty d\HH^{n-1} &\text{if}\ w\in GSBV^2(\Omega),\ v=1\ \text{a.e. in}\ \Omega,\\
+\infty &\text{otherwise in}\ L^1(\Omega)\times L^1(\Omega).
\end{cases}
\end{equation}
\end{theorem}
%
%
\begin{proof}
The proof is divided into two main steps.

\medskip

\noindent{\bf Step 1}: liminf inequality.

\smallskip 

The liminf inequality is proven in two substeps first considering the case $n=1$ and then the case $n\geq2$. 

\smallskip 

\noindent{\bf Substep 1.1}: the case $n=1$.

\smallskip 

Let $\Omega=I:=(a,b)$ be an open bounded interval. We start showing that in the one-dimensional case the domain of the $\Gamma$-limit is $SBV^2(I)\times \{1\}$. 
To this end let $(w,v)\in L^1(I)\times L^1(I)$ and let $(w_\e,v_\e)\subset L^1(I)\times L^1(I)$ be a sequence such that $(w_\e,v_\e)\to (w,v)$ strongly in $L^1(I)\times L^1(I)$ and satisfying 
\begin{equation}\label{bound:rescaled1}
\sup_\e H_\e(w_\e,v_\e)<+\infty.
\end{equation}
Let $\tilde{v}_\e$ and $\tilde{w}_\e$ be the piecewise affine interpolations of $v_\e$, $w_\e$ on $I_\delta:=I\cap \delta\Z$, respectively. From \eqref{bound:rescaled1} we deduce that $\tilde{v}_\e\to 1$ in $L^2(I)$; hence $v=1$ a.e. in $I$.

We now prove that $w\in SBV^2(I)$. To this end we consider the discrete set
\[J_\e:=\{i\in I_\delta:\ v_\e^i<1/2\}.\]
Again appealing to \eqref{bound:rescaled1} we deduce that 
\[c\geq\sum_{i\in J_\e}(v_\e-1)^2>\frac{1}{4}\#(J_\e).\]
for some $c>0$, uniformly in $\e$. Thus, we deduce that there exists $N\in\N$ such that
\[\#(J_\e)\leq N\quad\text{for every}\; \e>0.\]
Without loss of generality we then write
\[J_\e=\{i_1(\e),\ldots,i_N(\e)\},\quad a<i_1(\e)\leq i_2(\e)\leq\ldots \leq i_N(\e)<b,\]
where $N$ is independent of $\e$. Then for every $1\leq j\leq N$ the sequence $i_j(\e)$ is bounded and thus there exists $t\in [a,b]$ such that (up to subsequences) $i_j(\e)\to t$. Denote by
\[J:=\{t\in[a,b]:\ \exists\ j\in\{1,\ldots,N\}\ {\rm s.t.}\ t=\lim_{\e\to 0}i_j(\e)\}\]
the set of these limit points. Then we may write
\[J=\{t_1,\ldots,t_M\},\quad a\leq t_1<t_2<\ldots<t_M\leq b,\quad M\leq N.\]
We set $d_0:=\min\{t_{l+1}-t_l:\ 1\leq l\leq M-1\}$. Let $\eta\in (0,d_0)$ and $j\in\{1,\ldots,N\}$ be arbitrary. By definition of $J$ there exist $l\in\{1,\ldots,M\}$ and $\e(j)>0$ such that
\[i_j(\e)\in(t_l-\eta/2,t_l+\eta/2)\quad\text{for every}\; \e\in(0,\e(j)).\]
By the arbitrariness of $j\in\{1,\ldots,N\}$, setting $\e_0:=\min_{1\leq j\leq N}\e(j)$ we thus deduce that
\[v_\e^i\geq\frac{1}{2}\quad\forall\ i\in I_\delta\setminus(J+[-\eta/2,\eta/2])\quad\text{for every}\; \e\in(0,\e_0).\]
Hence, for $\e\leq\e_0$ we get
\[\sum_{i\in I_\delta\setminus(J+[-\eta/2,\eta/2])}\hspace*{-2em}\delta(v_\e^i)^2\left|\frac{w_\e^i-w_\e^{i\pm\delta}}{\delta}\right|^2\geq\frac{1}{4}\sum_{i\in I_\delta\setminus(J+[-\eta/2,\eta/2])}\hspace*{-0.5em}\delta\left|\frac{w_\e^i-w_\e^{i\pm\delta}}{\delta}\right|^2\geq\frac{1}{2}\int_{(a+\eta,b-\eta)\setminus(J+[-\eta,\eta])}(\tilde{w}_\e')^2dt,\]
thus in view of \eqref{bound:rescaled1} we deduce that the $L^2((a+\eta,b-\eta)\setminus(J+[-\eta,\eta]))$-norm of $\tilde w'_\e$ is equibounded. Therefore, since $w_\e \to w$ in $L^1(I)$ the Poincar\'e-Wirtinger inequality implies that $\tilde{w}_\e$ is equibounded in $W^{1,2}((a+\eta,b-\eta))\setminus(J+[-\eta,\eta]))$ and $\tilde{w}_\e\rightharpoonup w$ weakly in $W^{1,2}((a+\eta,b-\eta))\setminus(J+[-\eta,\eta])$. Moreover, since $\tilde{v}_\e\to 1$ in $L^2(I)$, we have $\tilde{v}_\e\tilde{w}_\e'\rightharpoonup w'$ weakly in $L^1((a+\eta,b-\eta)\setminus(J+[-\eta,\eta]))$.
Hence, using estimate \eqref{est:03} in the proof of Proposition \ref{liminf:1d} entails
\begin{align}\nonumber
\liminf_{\e\to 0}\frac{1}{2}\sum_{i\in\overset{\circ}{I}_\delta}\delta(v_\e^i)^2\left|\frac{w_\e^i-w_\e^{i\pm\delta}}{\delta}\right|^2 &\geq\liminf_{\e\to 0}\int_{(a+\eta,b-\eta)\setminus(J+[-\eta,\eta])}(\tilde{v}_\e)^2(\tilde{w}_\e')^2dt
\\\label{est:super:01}
&\geq\int_{(a+\eta,b-\eta)\setminus(J+[-\eta,\eta])}(w')^2dt.
\end{align}
Then, by the arbitrariness of $\eta\in (0,d_0)$ we deduce both that $w\in SBV^2(I)$ and $S_w\cap I\subset J$. We therefore set
\[S_w\cap I:=\{t_1,\ldots,t_L\},\quad L\leq M.\]
It only remains to prove that
\begin{equation}\label{est:super:02}
\liminf_{\e\to 0}\frac{1}{2}\left(\sum_{i\in I_\delta}(v_\e^i-1)^2+\sum_{i\in\overset{\circ}{I}_\delta}\frac{\e^2}{\delta^2}|v_\e^i-v_\e^{i+\delta}|^2\right)\geq L.
\end{equation}
To this end, we show that the number of lattice points $i\in I_\delta$ such that $v_\e^i\to 0$ is at least $2L$. Let $1\leq l\leq L$, set $I'_l:=(t_l-\eta/2,t_l+\eta/2)$ and define
\[m_l:=\liminf_{\e\to 0}\inf\left\{\frac{(v_\e^{i+\delta})^2+(v_\e^i)^2}{2}:\ i\in (I'_l+(-\delta,\delta))\cap\delta\Z\right\}.\]
We claim that $m_l=0$. Let us assume by contradiction that $m_l>0$. Then there exists a subsequence $\e_j$ such that for $j$ sufficiently large we have
\[\frac{1}{m_l}\frac{(v_{\e_j}^{i+\delta_j})^2+(v_{\e_j}^i)^2}{2}\geq 1\quad\text{for every}\;
i\in (I'_l+(-\delta_j,\delta_j))\cap\delta_j\Z.\]
Therefore, we get
\begin{align*}
\int_{I_l'}(\tilde{w}_{\e_j}')^2dt &\leq\hspace*{-1em}\sum_{i\in(I_l'+(-\delta_j,\delta_j))\cap\delta_j\Z}\hspace*{-1em}\delta_j\left|\frac{w_{\e_j}^i-w_{\e_j}^{i+\delta_j}}{\delta_j}\right|^2\\
&\leq\frac{1}{m_l}\sum_{i\in(I_l'+(-\delta_j,\delta_j))\cap\delta_j\Z}\hspace*{-2em}\delta_j\frac{(v_{\e_j}^{i+\delta_j})^2+(v_{\e_j}^i)^2}{2}\left|\frac{w_{\e_j}^i-w_{\e_j}^{i+\delta_j}}{\delta_j}\right|^2\leq c
\end{align*}
uniformly in $j$. Thus, since $(\tilde{w}_{\e_j})\subset W^{1,2}(I)$ and $\tilde{w}_{\e_j}\to w$ in $L^1(I)$ as $j\to+\infty$ we would deduce that $\tilde{w}_{\e_j}\rightharpoonup w$ in $W^{1,2}(I_l')$ and hence $w\in W^{1,2}(I_l')$, which contradicts the fact $t_l\in I_l'$. Hence, we may deduce that $m_l=0$. Consequently we can find a sequence of lattice points $i_l(\e)\in (I_l'+(-\delta,\delta))\cap\delta\Z$ such that
\[\frac{(v_\e^{i_l(\e)})^2+(v_\e^{i_l(\e)+\delta})^2}{2}\to 0\quad\text{as}\ \e\to 0,\]
the latter implies
\begin{equation*}
v_\e^{i_l(\e)},v_\e^{i_l(\e)+\delta}\to 0\quad\text{as}\ \e\to 0,
\end{equation*}
which in its turn gives
\[\liminf_{\e\to 0}\hspace*{-0.5em}\sum_{i\in I_\delta\cap(t_l-\eta,t_l+\eta)}\hspace*{-2em}(v_\e^i-1)^2\geq 2.\]
Thus, since $\eta<d_0$ we get
\[\liminf_{\e\to 0}\sum_{i\in I_\delta}(v_\e^i-1)^2\geq\sum_{l=1}^L\liminf_{\e\to 0}\hspace*{-0.5em}\sum_{i\in I_\delta\cap(t_l-\eta,t_l+\eta)}\hspace*{-2em}(v_\e^i-1)^2\geq 2L,\]
from which we deduce \eqref{est:super:02}. Gathering \eqref{est:super:01} and \eqref{est:super:02} finally yields the liminf-inequality by the arbitrariness of $\eta>0$.

\smallskip 

\noindent{\bf Substep 1.2}: the case $n\geq 2$.

\smallskip 

In this case the proof of the liminf inequality directly follows from the previous substep arguing as in the proof of Proposition \ref{compactness}.


\medskip

\noindent {\bf Step 2:} limsup inequality.

\smallskip
It is enough to show that 
\begin{equation}\label{est:res01}
\Gamma\hbox{-}\limsup_{\e\to 0}H_\e(w,1)\leq H(w,1)\quad\text{for every}\; w\in\W(\Omega),
\end{equation}
where $\W(\Omega)$ is the space of functions introduced in the proof of Proposition \ref{prop:limsup}. Indeed, if \eqref{est:res01} holds than \cite[Theorem 3.1, Remark 3.2 and Remark 3.3]{CT}  allow us to apply a standard density and lower-semicontinuity argument to deduce that
$$\Gamma\hbox{-}\limsup_{\e\to 0}H_\e(w,1)\leq H(w,1)\quad\text{for every}\;  w\in SBV^2(\Omega)\cap L^\infty(\Omega).$$
Finally, the case $w\in GSBV^2(\Omega)$ follows by a standard truncation argument. 

Therefore we now turn to the proof of \eqref{est:res01}. Let $w\in\W(\Omega)$. To simplify the argument we only discuss the case $\overline{S_w}=K\cap\Omega$, where $K$ is a closed and convex set contained in $\Pi_\nu$, with $\nu:=(\nu_1,\ldots,\nu_n)\in S^{n-1}$; then the general case follows as in the proof of Proposition \ref{prop:limsup}, Step 2.  For every $x,y\in\R^n$ we denote by $\mathcal S_{(x,y)}$ the open segment joining $x$ and $y$; moreover, for every $h\geq 0$ set $K_h:=\{x\in\Pi_\nu:\ \dist(x,K)\leq h\}$. Without loss of generality we suppose that $|\nu|_\infty=|\nu_n|=\nu_n$. Further, upon considering a shifted lattice $\delta\Z^n+\xi_\e$ for a suitable sequence $(\xi_\e)\subset\R^n$ converging to $0$ as $\e\to0$ we may assume that
\begin{equation*}
\overline{S_w}\cap\delta\Z^n=\emptyset\quad\text{for every}\; \e>0
\end{equation*}
and we define
\[w_\e^i:=w(i)\quad\quad\text{for every}\;\ i\in\Omega_\delta.\]
Finally, we set
\[v_\e^i:=
\begin{cases}
0 &\text{if}\ \mathcal S_{(i-\delta e_n,i+\delta e_n)}\cap K_{\sqrt{n}\delta}\neq\emptyset,\\
1 &\text{otherwise in}\ \Omega_\delta.
\end{cases}\]
We clearly have $(w_\e,v_\e)\to (w,1)$ in $L^1(\Omega)\times L^1(\Omega)$. Moreover, we claim that there holds
\begin{equation}\label{est:res05}
\limsup_{\e\to 0}\frac{1}{2}\bigg(\sum_{i\in\Omega_\delta}\delta^{n-1}(v_\e^i-1)^2+\sum_{i\in\Omega_\delta}\sum_{\begin{smallmatrix}k=1\\i+\delta e_k\in\Omega_\delta\end{smallmatrix}}^n\delta^{n-1}\frac{\e^2}{\delta^2}|v_\e^i-v_\e^{i+\delta e_k}|^2\bigg)\leq\HH^{n-1}(S_w\cap\Omega)|\nu|_\infty.
\end{equation}
Indeed, since $\#\{i\in\Omega_\delta:\ v_\e^i=0\}=\mathcal{O}\left(\frac{1}{\delta^{n-1}}\right)$, we get
\[\sum_{i\in\Omega_\delta}\sum_{\begin{smallmatrix}k=1\\i+\delta e_k\in\Omega_\delta\end{smallmatrix}}^n\delta^{n-1}\frac{\e^2}{\delta^2}|v_\e^i-v_\e^{i+\delta e_k}|^2\leq n\,\#\{i\in\Omega_\delta:\ v_\e^i=0\}\delta^{n-1}\frac{\e^2}{\delta^2}\to 0\quad\text{as}\ \e\to 0,\]
while the fact that
\[\#\{i\in\Omega_\delta:\ \mathcal S_{(i-\delta e_n,i+\delta e_n)}\cap K_{\sqrt{n}\delta}\neq\emptyset\}=2\left\lfloor\frac{\HH^{n-1}(\Omega\cap K_{\sqrt{n}\delta})\left<\nu,e_n\right>}{\delta^{n-1}}\right\rfloor\]
yields
\begin{align*}
\limsup_{\e\to 0}\frac{1}{2}\sum_{i\in\Omega_\delta}\delta^{n-1}(v_\e^i-1)^2 &=\limsup_{\e\to 0}\frac{\delta^{n-1}}{2}\#\{i\in\Omega_\delta:\ \mathcal S_{(i-\delta e_n,i+\delta e_n)}\cap K_{\sqrt{n}\delta}\neq\emptyset\}\\
&=\limsup_{\e\to 0}\delta^{n-1}\left\lfloor\frac{\HH^{n-1}(\Omega\cap K_{\sqrt{n}\delta})|\nu_n|}{\delta^{n-1}}\right\rfloor\nonumber\\
&=\HH^{n-1}(S_w\cap\Omega)|\nu_n|=\HH^{n-1}(S_w\cap\Omega)|\nu|_\infty.
\end{align*}
Then, it remains to show that
\begin{equation}\label{est:res04}
\limsup_{\e\to 0}\frac{1}{2}\sum_{i\in\Omega_\delta}(v_\e^i)^2\sum_{\begin{smallmatrix}k=1\\i\pm\delta e_k\end{smallmatrix}}^n\delta^n\left|\frac{w_\e^i-w_\e^{i\pm\delta e_k}}{\delta}\right|^2\leq\int_\Omega|\nabla w|^2dx.
\end{equation}
To do so, we first notice that for $i\in\Omega_\delta$ and $k\in\{1,\ldots,n\}$ with $\mathcal S_{(i,i+\delta e_k)}\cap\overline{S_w}=\emptyset$ by Jensen's inequality we have
\[\left|\frac{w_\e^i-w_\e^{i+\delta e_k}}{\delta}\right|^2\leq\frac{1}{\delta}\int_0^\delta|\left<\nabla w(i+te_k),e_k\right>|^2dt\quad\forall\ 1\leq k\leq n.\]
Hence, thanks to the mean-value theorem, using Fubini's Theorem we get
\begin{align}\label{est:res03}
\sum_{i\in\Omega_\delta}\delta^n\hspace*{-1.3em}\sum_{\begin{smallmatrix}k=1\\i+\delta e_k\in\Omega_\delta\\\mathcal{S}_{(i,i+\delta e_k)}\cap\overline{S_w}=\emptyset\end{smallmatrix}}^n\hspace*{-2em}\frac{(v_\e^i)^2+(v_\e^{i+\delta e_k})^2}{2}\left|\frac{w_\e^i-w_\e^{i+\delta e_k}}{\delta}\right|^2\leq\int_\Omega|\nabla w(x)|^2dx+\mathcal{O}(\delta).
\end{align}
We now claim that 
\begin{equation}\label{eq:res01}
\sum_{i\in\Omega_\delta}\delta^n\hspace*{-1.3em}\sum_{\begin{smallmatrix}k=1\\i+\delta e_k\in\Omega_\delta\\\mathcal{S}_{(i,i+\delta e_k)}\cap\overline{S_w}\neq\emptyset\end{smallmatrix}}^n\frac{(v_\e^i)^2+(v_\e^{i+\delta e_k})^2}{2}\left|\frac{w_\e^i-w_\e^{i+\delta e_k}}{\delta}\right|^2=0,
\end{equation}
so that combining \eqref{est:res03} and \eqref{eq:res01} entails \eqref{est:res04}.
 
To prove the claim let $i\in\overset{\circ}{\Omega}_\delta$ and $1\leq k\leq n$ be such that $\mathcal S_{(i,i+\delta e_k)}\cap\overline{S}_w\neq\emptyset$. This implies $\langle \nu,e_k\rangle = \nu_k \neq 0$. Without loss of generality we assume that $\nu_k>0$. We now show that 
\begin{equation}\label{eq:res02}
v_\e^i=v_\e^{i+\delta e_k}=0.
\end{equation}
If $k=n$ then \eqref{eq:res02} follows directly from the definition of the sequence $(v_\e)$. Let us now prove \eqref{eq:res02} when $k\in\{1,\ldots,n-1\}$. We have to show that 
\[\mathcal S_{(i-\delta e_n,i+\delta e_n)}\cap K_{\sqrt{n}\delta}\neq\emptyset\quad\text{and}\quad \mathcal S_{(i+\delta e_k-\delta e_n,i+\delta e_k+\delta e_n)}\cap K_{\sqrt{n}\delta}\neq\emptyset.\]
To do so, we choose $t\in(0,\delta)$ such that $i+te_k\in\overline{S_w}$. Then
\[i+t\frac{\sin\alpha_k}{\sin\alpha_n}e_n\in K_{\sqrt{n}\delta},\]
where $\alpha_k$ and $\alpha_n$ denote the angle between $e_k$ and $\Pi_\nu$ and between $e_n$ and $\Pi_\nu$, respectively; we note that
\[\sin\alpha_k=\left<\nu,e_k\right>=\nu_k\quad\text{and}\quad\sin\alpha_n=\left<\nu,e_n\right>=\nu_n.\]
Thus, since $\nu_n=|\nu|_\infty\geq\nu_k>0$
we deduce that $t\frac{\sin\alpha_k}{\sin\alpha_n}\in (0,\delta)$ and hence
\[i+t\frac{\sin\alpha_k}{\sin\alpha_n}e_n\in \mathcal S_{(i,i+\delta e_n)},\]
which yields $\mathcal S_{(i,i+\delta e_n)}\cap K_{\sqrt{n}\delta}\neq\emptyset$ and then $v_\e^i=0$ by definition. Moreover, we have
\[(i+\delta e_k)+(t-\delta)\frac{\sin\alpha_k}{\sin\alpha_n}e_n\in K_{\sqrt{n}\delta}.\]
Since $(t-\delta)\frac{\sin\alpha_k}{\sin\alpha_n}\in(-\delta,0)$, we get
\[(i+\delta)e_k+(t-\delta)\frac{\sin\alpha_k}{\sin\alpha_n}e_n\in \mathcal S_{(i+\delta e_k-\delta e_n,i+\delta e_k)},\]
from which we deduce that $\mathcal S_{(i+\delta e_k-\delta e_n,i+\delta e_k)}\cap K_{\sqrt{n}\delta}\neq\emptyset$ and thus $v_\e^{i+\delta e_k}=0$. The latter in its turn implies \eqref{eq:res01} and eventually \eqref{est:res04}.
%
Thus gathering \eqref{est:res05} and \eqref{est:res04} gives the limsup inequality.
\end{proof}

\begin{rem}[a $\Gamma$-expansion of $E_\e$]
For every fixed $w\in GSBV^2(\Omega)$ set $u:=\sqrt\frac{\delta}{\e}\, w$, then using the fact that $S_u=S_w$ and $\nu_u=\nu_w$ we get 
\[\frac{\delta}{\e} H(w,1)=\int_\Omega|\nabla u|^2dx+\frac{\delta}{\e}\int_{S_u\cap\Omega}|\nu_u|_\infty d\HH^{n-1}\]
which for $\ell=+\infty$ is $\Gamma$-equivalent to $E_\e(u,v)$ in the sense of \cite{BT08}.
\end{rem}
%
%
\section{Interpolation properties of $\varphi_\ell$}\label{sec:interp} 
In this last section we show that when $n=2$ the surface energy density $\varphi_\ell$ satisfies the following interpolation properties. 
\begin{prop}
Let $\ell\in (0,+\infty)$ and for $\nu \in S^1$ let $\varphi_\ell(\nu)$ be as in \eqref{as:form}. Then, we have
\[\lim_{\ell\to +\infty}\varphi_\ell(\nu)=+\infty \quad\text{and}\quad \lim_{\ell\to 0}\varphi_\ell(\nu)=1,\]
for every $\nu \in S^1$. Moreover, there holds
\[\lim_{\ell\to +\infty}\frac{\varphi_\ell(\nu)}{\ell}=|\nu|_\infty,\]
for every $\nu \in S^1$.
\end{prop}
\begin{proof}
Let $\nu\in S^1$ and for every $T>0$ let $\varphi^T_\ell$ be the auxiliary function as in \eqref{def:phiT}. We start showing that $\varphi_\ell(\nu)\to +\infty$ as $\ell\to +\infty$. Indeed, for every $T>0$ and for every function $v_T$ that is admissible for $\varphi^T_\ell(\nu)$ there exists a channel $\C$ in $TQ^\nu\cap\Z^2$ along which $v_T=0$. Then it suffices to notice that $\#(\C)\geq T$ to deduce that for every $T>0$ we have
\[\varphi^T_\ell(\nu)\geq \frac{\ell}{2T}\sum_{i\in TQ^\nu\cap\Z^2}(v_T^i-1)^2\geq\frac{\ell}{2}.\]
Hence passing to the limit as $T\to +\infty$ we get $\varphi_\ell(\nu)\geq \ell/2$ and thus the claim.

We now prove that $\varphi_\ell(\nu)\to 1$ as $\ell\to 0$. To this end we first show that
\begin{equation}\label{int:lowerbound}
\liminf_{\ell\to 0}\varphi_\ell(\nu)\geq 1.
\end{equation}
Let $\nu\in S^1$, $T>0$, and let $v_T$ be an arbitrary test function for $\varphi^T_\ell(\nu)$. Then in particular $v_T=1$ in a neighbourhood of the two opposite sides of $Q_T^\nu$ perpendicular to $\nu$. For our purposes it is convenient to extend $v_T$ to $1$ to the discrete stripe $S_T^\nu\cap \Z^2$, where
$$
S_T^\nu:= \{x\in \R^2 \colon -T/2\leq \langle x,\nu^\perp \rangle \leq T/2 \}.
$$
With a little abuse of notation we still denote by $v_T$ such an extension.  
Moreover, we recall that by definition of $v_T$ there exists a channel $\C$ in $TQ^\nu\cap\Z^2$ such that $v_T=0$ on $\C$. Since by definition $\C$ is a strong path (see Figure \ref{fig:strongpath}), we can find a triangulation $\T$ of $S^\nu_T$ with vertices in $\Z^2$ such that if a triangle $\tau\in\T$ has one vertex in $\C$ then all its vertices belong to $\C$. Denote with $\tilde{v}_T$ the piecewise affine interpolation of $v_T$ on the triangulation $\T$. We have
\begin{align}\label{int:est:01}
\frac{1}{2T} &\Bigg(\ell\sum_{i\in TQ^\nu\cap\Z^2}(v_T^i-1)^2+\frac{1}{2\ell}\sum_{\begin{smallmatrix}i,j\in TQ^\nu\cap\Z^2\\|i-j|=1\end{smallmatrix}}|v_T^i-v_T^j|^2\Bigg)\nonumber\\
=\frac{1}{2T} &\Bigg(\ell\sum_{i\in S_T^\nu\cap\Z^2}(v_T^i-1)^2+\frac{1}{2\ell}\sum_{\begin{smallmatrix}i,j\in S_T^\nu\cap\Z^2\\|i-j|=1\end{smallmatrix}}|v_T^i-v_T^j|^2\Bigg)\nonumber\\ 
&\geq\frac{1}{2T}\int_{S_{T-\sqrt 2}^\nu}\ell\,(\tilde{v}_T-1)^2+\frac{1}{\ell}|\nabla\tilde{v}_T|^2dx\nonumber\\
&=\frac{1}{2T}\int_{\Pi_\nu\cap Q^\nu_{T-\sqrt 2}}\left(\int_{-(T+\sqrt 2)/2}^{(T+\sqrt 2)/2}\ell\,(\tilde{v}_T^{y,\nu}(t)-1)^2+\frac{1}{\ell}((\tilde{v}_T^{y,\nu})'(t))^2dt\right)d\HH^1(y),
\end{align}
where $\tilde{v}_T^{y,\nu}(t):=\tilde{v}_T(y+t\nu)$, for every $t\in (-(T+\sqrt 2)/2,(T+\sqrt 2)/2)$. Thus, by definition of $v_T$ for every $y\in\Pi_\nu\cap Q_{T-\sqrt 2}^\nu$ we have $\tilde{v}_T^{y,\nu}((T+\sqrt 2)/2)=\tilde{v}_T^{y,\nu}(-(T+\sqrt 2)/2)=1$. Moreover, thanks to the choice of our triangulation $\T$, for every $y\in\Pi_\nu\cap Q_{T-\sqrt 2}^\nu$ there exists $t_y\in (-(T+\sqrt 2)/2, (T+\sqrt 2)/2)$ such that $\tilde{v}_T^{y,\nu}(t_y)=0$. Thus, for every $y\in\Pi_\nu\cap Q^\nu_{T-\sqrt 2}$ we get
\begin{align*}
\frac{1}{2} &\int_{-(T+\sqrt 2)/2}^{(T+\sqrt 2)/2}\ell(\tilde{v}_T^{y,\nu}(t)-1)^2+\frac{1}{\ell}((\tilde{v}_T^{y,\nu})'(t))^2dt\\
&\geq\int_{t_y}^{(T+\sqrt 2)/2}(1-\tilde{v}_T^{y,\nu}(t))|(\tilde{v}_T^{y,\nu})'(t)|dt+\int_{-(T+\sqrt 2)/2}^{t_y}(1-\tilde{v}_T^{y,\nu}(t))|(\tilde{v}_T^{y,\nu})'(t)|dt\\
&=2\int_0^1(1-z)dz=1.
\end{align*}
Therefore in view of \eqref{int:est:01} we deduce 
\[\frac{1}{2T}\bigg(\ell\sum_{i\in TQ^\nu\cap\Z^2}(v_T^i-1)^2+\frac{1}{2\ell}\sum_{\begin{smallmatrix}i,j\in TQ^\nu\cap\Z^2\\|i-j|=1\end{smallmatrix}}|v_T^i-v_T^j|^2\bigg)\geq\frac{T-\sqrt{2}}{T},\]
hence, by the arbitrariness of $v_T$ we get $\varphi^T_\ell(\nu)\geq\frac{T-\sqrt{2}}{T}$ for every $T>0$ and every $\ell>0$. Then, letting $T\to +\infty$ gives $\varphi_\ell(\nu)\geq 1$ for all $\ell>0$ and thus \eqref{int:lowerbound}. 

Now it remains to show that
\begin{equation}\label{int:upperbound}
\limsup_{\ell\to 0}\varphi_\ell(\nu)\leq 1.
\end{equation}
To prove the above upper-bound inequality \eqref{int:upperbound} we construct a suitable test function $v_T$ for $\varphi_\ell(\nu)$. 
To this end let $\eta>0$ be fixed; let $T_\eta>0$ and $f\in C^2([0,T_\eta])$ be such that $f(0)=0$, $f(T_\eta)=1$, $f'(T_\eta)=f''(T_\eta)=0$ , and
\[\int_0^{T_\eta}(f-1)^2+(f')^2dt\leq 1+\eta.\]
Clearly, up to setting $f(t)=1$ for every $t\geq T_\eta$ we can always assume that $f\in C^2([0,+\infty))$. 
Let $T>0$ and set $T':=T-\sqrt{2}$. Denoting by $d(x)$ the distance of $x$ from $\Pi_\nu$ we set
\[v_T(x):=
\begin{cases}
0 &\text{if}\ d(x)\leq \sqrt{2},\\
f(\ell(d(x)-\sqrt{2})) &\text{if}\ x\in T'Q^\nu,\ d(x)>\sqrt{2},\\
1 &\text{otherwise},
\end{cases}\]
which is well-defined for $T>\frac{T_\eta}{\ell}+2\sqrt{2}$. We now define the sets
\begin{align*}
A &:=\{i\in TQ^\nu\cap\Z^2:\ d(i)\leq\sqrt{2}\},\\
B &:=\{i\in T'Q^\nu\cap\Z^2:\ \sqrt{2}<d(i)<T_\eta/\ell+2\sqrt{2}\},\\
C &:=\{i\in (TQ^\nu\setminus T'Q^\nu)\cap\Z^2: d(i)>\sqrt{2}\}.
\end{align*}
We notice that the set $A$ contains a channel $\C$ along which $v_T=0$. In particular, $v_T$ is admissible for $\varphi^T_\ell(\nu)$. Thus we obtain
\begin{align}\label{int:est02}
\varphi^T_\ell(\nu) &\leq\frac{1}{2T}\Bigg(\ell\sum_{i\in TQ^\nu\cap\Z^2}(v_T^i-1)^2+\frac{1}{2\ell}\sum_{\begin{smallmatrix}i,j\in TQ^\nu\cap\Z^2\\|i-j|=1\end{smallmatrix}}|v_T^i-v_T^j|^2\Bigg)\nonumber\\
&=\frac{1}{2T}\Bigg(\sum_{i\in A}\ell\, (v_T^i-1)^2+\frac{1}{\ell}\sum_{i\in C}\hspace*{-1em}\sum_{\begin{smallmatrix}k=1\\i+e_k\in TQ^\nu\setminus C\end{smallmatrix}}^2|v_T^i-v_T^{i+e_k}|^2\nonumber\\
&\hspace*{2em}+\sum_{i\in B}\Big(\ell\,(v_T^i-1)^2+\frac{1}{\ell}\sum_{\begin{smallmatrix}k=1\\i+e_k\in B\end{smallmatrix}}^2|v_T^i-v_T^{i+e_k}|^2\Big)+\frac{1}{\ell}\sum_{i\in B}\sum_{\begin{smallmatrix}k=1\\i+e_k\in A\end{smallmatrix}}^2|v_T^i|^2\Bigg).
\end{align}
We estimate the terms on the right hand side of \eqref{int:est02} separately. First notice that $\#(A)\leq cT$, while
$$
\#(\{i\in C:\ i+e_k\in TQ^\nu\setminus C\ \text{for some}\ k=1,2\})\leq c/\ell\,, 
$$ 
for some $c>0$. Thus we get
\begin{align}\label{int:est03}
\frac{1}{2T}\sum_{i\in A}\ell\,(v_T^i-1)^2=\frac{\ell}{2T}\#(A)\leq c \,\ell
\end{align}
and
\begin{align}\label{int:est04}
\frac{1}{2T\ell}\sum_{i\in C}\hspace*{-1em}\sum_{\begin{smallmatrix}k=1\\i+e_k\in TQ^\nu\setminus C\end{smallmatrix}}^2|v_T^i-v_T^{i+e_k}|^2\leq\frac{c}{T\ell^2}.
\end{align}
Moreover, if $i\in B$ is such that $i+e_k\in A$ for some $k\in\{1,2\}$ then $|v^i|^2\leq c\,\ell^2$ for some $c>0$ and $\#(\{i\in B:\ i+e_k\in A\ \text{for some}\ k=1,2\})\leq cT$. Then
\begin{align}\label{int:est05}
\frac{1}{2T\ell}\sum_{i\in B}\sum_{\begin{smallmatrix}k=1\\i+e_k\in A\end{smallmatrix}}^2|v_T^i|^2\leq c\,\ell.
\end{align}
Finally, arguing as in the proof of Proposition \ref{prop:limsup} yields 
\begin{align}
\frac{1}{2T} &\sum_{i\in B}\Big(\ell\,(v_T^i-1)^2+\frac{1}{\ell}\sum_{\begin{smallmatrix}k=1\\i+e_k\in B\end{smallmatrix}}^2|v_T^i-v_T^{i+e_k}|^2\Big)\nonumber\\
&\leq\frac{1}{2T}\sum_{i\in B}\left(\int_{i+[0,1)^2}\ell\,(v_T(x)-1)^2+\frac{1}{\ell}|\nabla v_T(x)|^2dx+c\,\ell^2\right)\nonumber\\
&\leq\frac{1}{T}\int_{\Pi_\nu\cap TQ^\nu}\left(\int_0^{T_\eta}(f-1)^2+(f')^2dt\right)d\HH^1+c\,\ell\leq (1+\eta)+c\,\ell.
\end{align}
Gathering \eqref{int:est02}-\eqref{int:est05} we then obtain
\[\varphi^T_\ell(\nu)\leq 1+\eta+ c\Big(\ell+\frac{1}{T\ell^2}\Big).\]
Passing first to the limit as $T\to+\infty$ and then letting $\ell\to 0$, \eqref{int:upperbound} follows by the arbitrariness of $\eta>0$.

We now show that ${\varphi_\ell(\nu)}/{\ell} \to|\nu|_\infty$ as $\ell \to +\infty$.
To this end let $\nu=(\nu_1,\nu_2)\in S^1$; without loss of generality we may assume that $|\nu|_\infty=|\nu_2|$. 
Let $p_2:\R^2\to\Pi_{e_2}$ be the orthogonal projection onto $\Pi_{e_2}$ and for every $j\in\Pi_{e_2}\cap\Z^2$ let $R_j:=\{k\in\Z:\ j+ke_2\in TQ^\nu\}$. Let $T>0$ and suppose that $v_T$ is a test function for $\varphi^T_\ell(\nu)$. Let $\C$ be a channel in $TQ^\nu\cap\Z^2$ along which $v_T=0$. Since $\C$ is a strong path we  deduce that for every $j\in p_2(\Pi_\nu)\cap\Z^2$ there exist at least two points $k_1,k_2\in R_j$ such that $v_T^{j+k_1e_2}=v_T^{j+k_2e_2}=0$. This yields
\[\frac{1}{2T}\sum_{i\in TQ^\nu\cap\Z^2}(v_T^i-1)^2\geq\frac{1}{2T}\sum_{j\in p_2(\Pi_\nu)\cap\Z^2}\sum_{k\in R_j}(v_T^{j+ke_2}-1)^2\geq\frac{\lfloor T\rfloor \,|\nu_2|}{T}.\]
Letting $T\to +\infty$ we then obtain
\begin{equation}\label{int:est06}
\frac{\varphi^\ell(\nu)}{\ell}\geq |\nu_2|=|\nu|_\infty,
\end{equation}
for every $\ell>0$.

To prove that, up to a small error, the reverse inequality also holds, we construct a suitable test function $v_T$ for $\varphi^T_\ell(\nu)$. 
To this end we set
\[v_T^i:=
\begin{cases}
0 &\text{if}\ \mathcal S_{(i-e_2,i+e_2]}\cap\Pi_\nu\neq\emptyset,\\
1 &\text{otherwise in}\ TQ^\nu\cap\Z^2.
\end{cases}\]
Then, arguing as in the proof of Theorem \ref{thm:rescaled}, Step 2 one can show that the set $\{v_T=0\}$ is a channel in $TQ^\nu\cap\Z^2$. In particular, $v_T$ is admissible for $\varphi^T_\ell(\nu)$. Moreover, a direct computation gives
\begin{align*}
\frac{1}{\ell}\varphi^T_\ell(\nu) &\leq\frac{1}{2T}\Bigg(\sum_{i\in TQ^\nu\cap\Z^2}(v_T^i-1)^2+\frac{1}{2\ell^2}\sum_{\begin{smallmatrix}i,j\in TQ^\nu\cap\Z^2\\|i-j|=1\end{smallmatrix}}|v_T^i-v_T^j|^2\Bigg)\\
&\leq\frac{1}{T}\lfloor T |\nu_2|\rfloor\left(1+\frac{2}{\ell^2}\right).
\end{align*}
Thus, letting $T\to+\infty$, for every $\ell>0$ we get
\[\frac{1}{\ell}\varphi_\ell(\nu)\leq |\nu|_\infty\left(1+\frac{2}{\ell^2}\right),\]
which together with \eqref{int:est06} gives the thesis.
\end{proof}

\bigskip

\noindent{\bf Acknowledgments.}
Andrea Braides acknowledges the MIUR Excellence Department Project awarded to the Department of Mathematics, University of Rome Tor Vergata, CUP E83C18000100006.

\end{document}